\documentclass{amsproc}

\usepackage{amsmath}
\usepackage{enumerate}
\usepackage{amsfonts}
\newtheorem{theorem}{Theorem}
\newtheorem{proposition}{Proposition}[section]
\newtheorem{corollary}[proposition]{Corollary}
\newtheorem{lemma}[proposition]{Lemma}

\newcommand*{\DD}{\mathbb{D}}
\newcommand*{\HT}{H^2(\DD)}
\newcommand*{\HI}{H^\infty(\DD)}
\newcommand*{\XX}{\mathfrak{X}}
\newcommand*{\CC}{\mathbb{C}}

\begin{document}

%NOTE TO SELF: need citation for multiplication of Dirichlet series on p.49.  DEAL WITH subsection 6.4 which needs overhaul.

%This answers a question of Knuth!
%\title[Bonzer]{The running time of the binary Euclidean algorithm is asymptotically normally distributed}
\title[Binary Euclidean algorithm]{A rigorous version of R. P. Brent's model for the binary Euclidean algorithm}
\author{Ian D. Morris}
\begin{abstract}
The binary Euclidean algorithm is a modification of the classical Euclidean algorithm for computation of greatest common divisors which avoids ordinary integer division in favour of division by powers of two only. The expectation of the number of steps taken by the binary Euclidean algorithm when applied to pairs of integers of bounded size was first investigated by R. P. Brent in 1976 via a heuristic model of the algorithm as a random dynamical system. Based on numerical investigations of the expectation of the associated Ruelle transfer operator, Brent obtained a conjectural asymptotic expression for the mean number of steps performed by the algorithm when processing pairs of odd integers whose size is bounded by a large integer. In 1998 B. Vall\'ee modified Brent's model via an induction scheme to rigorously prove an asymptotic formula for the average number of steps performed by the algorithm; however, the relationship of this result with Brent's heuristics remains conjectural. In this article we establish previously conjectural properties of Brent's transfer operator, showing directly that it possesses a spectral gap and preserves a unique continuous density. This density is shown to extend holomorphically to the complex right half-plane and to have a logarithmic singularity at zero. By combining these results with methods from classical analytic number theory we prove the correctness of three conjectured formul{\ae} for the expected number of steps, resolving several open questions promoted by D. E. Knuth in \emph{The Art of Computer Programming}.
%We extend this analysis to prove asymptotic expressions for the mean running cost of the algorithm with respect to a general class of cost measurements in the case of both odd and general integer inputs. 

%Important note: G. Maze has shown that the limit distribution exists and that the Brent operator has a unique eigenfunction in L^1. However this leaves two of the three questions in Knuth's book.


\end{abstract}
\maketitle

\section{Introduction}
The classical Euclidean algorithm for the computation of the greatest common divisor (GCD) of a pair of natural numbers has been described as the oldest nontrivial algorithm which remains in use to the present day \cite[p.335]{Knuth97}. The investigation of the number of division steps required by the Euclidean algorithm dates back at least to the $16^{\mathrm{th}}$ century, when it was observed that pairs of consecutive Fibonacci numbers result in particularly long running times \cite{Shal2}. The mathematically rigorous analysis of the number of division steps began in the mid-$19^{\mathrm{th}}$ century with P.-J.-\'E. Finck's demonstration in \cite{Finck} that the number of division steps required for the algorithm to process a pair of integers is bounded by a constant multiple of the logarithm of the largest of the two integers (see \cite{Shant} for historical details). Asymptotic expressions for the mean number of division steps required to process a pair of natural numbers $(u,v)$ such that $1 \leq u \leq v \leq n$ were obtained in the twentieth century by J. D. Dixon \cite{Dixon} and H. Heilbronn \cite{Heil69} and were subsequently refined by J. W. Porter \cite{Porter}. In 1994 it was shown by D. Hensley \cite{Hens94} that the distribution of the number of division steps about its mean is asymptotically normal in the limit as $n \to \infty$, and this result has been extended and generalised by V. Baladi and B. Vall\'ee \cite{BaVa05,BVcorr}.

The binary Euclidean algorithm, proposed in 1967 by J. Stein \cite{Stein} but possibly used in $1^{\mathrm{st}}$-century China \cite[p.340]{Knuth97}, is a variant of the Euclidean algorithm which is adapted to the requirements of binary arithmetic, and is one of the fundamental algorithms for the computation of greatest common divisors. In sharp contrast to the classical Euclidean algorithm it is one of the least well-understood algorithms for GCD computation \cite[\S3]{Vall06}. Early heuristic investigations by R. P. Brent \cite{Bren76} led to a conjectured asymptotic expression for the mean number of steps performed by the binary Euclidean algorithm which remains unproved: B. Vall\'ee has shown rigorously that the mean number of steps performed by the algorithm grows logarithmically with the size of the input \cite{Vall98}, but the relationship of her result to the heuristic formul{\ae} given in earlier research remains conjectural. The purpose of this article is to directly transform the heuristic investigations of R. P. Brent into a rigorous argument and to prove the validity of the various conjectured asymptotic expressions for the mean number of steps, resolving a number of open questions promoted by D. E. Knuth in \emph{The Art of Computer Programming} (\cite[p.339]{Knuth81} and \cite[p.355]{Knuth97}).\footnote{\emph{The Art of Computer Programming} uses a scale from 0 to 50 to rank the difficulty of exercises, where $0$ denotes triviality and 50 indicates a formidable unsolved research problem.  The problems solved in this article -- exercises 31 and 34 of \cite[\S4.5.2]{Knuth97} -- are rated at difficulties of 46 and 49 respectively. To place these figures in perspective, examples of ``exercises'' rated 50 include the Diophantine equation $a^n+b^n+c^n = d^n$ in integers with $n>4$, the equidistribution of $(\frac{3}{2})^n$ modulo $1$, and the existence of infinitely many Mersenne primes (see respectively pages xi, 180 and 413 of \cite{Knuth97}).}

\section{Overview of previous results}

Let us now describe in detail the binary Euclidean algorithm and the current state of its analysis. The binary Euclidean algorithm begins with the following observation: given an arbitrary pair of natural numbers $(u,v)$ it is sufficient to compute the greatest common divisor of the odd parts of $u$ and $v$ respectively, since if $(u,v)=(2^k a,2^\ell b)$ for odd numbers $a$ and $b$ then $\mathrm{gcd}(u,v)=2^{\min\{k,\ell\}}\mathrm{gcd}(a,b)$. Given a pair of odd natural numbers $(u,v)$ with $u \leq v$, the algorithm operates as follows. If $u$ and $v$ are equal then their common value is returned as the value of the greatest common divisor. Otherwise since $u$ and $v$ are odd their difference $v-u$ is even, and there exists a greatest natural number $k$ such that $v-u$ is divisible by $2^k$. The pair $(u,v)$ is replaced with the new pair of odd natural numbers $(u,2^{-k}(v-u))$, and if the former of these two numbers is greater than the latter then the two are exchanged. This sequence of steps is repeated until a pair of equal numbers is obtained and the GCD is returned. Since the maximum of the two integers is strictly decreased by every iteration it is clear that the algorithm eventually terminates.

The analysis of the mean number of steps required for the algorithm to terminate was first attempted by R. P. Brent \cite{Bren76} using an heuristic argument which we now describe\footnote{The reader is cautioned that where some other authors' analyses use logarithms to base $2$, we will use natural logarithms unless otherwise specified and therefore some constants may superficially vary.}. We first note that the number of steps required to process the pair of odd numbers $(u,v)$ is unaffected if both numbers are divided by their GCD, and by identifying the pair of numbers with the result of that operation we may view the algorithm as acting instead on fractions $\frac{u}{v} \in (0,1]$ with odd numerator and denominator, which we will refer to as \emph{odd fractions}. In this representation each iteration of the algorithm transforms the odd fraction $\frac{u}{v}$ to the odd fraction $T_k(\frac{u}{v})$, where $k$ is the maximum integer such that $2^k $ divides $v-u$ and
\[T_k(x):=\Bigg\{\begin{array}{cl}\frac{2^kx}{1-x}&\text{if }0 < x \leq \frac{1}{1+2^k}\\\frac{1-x}{2^kx}&\text{if }\frac{1}{1+2^k}\leq x \leq 1.\end{array}\]
The exact number of steps required to process the pair of odd natural numbers $(u,v)$ is thus equal to the least integer $n \geq 0$  such that
\[\left(T_{k_n} \circ \cdots \circ T_{k_1}\right)\left(\frac{u}{v}\right)=1\]
where for each $i=1,\ldots,n$ the integer $k_i$ is equal to the number of factors of $2$ which divide the difference between the numerator and the denominator of the odd fraction $\left(T_{k_{i-1}} \circ \cdots \circ T_{k_1}\right)\left(\frac{u}{v}\right)$.

In the set of all odd fractions $\frac{u}{v}\in(0,1]$ such that $v \leq n$, the probability that the integer $k_1$ defined above is equal to a fixed natural number $k$ converges to $2^{-k}$ in the limit as $n \to \infty$. Brent's model for the binary Euclidean algorithm, published in \cite{Bren76}, makes the heuristic assumption that for all sufficiently large $n$, the behaviour of the algorithm when applied to the set of all odd fractions $\frac{u}{v} \in (0,1]$ with denominator bounded by $n$ is well modelled by considering instead the effect of the maps $T_k$ defined above on the uniform probability measure on $(0,1]$, with each map $T_k$ being applied with probability $2^{-k}$ independently at each step. After a single iteration of this random dynamical system the expectation of an absolutely continuous probability measure on $(0,1]$ with density $f \in L^1([0,1])$ is thus given by the absolutely continuous probability measure with density equal to
\begin{align}\label{BrOp}(\mathcal{L}f)(x)&:=\sum_{k=1}^\infty \frac{1}{2^k}\left(\sum_{T_ky=x} \frac{1}{|T'(x)|}f\left(y\right)\right)\\\nonumber
&=\sum_{k=1}^\infty \frac{1}{(1+2^kx)^2}f\left(\frac{1}{1+2^kx}\right)+\frac{1}{(x+2^k)^2}f\left(\frac{x}{x+2^k}\right)\end{align}
almost everywhere (see \cite{Bren76,Knuth97} for further details). Based on computer experiments Brent hypothesised, but was unable to prove, that the constant density $\mathbf{1}$ converges exponentially fast under the application of $\mathcal{L}^n$ to a continuous limit density $\xi \colon (0,1] \to\mathbb{R}$. Under the heuristic approximation that this limit distribution is exactly attained after a bounded number of iterations, the expected decrease in the value of $\log (u+v)$ under one application of the algorithm to the fraction $\frac{u}{v}$ can then be calculated to equal
\[\beta:=\log 2 + \int_0^1 \left(\sum_{k=2}^\infty \left(\frac{1-2^{-k}}{1+(2^k-1)x}\right)-\frac{1}{2(1+x)}\right)\left(\int_0^x\xi(t)dt\right)dx\]
and hence the expected number of iterations required to reduce the odd fraction $\frac{u}{v}$ to $1$, where $1 \leq u \leq v \leq n$, was anticipated in \cite{Bren76} to asymptotically grow as $\frac{1}{\beta}\log n$ in the limit as $n \to \infty$. An alternative calculation sharing the same underlying assumptions but  based on the rate of growth of $\log\sqrt{uv}$ leads instead to the coefficient
\[\tilde{\beta}:=\log 2 - \frac{1}{2}\int_0^1\log(1-x)\xi(x)dx\]
in place of $\beta$, and this version of Brent's argument is presented in \cite{Bren98,Knuth97}. 

In order to convert Brent's heuristic into a rigorous argument it would be natural to begin by investigating the operator $\mathcal{L}$ with the aim of constructing the hypothesised limit density $\xi$. Since $\mathcal{L}$ does not have good spectral properties when acting on $L^1([0,1])$ this might naturally be attempted by studying $\mathcal{L}$ on a smaller space of functions as undertaken in standard texts on transfer operators such as \cite{Bala00,PP90,Ruel78}, but this is complicated by the fact that $\mathcal{L}$ does not preserve the space of continuous functions on $[0,1]$: when $\mathcal{L}$ is applied to the constant function $\mathbf{1}$, for example, one may see that a singularity near $0$ of roughly logarithmic magnitude arises, since for very large $N >0$ the size of the quantity $\sum_{k=1}^\infty \frac{1}{(1+2^k(2^{-N}))^2}$ which arises in the series defining $\left(\mathcal{L}\mathbf{1}\right)\left(2^{-N}\right)$ is of the order of magnitude of $N$. As such the operator $\mathcal{L}$ cannot be analysed by considering its action on spaces of functions which are bounded on $[0,1]$. 

In the 1998 article \cite{Vall98} B. Vall\'ee addressed the problem of making Brent's argument rigorous with the introduction of several innovations. Vall\'ee noted that the singular behaviour of $\mathcal{L}$ close to $0$ can be accommodated by working in a Hardy space of holomorphic functions defined on an open disc $D \subset \CC$ and having square-integrable extension to the boundary circle, where the disc $D$ is chosen such that $(0,1] \subset D$ and $0$ lies on the boundary of $D$. On the other hand, in this environment the fact that the transformations $z \mapsto \frac{z}{z+2^k}$ fix the point $0$ significantly complicates the spectral behaviour of the operator $\mathcal{L}$. Vall\'ee circumvented the problem of studying the spectrum of $\mathcal{L}$ by considering instead the family of operators $\mathfrak{V}_s$ on Hardy space defined by
\begin{equation}\label{valop}(\mathfrak{V}_sf)(z):=\sum_{k=1}^\infty \sum_{\substack{a \text{ odd}\\0<a<2^k}} \frac{1}{\left(a+2^kz\right)^{2s}}f\left(\frac{1}{a+2^kz}\right)\end{equation}
for all $z \in D$, where $s$ is allowed to be any complex number in the region $\Re(s)>\frac{1}{2}$. The operator $\mathfrak{V}_1$ is related to the operator $\mathcal{L}$ by an induction process: a single iteration of $\mathfrak{V}_1$ models the effect of applying the main loop of the binary Euclidean algorithm to the fraction $\frac{u}{v}$ several times until the first point at which the numerator and denominator are exchanged. Since this operator is defined only in terms of transformations $z \mapsto \frac{1}{a+2^kz}$ which lack fixed points in the boundary of the disc, it can be shown that each $\mathfrak{V}_s$ is a compact operator on the Hardy space associated to the disc $D$. The existence of an analytic function $\eta$ taking positive values on $(0,1]$ and fixed by $\mathfrak{V}_1$ can then be demonstrated using classical fixed-point theorems for compact operators. Vall\'ee derived a rigorous result from the spectral analysis of the operator by proving that the number of exchanges $E(u,v)$ taken by the binary algorithm to process the pair $(u,v)$ satisfies the expression
\begin{equation}\label{arghargharh}\left(\sum_{n=1}^\infty n\mathfrak{V}_s^{n-1}\mathbf{1}\right)(1)=\sum_{v\text{ odd}} \sum_{\substack{1 \leq u \leq v\\\mathrm{gcd}(u,v)=1\\u\text{ odd}}} \frac{E(u,v)}{v^{2s}}\end{equation}
when $s \in \CC$ with $\Re(s)>1$. Vall\'ee also derived related functional-analytic formul{\ae}  for the total number of steps $S(u,v)$ and the total number of divisions by two $T(u,v)$ performed by the algorithm, and using Tauberian theory was able to rigorously derive asymptotic expressions for the mean of each of these three quantities taken over all odd pairs $(u,v)$ with $1 \leq u \leq v \leq n$. The following statement summarises Vall\'ee's results:
\begin{theorem}[B. Vall\'ee]\label{vg}
There exists a unique analytic function $\eta \colon (0,1] \to (0,+\infty)$ such that $\mathfrak{V}_1\eta=\eta$ and $\int_0^1\eta(x)dx=1$. If for each $n \geq 1$ we define
\[\Omega_n:=\left\{(u,v) \colon 1 \leq u \leq v \leq n,\text{ }u,v\text{ odd and }\mathrm{gcd}(u,v)=1\right\}\]
\[\tilde{\Omega}_n:=\left\{(u,v) \colon 1 \leq u \leq v \leq n,\text{ and }u,v\text{ odd}\right\},\]
then
\[\lim_{n \to \infty} \frac{1}{\#\Omega_n \log n} \sum_{(u,v) \in \Omega_n} E(u,v) =\frac{2}{\pi^2\eta(1)}, \]
\begin{equation}\label{valconst}\lim_{n \to \infty} \frac{1}{\#\Omega_n  \log n} \sum_{(u,v) \in \Omega_n} S(u,v) =\left(\frac{2}{\pi^2\eta(1)}\right)\left(\sum_{a \text{ odd}} \frac{1}{2^{\lfloor \log_2 a\rfloor}}\int_0^{\frac{1}{a}} \eta(x)dx\right), \end{equation}
\[\lim_{n \to \infty} \frac{1}{\#\Omega_n  \log n} \sum_{(u,v) \in \Omega_n} T(u,v) =\left(\frac{4}{\pi^2\eta(1)}\right)\left(\sum_{a \text{ odd}} \frac{1}{2^{\lfloor \log_2 a\rfloor}}\int_0^{\frac{1}{a}} \eta(x)dx\right), \]
and similarly for $\tilde\Omega_n$ in place of $\Omega_n$.
\end{theorem}
Vall\'ee's theorem thus proves that the mean number of steps in the binary Euclidean algorithm is asymptotically logarithmic, but its relationship to Brent's model is indirect and many questions remain open.
 Prior to the present work no proof has been given that the constant in \eqref{valconst} is genuinely equal to the constants $\frac{1}{\beta}$ and $\frac{1}{\tilde\beta}$ conjectured by Brent and Knuth in \cite{Bren76,Bren98,Knuth97}. The existence of the continuous density $\xi \colon (0,1] \to \mathbb{R}$ preserved by $\mathcal{L}$ and the exponential convergence under $\mathcal{L}$ of the uniform measure to the measure of density $\xi$ also remain unproven. In this article we prove all of these conjectures, showing furthermore that the invariant density $\xi$ is real-analytic and admits an analytic continuation to the complex right half-plane $\Re(z)>0$. We apply these results to give a direct proof that Brent's model correctly describes the asymptotic mean running time of the binary Euclidean algorithm for both odd and general natural number inputs, answering an open problem from \emph{The Art of Computer Programming} which was first listed in 1981 (see \cite[p.339]{Knuth81} and \cite[p.355]{Knuth97}). 

The constants in the heuristic formul{\ae} derived by Brent and Knuth are appreciably more amenable to computation than the rigorous expressions obtained by Vall\'ee. The exponentially increasing number of summations involved in the definition of $\mathfrak{V}_s$ and the necessity of summing over all odd integers in the second and third expressions in Theorem \ref{vg} make approximate computation of Vall\'ee's constants problematic, and to the author's knowledge no computation of these constants based on Vall\'ee's definitions has yet been attempted. On the other hand, in \cite[\S4]{Vall98} Vall\'ee conjectured that if the continuous invariant density $\xi$ exists then the constant in \eqref{valconst} satisfies the simpler expression
\begin{equation}\label{veqn}\left(\frac{2}{\pi^2\eta(1)}\right)\left(\sum_{a \text{ odd}} \frac{1}{2^{\lfloor \log_2 a\rfloor}}\int_0^{\frac{1}{a}} \eta(x)dx\right)=\frac{4}{\pi^2\xi(1)}.\end{equation}
This later quantity is far easier to accurately approximate: Brent (\cite{Bren98}, also reported in \cite[p.350]{Knuth97}) has computed the approximation
\[\xi(1)\simeq 0.3979226811883166440767071611426549823098\ldots\]
which is believed to be correct to the number of decimal places shown. The verification of the useful identity \eqref{veqn} was therefore also listed as an open problem by Knuth \cite[p.355]{Knuth97}. In this article we will prove the correctness of this conjectured identity.

\section{Statement of results}\label{arsity}
In establishing specific results on the mean number of exchanges, subtractions and dyadic divisions performed by the algorithm we work within a general framework defined in terms of the \emph{cost} of processing the pair $(u,v)$, following the approach of V. Baladi and B. Vall\'ee \cite{BaVa05}. We attach a non-negative real weight to each of the fundamental actions which the algorithm may perform at each step, namely: for each natural number $k$ the algorithm might subtract $u$ from $v$ and then divide by $2^k$; or for each natural number $k$ we might subtract $u$ from $v$, divide by $2^k$ and then exchange $u$ and $v$. Clearly the application of the algorithm to a pair $(u,v)$ consists precisely in a particular sequence of repetitions of these fundamental actions. Formally, let us say that a \emph{cost function} associated to the binary Euclidean algorithm is a non-negative function $c \colon \{1,2\} \times \mathbb{N} \to \mathbb{R}$ which is not identically zero. A cost function will be called \emph{regular} if there exists $C>0$ such that $c(i,k)\leq Ck$ for every $(i,k) \in \{1,2\}\times\mathbb{N}$. We consider the quantity $c(1,k)$ to represent the cost associated to subtraction followed by division by $2^k$ and then exchange, and the quantity $c(2,k)$ to represent the cost associated to subtraction followed by division by $2^k$ without exchange. We define the \emph{total cost} $C(u,v)$ associated to the odd pair $(u,v)$ to be the sum of the costs of the fundamental actions performed when processing $(u,v)$. Since the final step of the algorithm results in a pair of the form $(n,n)$ it is \emph{a priori} ambiguous whether or not an exchange is performed in the final step, so by convention we shall always consider that the final step involves no exchange. We define the cost of a general pair of natural numbers to be the cost of the pair formed from the odd parts of the two numbers. The reader may note that, for example, the total number of exchanges $E(u,v)$ may be obtained as the total cost $C(u,v)$ when $c$ is given by $c(1,k)\equiv 1$ and $c(2,k) \equiv 0$, to obtain $C(u,v)\equiv T(u,v)$ one takes $c(i,k)\equiv k$, and to obtain $C(u,v)\equiv S(u,v)$ one simply takes $c(i,k)\equiv 1$.

For each $n \geq 1$ let us define
\begin{align*}\Xi_n^{(1)}&:=\left\{(u,v) \in \mathbb{N}^2 \colon u,v \text{ odd}, 1 \leq u<v \leq n\text{ and }\mathrm{gcd}(u,v)=1\right\}\\
\Xi_n^{(2)}&:=\left\{(u,v) \in \mathbb{N}^2 \colon u,v \text{ odd and }1 \leq u<v \leq n\right\}\\
\Xi_n^{(3)}&:=\left\{(u,v) \in \mathbb{N}^2 \colon 1\leq u<v \leq n\text{ and }\mathrm{gcd}(u,v)=1\right\}\\
\Xi_n^{(4)}&:=\left\{(u,v) \in \mathbb{N}^2 \colon 1\leq u<v \leq n\right\},\end{align*}
and for each $i=1,2,3,4$ let $\Xi^{(i)}:=\bigcup_{n=1}^\infty \Xi^{(i)}_n$. We prove the following theorem on the mean cost of the binary Euclidean algorithm:
\begin{theorem}\label{outcome}
There exists a unique $\xi \in L^1([0,1])$ such that $\int_0^1\xi(x)dx=1$ and such that 
\begin{equation}\label{xieq}\xi(x)=\sum_{k=1}^\infty \frac{1}{(1+2^kx)^2}\xi\left(\frac{1}{1+2^kx}\right)+\frac{1}{(x+2^k)^2}\xi\left(\frac{x}{x+2^k}\right)\end{equation}
Lebesgue almost everywhere. This function may be realised as a real-analytic function $\xi \colon (0,1] \to (0,+\infty)$ which extends analytically to a holomorphic function defined on the right half-plane $\Re(z)>0$. If $c \colon \{1,2\}\times\mathbb{N}\to\mathbb{R}$ is a regular cost function and
\[\mu(c):=\frac{4}{\pi^2\xi(1)}\sum_{k=1}^\infty \frac{1}{2^k}\left(c(2,k)\int_0^{\frac{1}{1+2^k}}\xi(x)dx + c(1,k)\int_{\frac{1}{1+2^k}}^1\xi(x)dx\right),\]
 then for each $i=1,2,3,4$
\[\lim_{n \to \infty} \frac{1}{\#\Xi_n^{(i)}\log n} \sum_{(u,v) \in \Xi_n^{(i)}}C(u,v) = \mu(c).\]
In particular the following asymptotic results hold. If $S(u,v)$ denotes the number of subtractions performed when processing the pair $(u,v)$, then
\begin{align}
\lim_{n \to \infty} \frac{1}{\#\Xi_n^{(i)}\log n} \sum_{(u,v) \in \Xi_n^{(i)}}S(u,v) &=\frac{1}{\sum_{k=1}^\infty \frac{1}{2^k} \int_0^1\log\left(\frac{2^k(1+x)}{1+(2^k-1)x}\right)\xi(x)dx}\label{brentform}\\
&=\frac{2}{\log 4 - \int_0^1\log(1-x)\xi(x)dx}\label{knuthform}\\
&= \frac{4}{\pi^2\xi(1)}\label{valleform}\end{align}
for each $i=1,2,3,4$. If $T(u,v)$ denotes the total number of divisions by $2$ performed when processing the pair $(u,v)$, then 
\[\lim_{n \to \infty} \frac{1}{\#\Xi_n^{(i)}\log n} \sum_{(u,v) \in \Xi_n^{(i)}}T(u,v) = \frac{8}{\pi^2\xi(1)},\]
and if $E(u,v)$ denotes the number of exchanges performed when processing the pair $(u,v)$ then
\begin{align}\label{exch1}\lim_{n \to \infty} \frac{1}{\#\Xi_n^{(i)}\log n} \sum_{(u,v) \in \Xi_n^{(i)}}E(u,v) &= \frac{4}{\pi^2\xi(1)}\left(\sum_{k=1}^\infty \frac{1}{2^k}\int_{\frac{1}{1+2^k}}^1\xi(x)dx\right)\\\label{exch2}
&= \frac{4}{\pi^2\xi(1)} \left(\int_{\frac{1}{2}}^1\xi(x)dx+\frac{2}{3}\int_{\frac{1}{3}}^1\xi(x)dx\right),\end{align}
for each $i=1,2,3,4$. 
\end{theorem}
The equation \eqref{brentform} proves the original heuristic conjecture of R. P. Brent \cite[\S6]{Bren76}. The alternative expression \eqref{knuthform} was conjectured by R. P. Brent \cite{Bren98} and D. E. Knuth \cite[p.351-352]{Knuth97}, the latter in the equivalent form
\[\frac{2}{\log 4 + \int_0^1\left(\frac{1-\int_0^x\xi(t)dt}{1-x}\right)dx}\]
which may be derived from the expression above using integration by parts. The equivalence of \eqref{brentform} with \eqref{knuthform}, proved in \S\ref{se7en} below, has been independently demonstrated by Brent in an unpublished manuscript \cite{Bren97}. The validity of the expression \eqref{valleform} was conjectured by B. Vall\'ee \cite[\S4]{Vall98} and was also listed as an open problem by D. E. Knuth \cite[p.355]{Knuth97}. Note also that Vall\'ee's Theorem \ref{vg} considers averages over $\Xi^{(i)}_n$ for $i=1,2$ but not for $i=3,4$. We have not computed the value of the constant $\sum_{k=1}^\infty \frac{1}{2^k}\int_{1/(1+2^k)}^1\xi(x)dx=\int_{\frac{1}{2}}^1\xi(x)dx+\frac{2}{3}\int_{\frac{1}{3}}^1\xi(x)dx$ which appears in the expressions for the mean number of exchanges, but based on empirical investigations of the number of exchanges conducted by Vall\'ee in \cite{Vall98} it would appear that this constant slightly exceeds one half.

G. Maze \cite{Maze} has previously proved the existence of a unique function $\xi \in L^1([0,1])$ such that $\int_0^1\xi(x)dx=1$ and $\mathcal{L}\xi=\xi$ but was not able to establish stronger regularity properties of $\xi$ such as continuity, nor any of the spectral properties of $\mathcal{L}$ which we require in our proof of Theorem \ref{outcome}. In particular Maze's result does not imply the existence of $\xi(1)$ as a well-defined quantity as is clearly necessary in order to establish \eqref{valleform}.

The results in this article are rooted in a deep study of an extension of Brent's transfer operator $\mathcal{L}$, and this analysis comprises more than half of the paper. Let us briefly introduce some essential notation. Throughout this article we let $\DD$ denote the translated complex unit disc $\DD:=\{z \in \CC \colon |z-1|<1\}$. The notation $\HT$ denotes the Hilbert space of holomorphic functions $\DD \to \CC$ which extend to square-integrable functions along the boundary circle, and $\HI$ denotes the Banach space of bounded holomorphic functions $\DD \to \CC$. When $\mathsf{X}$ is a Banach space we let $\mathcal{B}(\mathsf{X})$ and $\mathcal{K}(\mathsf{X})$ denote the sets of bounded and compact operators on $\mathsf{X}$ respectively. We recall that a function from an open subset $U$ of $\CC^2$ to $\mathsf{X}$ is called holomorphic if it is Fr\'echet differentiable at every point, and this is the case if and only if it is locally expressible as the limit of a convergent power series with coefficients in $\mathsf{X}$. A function from $U$ to $\mathsf{X}$ is holomorphic if and only if its composition with every element of $\mathsf{X}^*$ is holomorphic in the usual sense. A brief review of the concepts and properties from spectral theory and the theory of Banach spaces of holomorphic functions which are used in this article may be found in \S\ref{tweez} below. 

The following theorem summarises our investigation of Brent's operator:
\begin{theorem}\label{MainBrent}
Let $c \colon \{1,2\}\times \mathbb{N} \to \mathbb{R}$ be a regular cost function. Then there exists an open set $\mathcal{U} \subset \CC^2$ which contains the set $\{(s,\omega) \in \CC^2\colon \Re(s)>\frac{2}{3}\text{ and }\omega=0\}$ such that for each $(s,\omega) \in \mathcal{U}$ the formul{\ae}
\[\left(\mathfrak{L}_{s,\omega}f\right)(z):=\sum_{k=1}^\infty\left(\frac{e^{\omega c(1,k)}}{\left(1+2^kz\right)^{2s}}f\left(\frac{1}{1+2^kz}\right)+\frac{e^{\omega c(2,k)}}
{\left(z+2^k\right)^{2s}}f\left(\frac{z}{z+2^k}\right)\right),\]
\[\left(\mathfrak{D}_{s,\omega}f\right)(z):=\sum_{k=1}^\infty\frac{e^{\omega c(2,k)}}
{\left(z+2^k\right)^{2s}}f\left(\frac{z}{z+2^k}\right)\]
define bounded linear operators $\mathfrak{L}_{s,\omega},\mathfrak{D}_{s,\omega} \colon \HT \to \HT$. The corresponding operator-valued maps $(s,\omega) \mapsto \mathfrak{L}_{s,\omega}$ and $(s,\omega) \mapsto \mathfrak{D}_{s,\omega}$ are holomorphic functions from $\mathcal{U}$ to $\mathcal{B}(\HT)$. The following additional properties hold:
\begin{enumerate}[(a)]
\item
For each $s \in \CC$ with $\Re(s)>\frac{1}{2}$ the operator $\mathfrak{L}_{s,0}$ has essential spectral radius not greater than $\frac{\sqrt{2}}{4^{\Re(s)}-\sqrt{2}}$.
\item
The operator $\mathfrak{L}_{1,0}$ has spectral radius equal to one, has a simple eigenvalue at $1$, and has no other spectrum on the unit circle.
\item
There exists a function $\xi \in \HT$ such that $\int_0^1\xi(x)dx=1$, $\xi(x)>0$ for all $x \in (0,1]$, and $\mathfrak{L}_{1,0}\xi=\xi$. There exists $\chi \in \HI$ such that
\[\xi(z)=-\frac{3}{2}\xi(1)\log_2z+\chi(z)\]
 for all $z \in \DD$. More generally, if $\mathfrak{L}_{s,0}\hat\xi = \lambda \hat\xi$ for some $\hat\xi \in \HT$ and $\lambda \in \CC$ such that $|\lambda|>\frac{\sqrt{2}}{4^{\Re(s)}-\sqrt{2}}$ then there exists $\hat\chi \in \HI$ such that
\[\hat\xi(z)=-\frac{\hat\xi(1)}{\lambda-\frac{1}{4^s-1}}\log_2z+\hat\chi(z)\]
for all $z \in \DD$. If $\hat\xi \in \HT$ is an eigenfunction of $\mathcal{L}_{s,\omega}$ which corresponds to a nonzero eigenvalue then it admits an analytic continuation to the right half-plane $\Re(z)>0$. 
\item
The operator $\mathfrak{L}_{s,0}$ has spectral radius strictly less than one when $\Re(s) \geq 1$ and $s$ is not equal to one.
\item
There exist an open set $\mathcal{V}\subset \CC^2$ containing the point $(1,0)$, holomorphic functions $(s,\omega) \mapsto \mathcal{P}_{s,\omega}$ and $(s,\omega) \mapsto \mathcal{N}_{s,\omega}$ from $\mathcal{V}$ to $\mathcal{B}(\HT)$, and a holomorphic function $\lambda \colon \mathcal{V} \to \CC$ such that for all $(s,\omega) \in \mathcal{V}$:
\begin{enumerate}[(i)]
\item
The identity $\mathfrak{L}_{s,\omega} = \lambda(s,\omega)\mathcal{P}_{s,\omega}+\mathcal{N}_{s,\omega}$ holds in the space of bounded operators on $\HT$.
\item
We have $\mathcal{P}_{s,\omega}\mathcal{N}_{s,\omega}=\mathcal{N}_{s,\omega}\mathcal{P}_{s,\omega}=0$.
\item
The spectral radius of $\mathcal{N}_{s,\omega}$ is strictly less than one.
\item
The operator $\mathcal{P}_{s,\omega}$ is a projection with rank equal to one.
\end{enumerate}
The functions $\lambda$ and $\mathcal{P}$ also satisfy $\lambda(1,0)=1$ and $\mathcal{P}_{1,0}f = \left(\int_0^1f(x)dx\right)\xi$ for all $f \in \HT$.
\item
The operator $\mathfrak{L}_{1,0}$ acts continuously on $L^1([0,1])$ with norm $1$. If $f \in L^1([0,1])$ then $\int_0^1(\mathfrak{L}_{1,0}f)(x)dx=\int_0^1f(x)dx$ and $\lim_{n \to \infty}\mathfrak{L}^n_{1,0}f=(\int_0^1f(x)dx)\xi$. In particular if $f \in L^1([0,1])$ and $\mathfrak{L}_{1,0}f=f$ then $f$ is proportional to $\xi$.
\end{enumerate}
\end{theorem}
The proof of Theorem \ref{MainBrent} is quite protracted and is undertaken in several stages which together comprise the greater part of this article. Let us briefly describe the steps involved. The first stage of the proof of Theorem \ref{MainBrent} consists in showing that $\mathfrak{L}_{s,\omega}$ and $\mathfrak{D}_{s,\omega}$ are well-defined bounded operators which depend holomorphically on the parameters $(s,\omega)$, and that the former operator has small essential spectral radius as described in (a). This is the most straightforward part of the proof  and is somewhat similar to the arguments used by Vall\'ee in studying the operator family $\mathfrak{V}_s$. This part of the proof comprises \S\ref{fower} below.

The detailed spectral properties of $\mathfrak{L}_{s,0}$ described in Theorem \ref{MainBrent}(b)--(d) are more difficult to establish and between them their proofs occupy over a third of this article. The proof of these parts of Theorem \ref{MainBrent} comprises \S\ref{twerp} below. 
In constructing the invariant function $\xi$ we use a quasicompact extension of the Kre\u{\i}n-Rutman theorem due to R. Nussbaum \cite{Nuss81}; though versatile and concise this result does not seem to be widely appreciated in the existing literature on transfer operators. (Since our operator is quasicompact rather than compact, the classical results of M. A. Krasnoselski\u{\i} \cite{Kras64} used by Vall\'ee in the analysis of $\mathfrak{V}_s$ do not apply.)

In proving the other parts of Theorem \ref{MainBrent}(b)--(d) we must demonstrate that $\mathfrak{L}_{1,0}$ has no other spectrum on the unit circle, and that $\mathfrak{L}_{1+it,0}$ has no spectrum at all on the unit circle when $t$ is real and nonzero. The essential spectral estimate in Theorem \ref{MainBrent}(a) reduces this to the problem of establishing the absence of additional eigenfunctions corresponding to eigenvalues of unit modulus. Direct solutions to this problem such as are used in \cite{Faiv92,Vall03} involve comparing a presumed eigenfunction with the known positive eigenfunction $\xi$, but in our case this comparison is inhibited by the fact the putative eigenfunction may have a higher order of singularity at $0$ than does the positive invariant function $\xi$. (In the case of Vall\'ee's operators $\mathfrak{V}_s$ it can be shown very early in the proof that all eigenfunctions must have logarithmic singularities at zero and so in \cite{Vall98} this problem does not arise.) This same issue also prevents the use of the projective cone-contraction arguments favoured for such tasks by C. Liverani \cite{Live95}. To circumvent this obstacle we temporarily abandon the space $\HT$ and instead study $\mathfrak{L}_{s,0}$ on a smaller space of functions $\XX$ among whose elements the only possible singularity at $0$ is a logarithmic one. At the end of \S\ref{twerp} we digress slightly from the proof of Theorem \ref{MainBrent} to prove a minor conjecture of Brent (\cite[Conjecture 2.1]{Bren76}). Moving back to the proof of Theorem \ref{MainBrent} we then face the problem that the space $\XX$ is too restrictive to accommodate the action of the operator $\mathfrak{L}_{s,\omega}$ when $\omega$ is nonzero, and for this reason the final stage of the proof of Theorem \ref{MainBrent} consists in transferring our results for the action of $\mathfrak{L}_{s,0}$ on $\XX$ back to the action of $\mathfrak{L}_{s,0}$ on $\HT$. This final stage and the proof of (e)--(f) are undertaken in \S\ref{sixxxxxxxx}.

The fact that the eigenfunctions of $\mathfrak{L}_{s,0}$ extend analytically to the right half-plane suggests the possibility of replacing the space $\HT$ considered in Theorem \ref{MainBrent} (and perhaps also the space $\mathfrak{X}$ considered in \S\ref{twerp}) with a Banach space of holomorphic functions defined in the entire right half-plane. An analysis along these lines has been conducted in the case of the classical Euclidean algorithm by D. Mayer \cite{Maye91}; however, at the present time we have not been successful in identifying a suitable candidate Banach space. In order for such an analysis to result in a proof of Theorem \ref{outcome} the candidate Banach space would have to contain the constant function $\mathbf{1}$, but this is not the case for the spaces considered by Mayer.

The remainder of this article is structured as follows. In \S\ref{tweez} we briefly summarise the ideas from functional analysis and spectral theory which are used in this paper, and as was indicated earlier sections \S\ref{fower}--\ref{sixxxxxxxx} between them comprise the proof of Theorem \ref{MainBrent}. In \S\ref{se7en} we establish some properties of the derivatives of the function $\lambda$ considered in Theorem \ref{MainBrent} which are useful in describing the quantity $\mu(c)$, and in \S\ref{ate} we prove a series of technical results which allow us to relate Dirichlet series of cost functions to the family of operators $\mathfrak{L}_{s,\omega}$ via the equation
\begin{equation}\label{brondesbury}\sum_{(u,v) \in \Xi^{(1)}} \frac{e^{\omega C(u,v)}}{v^{2s}}=\sum_{n=1}^\infty \left(\mathfrak{D}_{s,\omega}\mathfrak{L}_{s,\omega}^{n-1}\mathbf{1}\right)(1)\end{equation}
which is our analogue of \eqref{arghargharh}. In \S\ref{nurf} we apply these results to derive Theorem \ref{outcome} via a Tauberian argument.

\section{Preliminaries from functional analysis}\label{tweez}

\subsection{Hardy spaces}
 The Hardy space $\HT$ is defined to be the set of all holomorphic functions $f\colon \DD \to \CC$ such that the quantity
\begin{equation}\label{hdyeq}\left\|f\right\|_{\HT}:=\sup_{0 < r < 1} \left(\frac{1}{2\pi}\int_0^{2\pi}\left|f\left(1+re^{i\theta}\right)\right|^2d\theta \right)^{\frac{1}{2}}\end{equation}
is finite. The function $\|\cdot\|_{\HT} \colon \HT \to \mathbb{R}$ is a complete norm on $\HT$. If $f \in \HT$ then $f$ extends to a measurable function on the boundary circle $\partial\DD:=\{1+e^{i\theta} \colon \theta \in \mathbb{R}\}$ and satisfies
\[\left\|f\right\|_{\HT}=\left(\frac{1}{2\pi}\int_0^{2\pi}\left|f\left(1+e^{i\theta}\right)\right|^2d\theta \right)^{\frac{1}{2}}.\]
If $f,g \in \HT$ then we may define an inner product on $\HT$ by
\[\left\langle f,g \right\rangle:=\frac{1}{2\pi}\int_0^{2\pi} f\left(1+e^{i\theta}\right)\overline{g\left(1+e^{i\theta}\right)}d\theta\]
and $\HT$ is a Hilbert space with respect to the inner product $\left\langle\cdot,\cdot\right\rangle$ which clearly generates the norm $\|\cdot\|_{\HT}$. The Hardy space $\HT$ admits the following alternative description which will be used heavily in this article: $f \colon \DD \to \CC$ belongs to $\HT$ if and only if there exists a sequence of complex numbers $(a_n)_{n=0}^\infty \in \ell_2$ such that for all $z \in \DD$
\[f(z)=\sum_{n=0}^\infty a_n\left(z-1\right)^n,\]
and when this is the case we have $\|f\|_{\HT}=\left(\sum_{n=0}^\infty a_n^2\right)^{\frac{1}{2}}$. The following standard estimate will be used frequently in the sequel:
\begin{lemma}\label{basic0}
Let $f \in \HT$. Then for all $z \in \DD$
\[|f(z)| \leq \frac{\|f\|_{\HT}}{\sqrt{1-|z-1|^2}}.\]
In particular we have
\[\int_0^1|f(x)|dx \leq \left(\int_0^1 \frac{\|f\|_{\HT}}{\sqrt{1-(x-1)^2}}dx\right)=\frac{\pi}{2}\|f\|_{\HT}.\]
\end{lemma}
\begin{proof}
Let $f(z)=\sum_{n=0}^\infty a_n(z-1)^n$ for all $z \in \DD$. By the Cauchy-Schwarz inequality,
\[|f(z)| = \left| \sum_{n=0}^\infty a_n(z-1)^n\right| \leq \left(\sum_{n=0}^\infty |a_n|^2\right)^{\frac{1}{2}} \left(\sum_{n=0}^\infty |z-1|^{2n}\right)^{\frac{1}{2}}
= \frac{\|f\|_{\HT}}{\sqrt{1-|z-1|^2}}.\]
\end{proof}
Lemma \ref{basic0} implies in particular that for each $z \in \DD$ the map $f \mapsto f(z)$ is a bounded linear functional on $\HT$.  We shall also make use of the Hardy space $\HI$ which is defined to be the set of bounded holomorphic functions $\DD \to \CC$ equipped with the complete norm $\|f\|_{\HI}:=\sup\{|f(z)|\colon z \in \DD\}$. The theory of Hardy spaces is described in detail in numerous textbooks, of which we mention \cite{Dure,Rose,Shap}; all of the properties of Hardy spaces listed above may be found in any of those texts.

\subsection{Essential spectrum}
Recall that a linear operator acting on a complex Banach space is called \emph{Fredholm} if its kernel has finite dimension and its range is closed and has finite codimension. If the codimension of the range is equal to the dimension of the kernel then the operator is said to be \emph{Fredholm of index zero}. For the purposes of this article we shall say that $\lambda \in \CC$ belongs to the \emph{essential spectrum} of a bounded linear operator $L\colon \mathsf{X} \to \mathsf{X}$ if $L-\lambda \mathrm{Id}_{\mathsf{X}}$ is not a Fredholm operator of index zero. A discussion of the relationship between this and other definitions of the essential spectrum may be found in \cite[\S I]{EdEv}.

Let $(X,d)$ be a metric space. The \emph{Kuratowski measure of noncompactness} of a set $A \subseteq X$ is defined to be the quantity
\[\psi(A):=\inf\left\{\delta>0 \colon A\text{ can be covered by finitely many sets of diameter }\leq \delta\right\}.\]
Clearly $\psi(A)=0$ if and only if $\overline{A}$ is compact. If $L$ is a bounded linear operator on a Banach space $(\mathsf{X},\|\cdot\|)$ then we define the \emph{Hausdorff measure of noncompactness} of the operator $L$ to be the quantity
\[\left\|L\right\|_\chi:=\psi\left(\left\{Lx \colon \|x\|\leq 1\right\}\right)\]
where $\psi$ is calculated according to the metric on $\mathsf{X}$ induced by the norm $\|\cdot\|$. It is likewise clear that $L \in \mathcal{K}(\mathsf{X})$ if and only if $\|L\|_{\chi}=0$, and furthermore $\|\cdot\|_{\chi}$ is in fact a seminorm  on $\mathcal{B}(\mathsf{X})$. If $L \in \mathcal{B}(\mathsf{X})$ then we also define
\[\left\|L\right\|_{\mathcal{K}}:=\inf\left\{\|L-K\| \colon K \in \mathcal{K}(\mathsf{X})\right\}.\]
The above definitions are related in the following result which originates in work of R. Nussbaum \cite{Nuss70} and Lebow and Schechter \cite{LeSc}. A complete exposition of this result and the concepts outlined above may be found in \cite[\S I]{EdEv}.
\begin{theorem}[Nussbaum, Lebow--Schechter]\label{Nbm}
Let $\rho_{\mathrm{ess}}(L)$ denote the maximum of the moduli of the elements of the essential spectrum of $L$. Then
\[\rho_{\mathrm{ess}}(L)=\lim_{n \to \infty}\left\|L^n\right\|^{\frac{1}{n}}_\chi=\lim_{n \to \infty}\left\|L^n\right\|^{\frac{1}{n}}_{\mathcal{K}}.  \]
\end{theorem}

Our interest in the essential spectrum is largely due to the following fact which will be frequently invoked without comment: if $\lambda \in \CC$ belongs to the spectrum of $L \in \mathcal{B}(\mathsf{X})$ but does not belong to the essential spectrum, then $\lambda$ is an eigenvalue of $L$ of finite multiplicity and is an isolated point of the spectrum of $L$ (see e.g. \cite[p.40]{EdEv}). Since the spectrum of $L$ is closed and bounded it follows in particular that if $\rho_{ess}(L)<\rho(L)$ then $L$ has an eigenvalue of modulus $\rho(L)$.

\subsection{Separation of spectrum}
Results of the following type are widely used in applications of the theory of transfer operators but the hypotheses have on occasion been unclearly stated. For this reason we include an indication of the proof.
\begin{proposition}\label{splittt}
Let $(\mathsf{X},\|\cdot\|)$ be a Banach space and $L \in \mathcal{B}(\mathsf{X})$ a bounded operator. Suppose that $\lambda$ is an isolated point of the spectrum of $L$, that $L-\lambda \mathrm{Id}_{\mathsf{X}}$ is Fredholm, that every other element of the spectrum of $L$ lies in a closed disc about the origin of radius strictly less than $|\lambda|$, and that $\lambda$ is a simple eigenvalue of $L$ in the sense that $\dim \ker (L-\lambda\mathrm{Id}_{\mathsf{X}})^n=1$ for every integer $n \geq 1$. Let $\Gamma$ be an anticlockwise-oriented closed curve in $\CC$ which encloses $\lambda$ and does not enclose or intersect any other points of the spectrum of $L$. Then the integral
\[P:=\frac{1}{2\pi i}\int_\Gamma \left(z\mathrm{Id}_{\mathsf{X}}-L\right)^{-1} dz\]
defines a bounded operator on ${\mathsf{X}}$ with rank one such that $P^2=P$ and $LP=PL$. If we further define $N:=L(\mathrm{Id}_{\mathsf{X}}-P) \in \mathcal{B}({\mathsf{X}})$ then $L=\lambda P + N$, $NP=PN=0$, and $\rho(N)<|\lambda|$.
\end{proposition}
\begin{proof}
By \cite[Theorem III.6.17]{Kato} the operator $P$ is bounded and satisfies $P^2=P$ and $LP=PL$. Let $X_1$ and $X_2$ denote its image and kernel respectively. Since $P$ is continuous $X_2$ is closed, and since $X_1 = \ker(\mathrm{Id}_{\mathsf{X}}-P)$, $X_1$ is also closed. Since $L$ and $P$ commute we have $LX_1\subseteq X_1$ and $LX_2 \subseteq X_2$. By the result just cited, the spectrum of $L$ restricted to $X_1$ is precisely $\{\lambda\}$, and the spectrum of $L$ restricted to $X_2$ equals the spectrum of $L$ acting on $\mathsf{X}$ with the element $\lambda$ removed; in particular the spectral radius of $L$ restricted to $X_2$ is strictly less than $|\lambda|$ and it follows easily that the spectral radius of $N:=L-LP$ is strictly less than $|\lambda|$. The identity $NP=PN=0$ follows directly from the properties already stated.

Since $L-\lambda\mathrm{Id}_{\mathsf{X}}$ is Fredholm its range is closed and its kernel is finite-dimensional. Using \cite[Lemma IV.5.29]{Kato} it follows that the restriction of $L-\lambda\mathrm{Id}_{\mathsf{X}}$ to $X_1$ also has closed range and finite-dimensional kernel, and by the combination of \cite[Theorem IV.5.30]{Kato} and \cite[Theorem IV.5.10]{Kato} it follows that the dimension of $X_1$ must be finite. The restriction of $L-\lambda\mathrm{Id}_{\mathsf{X}}$ to $X_1$ is thus a linear transformation on a finite-dimensional space with spectrum equal to $\{\lambda\}$, and since $\lambda$ is a simple eigenvalue in the sense described above $X_1$ must be one-dimensional. In particular we have $Lx=\lambda x$ for every $x \in X_1$ and the rank of $P$ is equal to one as claimed. Since $L=LP+N$ by the definition of $N$ it follows that $L=\lambda P +N$ as claimed.\end{proof}

\section{Beginning of the proof of Theorem \ref{MainBrent}}\label{fower}
We now start upon the route towards the proof of Theorem \ref{MainBrent}. In this and all subsequent sections we shall assume that a regular cost function $c \colon \{1,2\}\times\mathbb{N} \to \mathbb{R}$ has been specified. In this section we shall show that $\mathfrak{L}_{1,0}$ preserves integrals along the interval $(0,1)$, prove that the families of operators $\mathfrak{L}_{s,\omega}$ and $\mathfrak{D}_{s,\omega}$ are bounded and holomorphic on $\HT$, and estimate the essential spectral radius of $\mathfrak{L}_{s,0}$.
We begin with the following simple result.
\begin{lemma}\label{wath}
Let $f \colon (0,1] \to \CC$ be Lebesgue integrable. Then the series
\[\left(\mathfrak{L}_{1,0}f\right)(x):=\sum_{k=1}^\infty \frac{1}{(1+2^kx)^2}f\left(\frac{1}{1+2^kx}\right)+\frac{1}{(x+2^k)^2}f\left(\frac{x}{x+2^k}\right)\]
converges Lebesgue almost everywhere and defines a function $\mathfrak{L}_{1,0}f \in L^1([0,1])$ such that $\int_0^1f(x)dx=\int_0^1(\mathfrak{L}_{1,0}f)(x)dx$.
\end{lemma}
\begin{proof}
Suppose first that $g \colon (0,1] \to [0,+\infty]$ is Lebesgue integrable. For each $k \geq 1$ we have
\[\int_0^1\frac{1}{(1+2^kx)^2}g\left(\frac{1}{1+2^kx}\right)dx = \frac{1}{2^k}\int_{\frac{1}{1+2^k}}^1g(u)du\]
and
\[\int_0^1\frac{1}{(x+2^k)^2}g\left(\frac{x}{x+2^k}\right)dx = \frac{1}{2^k}\int_0^{\frac{1}{1+2^k}}g(v)dv\]
using the substitutions $u=\frac{1}{1+2^kx}$ and $v=\frac{x}{x+2^kx}$ respectively, and therefore
\[0 \leq \int_0^1 \left(\mathfrak{L}_{1,0}g\right)(x)dx= \sum_{k=1}^\infty \frac{1}{2^k}\int_0^1g(x)dx=\int_0^1g(x)dx<\infty.\]
In particular the sum which defines $\mathfrak{L}_{1,0}g$ converges almost everywhere to a finite value. The result for a general integrable function $f \colon (0,1] \to \CC$ follows by writing $f$ as a complex linear combination of integrable non-negative functions.
\end{proof}
The next result proves Theorem \ref{MainBrent} up to and including clause (a).
\begin{proposition}\label{FirstEBrent}
There exists an open set $\mathcal{U} \subset \CC^2$ which contains the region $\{(s,\omega) \in \CC^2 \colon \Re(s)>\frac{2}{3}\text{ and }\omega=0\}$ such that for all $(s,\omega) \in \mathcal{U}$ the formul{\ae}
\[\left(\mathfrak{L}_{s,\omega}f\right)(z):=\sum_{k=1}^\infty\left(\frac{e^{\omega c(1,k)}}{\left(1+2^kz\right)^{2s}}f\left(\frac{1}{1+2^kz}\right)+\frac{e^{\omega c(2,k)}}
{\left(z+2^k\right)^{2s}}f\left(\frac{z}{z+2^k}\right)\right),\]
\[\left(\mathfrak{G}_{s,\omega}f\right)(z):=\sum_{k=1}^\infty\frac{e^{\omega c(1,k)}}{\left(1+2^kz\right)^{2s}}f\left(\frac{1}{1+2^kz}\right)\]
and
\[\left(\mathfrak{D}_{s,\omega}f\right)(z):=\sum_{k=1}^\infty\frac{e^{\omega c(2,k)}}
{\left(z+2^k\right)^{2s}}f\left(\frac{z}{z+2^k}\right)\]
define bounded linear operators $\mathfrak{L}_{s,\omega},\mathfrak{G}_{s,\omega},\mathfrak{D}_{s,\omega}\in \mathcal{B}(\HT)$. The functions from $\mathcal{U}$ to $\mathcal{B}(\HT)$ defined by $(s,\omega) \mapsto \mathfrak{L}_{s,\omega}$, $(s,\omega) \mapsto \mathfrak{G}_{s,\omega}$ and $(s,\omega) \mapsto \mathfrak{D}_{s,\omega}$ are holomorphic, and the essential spectral radius of $\mathfrak{L}_{s,0}$ is less than or equal to $\frac{\sqrt{2}}{4^{\Re(s)}-\sqrt{2}}$. Finally, if $(s,\omega) \in \mathcal{U}$ then $\Re(s)>\frac{2}{3}$  and $|\omega|c(i,k)<\frac{k}{6}\log 2$ for all $k\geq 1$ and $i=1,2$.
\end{proposition}
\begin{proof}
Since $c$ is a regular cost function we may choose $C>0$ such that $c(i,k)\leq Ck$ for all $(i,k) \in \{1,2\}\times\mathbb{N}$. Define
\[\mathcal{U}:=\left\{(s,\omega)\in\CC^2 \colon \Re(s)>\frac{2}{3}\text{ and }|\omega|<\frac{\log 2}{6C}\right\}\]
so that when $(s,\omega) \in\mathcal{U}$ we have $\Re(s)>\frac{2}{3}$ and $|\omega|c(i,k) \leq \frac{k}{6}\log 2$ for all $k \geq 1$ and for $i=1,2$ as desired.
To prove that $\mathfrak{L}_{s,\omega}$ is a well-defined element of $\mathcal{B}(\HT)$ and that the corresponding function $(s,\omega) \mapsto \mathfrak{L}_{s,\omega}$ is holomorphic it is clearly sufficient to prove that these properties hold for $\mathfrak{G}_{s,\omega}$ and $\mathfrak{D}_{s,\omega}$, since the corresponding properties of $\mathfrak{L}_{s,\omega}$ then follow from the identity $\mathfrak{L}_{s,\omega}=\mathfrak{G}_{s,\omega}+\mathfrak{D}_{s,\omega}$. We begin by recalling the following classical result which may be found in \cite{Dure,Rose,Shap}: if $\varphi \colon \DD \to \DD$ is holomorphic then the formula ${C}_{\varphi}f:=f \circ \varphi$ defines a bounded linear operator $\mathcal{C}_\varphi \colon \HT \to \HT$, and
\begin{equation}\label{Littlewood}\left\|\mathcal{C}_\varphi\right\|_{\HT} \leq \sqrt{\frac{1+|\varphi(1)-1|}{1-|\varphi(1)-1|}}.\end{equation}
Furthermore, if  the closure of $\varphi(\DD)$ in $\CC$ is contained in $\DD$ then $\mathcal{C}_\varphi \in \mathcal{K}(\HT)$ (see \cite{Rose,Shap}).

For each $k \geq 1$ define two operators $\mathcal{G}_{s,\omega,k}$, $\mathcal{D}_{s,\omega,k}$ on $\HT$ by
\[\left(\mathcal{G}_{s,\omega,k}f\right)(z):=\frac{e^{\omega c(1,k)}}{(1+2^kz)^{2s}}f\left(\frac{1}{1+2^kz}\right),\]
\[\left(\mathcal{D}_{s,\omega,k}f\right)(z):=\frac{e^{\omega c(2,k)}}{(z+2^k)^{2s}}f\left(\frac{z}{z+2^k}\right).\]
It is clear from \eqref{Littlewood} that
\[\left\|\mathcal{G}_{s,\omega,k}\right\|_{\HT} \leq \left(\sup_{z \in \DD}\left|\frac{e^{\omega c(1,k)}}{(1+2^kz)^{2s}}\right|\right)\sqrt{\frac{1+|\frac{1}{1+2^k}-1|}{1-|\frac{1}{1+2^k}-1|}}<\infty\]
\begin{equation}\label{whang}\left\|\mathcal{D}_{s,\omega,k}\right\|_{\HT} \leq \left(\sup_{z \in \DD}\left|\frac{e^{\omega c(2,k)}}{(z+2^k)^{2s}}\right|\right)\sqrt{\frac{1+|\frac{1}{1+2^k}-1|}{1-|\frac{1}{1+2^k}-1|}}<\infty\end{equation}
so that in particular each $\mathcal{G}_{s,\omega,k}$ and each $\mathcal{D}_{s,\omega,k}$ belongs to $\mathcal{B}(\HT)$. Since each of the maps $z \mapsto 1/(1+2^kz)$ takes the closure of $\DD$ into the interior of $\DD$ the operators $\mathcal{G}_{s,\omega,k}$ are all compact. It is furthermore not difficult to see that each of these operators may be locally written as a convergent power series in $(s,\omega)$ with coefficients in $\mathcal{B}(\HT)$, and hence the operator-valued functions $(s,\omega) \mapsto \mathcal{G}_{s,\omega,k}$ and $(s,\omega)\mapsto\mathcal{D}_{s,\omega,k}$ are holomorphic. To show that $\mathfrak{G}_{s,\omega}$, $\mathfrak{D}_{s,\omega}$ are well-definded operators which depend holomorphically on $(s,\omega)$ it is therefore sufficient to show that the series $\sum_{k=1}^\infty \mathcal{G}_{s,\omega,k}$ and $\sum_{k=1}^\infty \mathcal{D}_{s,\omega,k}$ converge in $\mathcal{B}(\HT)$ in a locally uniform manner with respect to $(s,\omega)$. Since the sum of a convergent series of compact operators is compact this will also suffice to show that $\mathfrak{G}_{s,\omega}$ is compact for every $(s,\omega) \in \mathcal{U}$.

Let us therefore prove that these series converge in the required manner. The case of $\mathfrak{D}_{s,\omega}$ is straightforward: we have
\begin{equation}\label{whong}\sqrt{\frac{1+|\frac{1}{1+2^k}-1|}{1-|\frac{1}{1+2^k}-1|}}=\sqrt{\frac{1+2^k+|1-(1+2^k)|}{1+2^k-|1-(1+2^k)|}}=\sqrt{1+2^{k+1}}<2^{\frac{k}{2}+1}\end{equation}
for each $k \geq 1$, and since also
\begin{align*}\left|\frac{1}{(z+2^k)^{2s}}\right|&=\left|\exp\left(-2s\log\left(z+2^k\right)\right)\right|\\
&=\exp\Re\left(-2s\log\left(z+2^k\right)\right)\\
&=\exp\left(-2\Re(s)\log|z+2^k| + 2\Im(s)\arg(z+2^k)\right)\\
&\leq e^{\pi|\Im(s)|}\left|z+2^k\right|^{-2\Re(s)}\leq e^{\pi|\Im(s)|}4^{-k\Re(s)}\end{align*}
it follows from \eqref{whang} and \eqref{whong} that
\begin{align*}\sum_{k=1}^\infty \left\|\mathcal{D}_{s,\omega,k}\right\|_{\HT} &\leq 2e^{\pi|\Im(s)|} \sum_{k=1}^\infty e^{\Re(\omega)c(2,k)}2^{-k\left(2\Re(s)-\frac{1}{2}\right)}\\
&\leq 2e^{\pi|\Im(s)|} \sum_{k=1}^\infty 2^{-\frac{2}{3}k}  <\infty\end{align*}
so that the series $\sum_{k=1}^\infty \mathcal{D}_{s,\omega,k}$ converges locally uniformly in $(s,\omega)$ to the limit $\mathfrak{D}_{s,\omega}$ which is well-defined and depends holomorphically on $(s,\omega)$.

In order to bound the norms of the operators $\mathcal{G}_{s,\omega,k}$ we use an alternative estimate suggested by the analysis of B. Vall\'ee \cite{Vall98}, based on the following theorem of R. M. Gabriel \cite{Gabr28}: if $U \subset \CC$ is an open ball, $g \colon U \to \CC$ is holomorphic, $\Gamma$ is a circular contour in $U$, and $\gamma$ is a rectifiable convex Jordan curve enclosed by $\Gamma$, then
\begin{equation}\label{Gabr}\int_\gamma |g(z)|^2 |dz| \leq 2\int_\Gamma |g(z)|^2 |dz|.\end{equation}
Our interest is in the case where $\gamma$ is also circular, and in this case \eqref{Gabr} could also be deduced from a related theorem in which the integrand is taken to be positive and subharmonic \cite{Gab}. For a modern treatment and related results see \cite{Gran99}.

For each $k \geq 1$ let us define $\varphi_k(z):=\frac{1}{1+2^kz}$ for every $z \in \DD$. Using the substitution $u=\varphi_k(z)$ together with the estimate $|\omega|c(1,k)\leq \frac{k}{6}\log 2$ which follows from the definition of $\mathcal{U}$ we may obtain
\begin{align}\label{quackquack}\left\|\mathcal{G}_{s,\omega,k}f\right\|_{\HT}^2 &=  \int_{\partial\DD} \left|\frac{e^{\omega c(1,k)}}{(1+2^kz)^{2s}}f\left(\frac{1}{1+2^kz}\right)\right|^2|dz|\\
\nonumber &=\frac{e^{\Re(\omega) c(1,k)}}{2^k}\int_{\varphi_k(\partial\DD)} \left|u^{2s-2}f(u)\right|^2|du|\\
\nonumber &=2^{-\frac{5}{6}k}\int_{\varphi_k(\partial\DD)} \left|u^{2s-2}f(u)\right|^2|du|.\end{align}
Now, if $|z-1|=1$  then
\begin{align*}\left|\left(\varphi_k(z)\right)^{2s-2}\right|&=\left|\left(1+2^kz\right)^{2-2s}\right|\\
&=\left|\exp\left((2-2s)\log\left(1+2^kz\right)\right)\right|\\
&=\exp\left(\left(2-2\Re(s)\right)\log\left|1+2^kz\right| + 2\Im(s)\arg\left(1+2^kz\right)\right)\\
&\leq e^{\pi|\Im(s)|}|1+2^kz|^{2-2\Re(s)}< 4e^{\pi|\Im(s)|}2^{\frac{2}{3}k} \end{align*}
since $2-2\Re(s)< \frac{2}{3}$, and hence
\begin{equation}\label{quackquackquack}2^{-\frac{5}{6}k}\int_{\varphi_k(\partial\DD)} \left|u^{2s-2}f(u)\right|^2|du|\leq 4e^{\pi|\Im(s)|}2^{-\frac{k}{6}}\int_{\varphi_k(\partial\DD)} \left|f(u)\right|^2|du|.\end{equation}
Choose a circular contour $\Gamma$ in $\DD$ which is centered at $1$ and has radius large enough that $\Gamma$ encloses the curve $\varphi_k(\partial\DD)$. Combining \eqref{quackquack}, \eqref{quackquackquack} and \eqref{Gabr} we find that
\begin{align*}\left\|\mathcal{G}_{s,\omega,k}f\right\|_{\HT}^2 &\leq 4e^{\pi|\Im(s)|}2^{-\frac{k}{6}}\int_{\varphi_k(\partial\DD)} \left|f(u)\right|^2|du|\\
& \leq  8e^{\pi|\Im(s)|}2^{-\frac{k}{6}}\int_{\Gamma} \left|f(z)\right|^2|dz| \\
&\leq 8e^{\pi|\Im(s)|}2^{-\frac{k}{6}}\|f\|_{\HT}^2,\end{align*}
where the last inequality follows from the definition of $\|\cdot\|_{\HT}$ given in \eqref{hdyeq}. We conclude from this estimate that for each $(s,\omega) \in \mathcal{U}$ the sum $\mathfrak{G}_{s,\omega}=\sum_{k=1}^\infty \mathcal{G}_{s,\omega,k}$ is a convergent series of compact operators, and hence defines an element of $\mathcal{K}(\HT)$. Since this convergence is locally uniform with respect to $(s,\omega)$, the function $(s,\omega) \mapsto \mathfrak{G}_{s,\omega}$ is holomorphic.

To complete the proof of the proposition it remains to show that when $(s,0) \in \mathcal{U}$ the essential spectral radius of $\mathfrak{L}_{s,0}$ is bounded above by $\frac{\sqrt{2}}{4^{\Re(s)}-\sqrt{2}}$. The composition of a bounded operator with a compact operator is compact, and it follows that for each $n \geq 1$ the expression $\mathfrak{L}_{s,0}^n=(\mathfrak{G}_{s,0}+\mathfrak{D}_{s,0})^n$ expands into a sum of $2^{n-1}$ compact operators (which arise from products which involve at least one instance of $\mathfrak{G}_{s,0}$) and a single possibly noncompact operator, $\mathfrak{D}_{s,0}^n$. We therefore have
\[\inf\left\{\left\|\mathfrak{L}_{s,0}^n - K\right\|_{\HT} \colon K \in \mathcal{K}(\HT)\right\} \leq \left\|\mathfrak{D}_{s,0}^n\right\|_{\HT}\]
for every $n \geq 1$, and it follows from Theorem \ref{Nbm} that the essential spectral radius of $\mathfrak{L}_{s,\omega}$ is bounded by the ordinary spectral radius of $\mathfrak{D}_{s,\omega}$. To prove the proposition we will show that this latter quantity is bounded by $\frac{\sqrt{2}}{4^{\Re(s)}-\sqrt{2}}$.

For each $k \geq 1$ let us define $\phi_k \colon \DD \to \DD$ by $\phi_k(z):=\frac{z}{z+2^k}$. For each $f \in \HT$ and $z \in \DD$ we may write the sum defining the function $\mathfrak{D}_{s,0}f$ alternatively as
\[\left(\mathfrak{D}_{s,0}f\right)(z)=\sum_{k=1}^\infty \frac{1}{2^{ks}}\left(\phi_k'(z)\right)^s f\left(\phi_k(z)\right)\]
and in this manner we may for each $n \geq1$ write $\left(\mathfrak{D}_{s,0}^nf\right)(z)$ as
\[\sum_{k_1,\ldots,k_n=1}^\infty \left(2^{-\sum_{i=1}^n k_i}\prod_{i=1}^n \phi_{k_i}'((\phi_{k_{i-1}} \circ \cdots \circ  \phi_{k_1})(z))\right)^s f\left(\left(\phi_{k_n}\circ \cdots \circ \phi_{k_1}\right)(z)\right)\]
\[=\sum_{k_1,\ldots,k_n=1}^\infty \left(2^{-\sum_{i=1}^n k_i}(\phi_{k_n} \circ \cdots \circ  \phi_{k_1})'(z)\right)^s f\left(\left(\phi_{k_n}\circ \cdots \circ \phi_{k_1}\right)(z)\right).\]
Now, the composition $\phi_{k_n} \circ \cdots \circ  \phi_{k_1}$ has the form $(\phi_{k_n} \circ \cdots \circ  \phi_{k_1})(z)=(\alpha z + \beta)/(\gamma z + \delta)$ where $\alpha,\beta,\gamma,\delta$ satisfy
\[\left(\begin{array}{cc}\alpha&\beta\\\gamma&\delta\end{array}\right)=\left(\begin{array}{cc}1&0\\1&2^{k_n}\end{array}\right)\left(\begin{array}{cc}1&0\\1&2^{k_{n-1}}\end{array}\right)\cdots \left(\begin{array}{cc}1&0\\1&2^{k_2}\end{array}\right) \left(\begin{array}{cc}1&0\\1&2^{k_1}\end{array}\right).\]
An easy inductive argument establishes the relation
\[(\phi_{k_n} \circ \cdots \circ  \phi_{k_1})(z) = \frac{z}{(1+\sum_{i=2}^{n} 2^{k_i+\ldots +k_n})z + 2^{k_1+\ldots+k_n}   }\]%=\frac{2^{-\sum_{i=1}^n k_i}z}{(\sum_{i=1}^{n} 2^{-(k_1+\ldots +k_i)})z + 1 }\]
from which an elementary calculation yields
\[\left(2^{-\sum_{i=1}^n k_i}(\phi_{k_n} \circ \cdots \circ  \phi_{k_1})'(z)\right)^s =\frac{1}{\left((1+\sum_{i=2}^{n} 2^{k_i+\ldots +k_n})z + 2^{k_1+\ldots+k_n}  \right)^{2s} }
.\]
We may thus compute
\[\sup_{z \in \DD}\left|\left(2^{-\sum_{i=1}^n k_i}(\phi_{k_n} \circ \cdots \circ  \phi_{k_1})'(z)\right)^s \right| \leq \frac{e^{\pi|\Im(s)|}}{4^{(\sum_{i=1}^nk_i)\Re(s)}}\]
and
\begin{align*}\left\|f \circ (\phi_{k_n} \circ \cdots \circ  \phi_{k_1})\right\|_{\HT} &\leq \left(\sqrt{1+2\sum_{i=1}^n 2^{k_i+\ldots +k_n}}\right)\|f\|_{\HT}\\
 &\leq \left(\sqrt{2\sum_{i=0}^{k_1+\ldots +k_n}2^i}\right)\|f\|_{\HT}\\
&\leq\left(2^{1+\sum_{i=1}^n\frac{k_i}{2}}\right)\|f\|_{\HT}\end{align*}
in a similar manner to our earlier calculation of the bounds on $\|\mathcal{D}_{s,\omega,k}\|_{\HT}$. It follows that
\begin{align*}
\left\|\mathfrak{D}_{s,0}^n\right\|_{\HT} &\leq 2e^{\pi|\Im(s)|}\sum_{k_1,\ldots,k_n=1}^\infty 2^{-\left(2\Re(s)-\frac{1}{2}\right)(\sum_{i=1}^nk_i) }\\
&= 2e^{\pi|\Im(s)|}\left(\sum_{k=1}^\infty 2^{-\left(2\Re(s)-\frac{1}{2}\right)k}\right)^n\\
&=\frac{2e^{\pi|\Im(s)|} }{\left(2^{2\Re(s)-\frac{1}{2}}-1\right)^n}\end{align*}
and this clearly yields
\[\lim_{n \to \infty} \left\|\mathfrak{D}_{s,0}^n\right\|_{\HT}^{\frac{1}{n}}\leq \frac{\sqrt{2}}{4^{\Re(s)}-\sqrt{2}}\]
as desired. The proof is complete.
\end{proof}

\section{Analysis of Brent's operator on $\XX$}\label{twerp}
As was indicated in \S\ref{arsity}, in order to prove those parts of Theorem \ref{MainBrent} which pertain to the point spectrum of $\mathfrak{L}_{s,0}$ we will find it necessary to work on a smaller function space than $\HT$. This quite lengthy process is undertaken  in the current section.

Let $\XX$ be the set of all holomorphic functions $f \colon \DD \to \CC$ with the property that there exist $\alpha \in \CC$ and $g \in \HI$ such that $f(z)=\alpha \log_2 z + g(z)$ for all $z \in \DD$. Clearly every $f \in \XX$ has a unique representation in this form. If $f \in \XX$ has the form $f(z)=\alpha \log_2 z + g(z)$ for all $z \in \DD$ where $\alpha \in \CC$ and $g \in \HI$ then we define $\|f\|_{\XX}:=|\alpha|+\|g\|_{\HI}$. It is clear that $\XX$ is a Banach space with respect to this norm. The objective of this section is to prove the following result: 
\begin{theorem}\label{SimBrenSpec}
For each $s \in \CC$ with $\Re(s)>\frac{2}{3}$ the formula
\[\left(\mathfrak{L}_{s,0}f\right)(z):=\sum_{k=1}^\infty\left(\frac{1}{\left(1+2^kz\right)^{2s}}f\left(\frac{1}{1+2^kz}\right)+\frac{1}{\left(z+2^k\right)^{2s}}f\left(\frac{z}{z+2^k}\right)\right)\]
defines a bounded linear operator $\mathfrak{L}_{s,0} \in \mathcal{B}(\XX)$. This family of operators satisfies the following properties:
\begin{enumerate}[(i)]
\item
For each $s$ the essential spectral radius of $\mathfrak{L}_{s,0}$ acting on $\XX$ is less than or equal to $\frac{\sqrt{2}}{4^{\Re(s)}-\sqrt{2}}$.
\item
The operator $\mathfrak{L}_{1,0}$ acting on $\XX$ has spectral radius equal to one, has a simple isolated eigenvalue at $1$, and has no other spectrum on the unit circle.
\item
There exists a unique function $\xi \in \XX$ such that $\mathfrak{L}_{1,0}\xi = \xi$, $\int_0^1\xi(x)dx=1$, and $\xi(x)$ is real and strictly positive for all $x \in (0,1]$. There exists $\chi \in \HI$ such that for all $z \in \DD$,
\[\xi(z)=-\frac{3}{2}\xi(1)\log_2 z + \chi(z).\]
More generally, if $\mathfrak{L}_{s,0}\hat\xi = \lambda \hat \xi$ for some $\hat\xi \in \XX$ and complex number $\lambda \neq \frac{1}{4^s-1}$ then there exists $\hat\chi \in \HI$ such that
\[\hat\xi(z)=-\frac{\hat\xi(1)}{\lambda-\frac{1}{4^s-1}}\log_2z + \hat\chi(z)\]
for all $z \in \DD$.
\item
If $\Re(s) \geq 1$ and $s \neq 1$ then the spectral radius of $\mathfrak{L}_{s,0}$ acting on $\XX$ is strictly less than $1$.
\end{enumerate}
\end{theorem}
The proof of Theorem \ref{SimBrenSpec} is quite prolonged and is divided into a series of stages: the boundedness of the operator is proved below in Corollary \ref{eigs}, property (i) is proved in Proposition \ref{essspec}, and properties (ii)-(iv) are proved in Proposition \ref{itch}. 
With somewhat more effort one may show that the function $s\mapsto \mathfrak{L}_{s,0}$ is a holomorphic mapping into $\mathcal{B}(\XX)$, but this fact is not needed in order to prove the main results of this article. In any case we see no reason to believe that $\mathfrak{L}_{s,\omega}$ should preserve $\XX$ when $\omega$ is nonzero and $c$ is an arbitrary cost function, and this circumstance renders $\XX$ an unsuitable space in which to attempt to prove the full statement of Theorem \ref{MainBrent}.

By working in $H^p(\DD)$ in place of $\HT$ for some $p \in (2,+\infty)$ throughout this and the previous section it would be possible to sharpen the estimate for the essential spectral radius of $\mathfrak{L}_{s,0}$ acting on $\XX$ to $\frac{2^{1/p}}{4^{\Re(s)}-2^{1/p}}$ when $\Re(s)>1-\varepsilon$ for a constant $\varepsilon$ depending on $p$. By taking $p$ arbitrarily large  we could in this manner obtain a bound of $\frac{1}{4^{\Re(s)}-1}$ when $\Re(s) \geq 1$. Since we shall have no use for such a sharpened estimate in this document we omit this analysis.

A byproduct of the analysis in this section is that we may rigorously verify the following minor conjecture of R. P. Brent:
\begin{proposition}[{\cite[Conjecture 2.1]{Bren76}}]\label{bcnj}
Define inductively a sequence of functions $F_n \colon [0,1] \to [0,1]$ by $F_0(x):=x$ for all $x \in [0,1]$ and
\[F_{n+1}(x):=1+\sum_{k=1}^\infty \frac{1}{2^k}\left(F_n\left(\frac{x}{x+2^k}\right)-F_n\left(\frac{1}{1+2^kx}\right)\right)\]
for all $x \in [0,1]$ for every integer $n \geq 0$. Then there exist a real analytic function $F_\infty \colon (0,1] \to \mathbb{R}$ and real numbers $K>0$, $\theta \in (0,1)$ such that for all $x \in (0,1]$ and $n \geq 1$
\[\left|F_n(x)-F_\infty(x)\right|\leq K\theta^n |x \log x|.\]
\end{proposition}
Since the proof of this result is tangential to the main thrust of this section we postpone it to subsection \ref{brentconj} below.

\subsection{Elementary estimates}
We will begin the proof of Theorem \ref{SimBrenSpec} by listing some elementary but useful results which will be repeatedly applied in this and the following section. 
\begin{lemma}\label{basic1}
Let $z \in \CC$ with $\Re(z)>0$, and let $\ell \in \mathbb{Z}$. If $|z| \leq M$ for some real number $M>0$, then
\[\left|\frac{1}{1+2^\ell z}-1\right| \leq \frac{M}{\sqrt{M^2+4^{-\ell}}}.\]
\end{lemma}
\begin{proof}
We may write
\[\left|\frac{1}{1+2^\ell z}-1\right|^2=\frac{|2^\ell z|^2}{|1+2^\ell z|^2}=\frac{|z|^2}{|z|^2+2^{1-\ell}\Re(z)+4^{-\ell}}<\frac{|z|^2}{|z|^2+4^{-\ell}} \leq \frac{M^2}{M^2+4^{-\ell}}\]
where we have used the fact that for each $\delta>0$ the function $x \mapsto x^2/(x^2+\delta)$ is monotone increasing for positive real $x$.
\end{proof}
\begin{lemma}\label{basic2}
Let $f \in \HT$. Then for all $z \in \DD$
\[\left|f(z)-f(1)\right|\leq \frac{|z-1|\cdot\|f\|_{\HT}}{\sqrt{1-|z-1|^2}}.\]
\end{lemma}
\begin{proof}
Let $f(z)=\sum_{n=0}^\infty a_n(z-1)^n$ for all $z \in \DD$. Using the Cauchy-Schwarz inequality we have
\begin{align*}\left|f(z)-f(1)\right|&=\left|\sum_{n=1}^\infty a_n(z-1)^n\right| \\
&= |z-1|\cdot\left|\sum_{n=0}^\infty a_{n+1}(z-1)^n \right|\\
&\leq|z-1|\left(\sum_{n=0}^\infty |a_{n+1}|^2\right)^{\frac{1}{2}}\left(\sum_{n=0}^\infty |z-1|^{2n}\right)^{\frac{1}{2}}\\
&\leq\frac{|z-1|\cdot\|f\|_{\HT}}{\sqrt{1-|z-1|^2}}\end{align*}
for all $z \in \DD$ as required.
\end{proof}
\begin{lemma}\label{basic3}
If $f \in \XX$ then $f \in \HT$ and $\|f\|_{\HT} \leq2\|f\|_{\XX}$.
\end{lemma}
\begin{proof}
In view of the power series $\log z = \sum_{n=1}^\infty \frac{(-1)^{n+1}}{n}(z-1)^n$ which is valid for all $z \in \DD$ we have $\|\log\|_{\HT}^2=\frac{\pi^2}{6}$. Given $f \in \XX$ let us write $f(z)=\alpha \log_2 z + g(z)$ where $g \in \HI$ and $\alpha \in \CC$. Clearly
\[\|f\|_{\HT} \leq |\alpha|\cdot\|\log_2\|_{\HT} + \|g\|_{\HT} \leq \frac{|\alpha|\pi}{\sqrt{6}\log 2} + \|g\|_{\HI} \leq 2\|f\|_{\XX}\]
as required.
\end{proof}
\begin{lemma}\label{basic4}
Let $M>0$ and $s \in \CC$. Then there exists a constant $K\geq 0$ such that for all $z \in \CC$ with $\Re(z)>0$ and $|z|\leq M$ and all integers $k \geq 1$,
\[\left|\frac{1}{\left(1+2^{-k} z\right)^{2s}}-1\right| \leq \frac{K|z|}{2^k}.\]
\end{lemma}
\begin{proof}
When $z_1,z_2 \in\CC$ with $\Re(z_1),\Re(z_2)\geq 0$ the mean value theorem implies that
\[|\log(1+z_1)-\log(1+z_2)|\leq \left(\sup_{\substack{\omega \in \CC\\\Re(\omega) \geq 0}} \frac{1}{|1+\omega|}\right)|z_1-z_2| \leq |z_1-z_2|,\]
and we therefore in particular have $|\log(1+2^{-k} z)| =|\log(1+2^{-k} z)-\log 1| \leq 2^{-k}|z|$. Using the elementary inequality $|e^\omega-1|\leq |\omega|e^{|\omega|}$ which is valid for all $\omega \in \CC$ we obtain
\begin{align*}\left|\frac{1}{(1+2^{-k} z)^{2s}}-1\right| &=\left|e^{-2s\log(1+2^{-k} z)}-1\right|\\
& \leq |2s\log(1+2^{-k} z)|e^{|2s\log(1+2^{-k} z)|}\\
& \leq \left(|2s|e^{|2^{1-k}sz|}\right)\left(2^{-k}|z|\right)\\
&\leq \left(|2s|e^{M|s|}\right)2^{-k}|z|,\end{align*}
so we may take $K:=|2s|e^{M|s|}$.
\end{proof}

\subsection{Auxiliary operator estimates}

In this subsection we investigate the action on $\XX$ of the operator $\mathfrak{G}_{s,0}$ which was considered in the proof of Proposition \ref{FirstEBrent}. Our analysis centres around the observation by B. Vall\'ee in \cite[Prop. 3]{Vall98} that functions in the image of $\mathfrak{G}_{s,0}$ may be decomposed into three parts with very particular properties. However, where Vall\'ee decomposes a single function $\mathfrak{G}_{s,0}f \in \HT$ into a sum of three elements of $\HT$, we wish to decompose $\mathfrak{G}_{s,0}$ itself into a sum of three bounded operators from $\HT$ to $\XX$, and our analysis is correspondingly more intricate.
\begin{lemma}\label{BbddX}
For each $f \in \HT$ and $s,z \in \CC$ such that $\Re(z)>0$ and $\Re(s)>\frac{2}{3}$, the series
\[(\mathfrak{B}_sf)(z):=\sum_{k=0}^\infty \left(f(1)-\frac{1}{(1+2^{-k}z)^{2s}}f\left(\frac{1}{1+2^{-k}z}\right)\right)\]
converges absolutely. The function $\mathfrak{B}_sf$ thus defined is holomorphic in the region $\Re(z)>0$, and for each $M \geq 1$ there is a constant $C_1$ depending only on $M$ and $s$ such that 
\[\sup\left\{\left|(\mathfrak{B}_sf)(z)\right| \colon \Re(z)>0 \text{ and }|z| \leq M\right\} \leq C_1\|f\|_{\HT}.\]
\end{lemma}
\begin{proof}
Let $M \geq 1$ and $s \in \CC$ with $\Re(s)>\frac{2}{3}$. Let $z \in \CC$ with $\Re(z)>0$ and $|z|\leq M$, and let $k$ be a non-negative integer. By Lemma \ref{basic1} we have
\begin{equation}\label{bilge}\left|\frac{1}{1+2^{-k}z} - 1\right| \leq \frac{M}{\sqrt{M^2 +4^k}}  <1,\end{equation}
which in particular implies that $1/(1+2^{-k}z) \in \DD$ and therefore $f(1/(1+2^{-k}z))$ is well-defined. Using Lemma \ref{basic2} together with \eqref{bilge} it follows that
\begin{align}\label{madel}\left|f\left(\frac{1}{1+2^{-k}z}\right)-f(1)\right| &\leq \left|\frac{1}{1+2^{-k}z}-1\right|\left(\frac{\|f\|_{\HT}}{\sqrt{1-\left|\frac{1}{1+2^{-k}z}-1\right|^2}}\right)\\\nonumber
&\leq \left(\frac{M}{\sqrt{M^2+4^k}}\right)\left(\frac{\|f\|_{\HT}}{\sqrt{1-\frac{M^2}{M^2+4^k}}}\right)\\\nonumber
&= \frac{M\|f\|_{\HT}}{2^k}.
\end{align}
By Lemma \ref{basic4} there exists a constant $K>0$ depending on $M$ and $s$ such that
\begin{equation}\label{madele}\left|\frac{1}{(1+2^{-k}z)^{2s}} -1\right| \leq \frac{K|z|}{2^k} \leq \frac{KM}{2^k},\end{equation}
and this clearly implies in particular 
\begin{equation}\label{madelei}\left|\frac{1}{(1+2^{-k}z)^{2s}}\right| \leq 1+KM.\end{equation}
We have $|f(1)| \leq \|f\|_{\HT}$ by Lemma \ref{basic0}, and using this together with \eqref{madel}, \eqref{madele} and \eqref{madelei} we obtain
\begin{align*}\left|\frac{1}{(1+2^{-k}z)^{2s}}f\left(\frac{1}{1+2^{-k}z}\right) - f(1)\right| \leq& \left|\frac{1}{(1+2^{-k}z)^{2s}}\right|\cdot\left| f\left(\frac{1}{1+2^{-k}z}\right) - f(1)\right|\\
& + \left|\frac{1}{\left(1+2^{-k}z\right)^{2s}}-1\right|\cdot |f(1)|\\
\leq& \frac{\left(M+KM^2+KM\right)\|f\|_{\HT}}{2^k}\\
\leq& \frac{C_1\|f\|_{\HT}}{2^{k+1}},\end{align*}
say, for all $z \in \CC$ such that $\Re(z)>0$ and $|z|\leq M$, and all integers $k \geq 0$, where $C_1 \geq 1$ depends only on $s$ and on the constant $M \geq 1$. We deduce that the series defining $(\mathfrak{B}_sf)(z)$ converges uniformly with respect to $z$ in this region and hence defines a holomorphic function in its interior, which clearly satisfies the bound specified in the statement of the lemma. Since $M$ is arbitrary we conclude that for each fixed $s$, $\mathfrak{B}_sf$ is a holomorphic function defined for all $z \in \CC$ such that $\Re(z)>0$.
\end{proof}

\begin{lemma}\label{GbddH2}
If $f \in \HT$, then for each $s,z \in \CC$ such that $\Re(z)>0$ and $\Re(s)>\frac{2}{3}$ the series
\[(\mathfrak{G}_{s,0}f)(z):=\sum_{k=1}^\infty \frac{1}{(1+2^kz)^{2s}}f\left(\frac{1}{1+2^kz}\right)\]
is absolutely convergent, and the function $\mathfrak{G}_{s,0}f$ thus defined is holomorphic in the region $\Re(z)>0$. For each $s$ and each pair of real numbers $m,M$ such that $0<m\leq1<M$ there exists a constant $C_2>0$ which does not depend on $f$ such that
\[\sup\left\{|(\mathfrak{G}_{s,0}f)(z)| \colon \Re(z)>0 \text{ and }m \leq |z| \leq M\right\} \leq C_2 \|f\|_{\HT}.\]
\end{lemma}
\begin{proof}
Fix $f, m, M$ and $s$ throughout the proof. If $k \geq 1$ and $z \in \CC$ with $\Re(z)>0$ and $m\leq |z|\leq M$, then $1/(1+2^kz) \in \DD$ by Lemma \ref{basic1} and therefore $f(1/(1+2^kz))$ is well-defined. Using the elementary estimate
\[\frac{1}{1-\left|1-\frac{1}{1+2^kz}\right|^2} = \frac{|1+2^kz|^2}{\left|1+2^kz\right|^2-\left|2^kz\right|^2}=\frac{1+2^{k+1}\Re(z)+4^k|z|^2}{1+2^{k+1}\Re(z)}<4^{k+1}M^2\]
together with Lemma \ref{basic0} we may obtain the inequality
\[\left|f\left(\frac{1}{1+2^kz}\right)\right| \leq \frac{\|f\|_{\HT}}{\sqrt{1-\left|1-\frac{1}{1+2^kz}\right|^2}}\leq  2^{k+1}M\|f\|_{\HT}.\]
Now, since additionally
\begin{align*}\left|\frac{1}{(1+2^kz)^{2s}}\right|&=\left|\exp\left(-2s \log (1+2^kz)\right)\right|\\
&=\exp\left(\Re\left(-2s \log (1+2^kz)\right)\right)\\
&=\exp\left(-2\Re(s)\log\left|1+2^kz\right| +2\Im(s)\arg\left(1+2^kz\right)\right) \\
&\leq \frac{e^{\pi|\Im(s)|}}{\left|1+2^kz\right|^{2\Re(s)}}\leq \frac{e^{\pi|\Im(s)|}}{4^{k\Re(s)}m^{2\Re(s)}},\end{align*}
it follows that for all $z \in \CC$ such that $\Re(z)>0$ and $m \leq |z|\leq M$
\begin{align*}|\mathfrak{G}_{s,0}f(z)|&=\left|\sum_{k=1}^\infty \frac{1}{(1+2^kz)^{2s}}f\left(\frac{1}{1+2^kz}\right)\right| \\
&\leq \frac{2e^{\pi|\Im(s)|} M\|f\|_{\HT}}{m^{2\Re(s)}}\left(\sum_{k=1}^\infty 2^{-(2\Re(s)-1)k}\right)\\
&\leq \left(\frac{2e^{\pi|\Im(s)|} M}{m^{2\Re(s)}\left(\sqrt[3]{2}-1\right)} \right)\|f\|_{\HT}\\
&=C_2\|f\|_{\HT},\end{align*}
say, as required.
Since the series defining $(\mathfrak{G}_{s,0}f)(z)$ converges absolutely uniformly over this region it defines a holomorphic function in the interior of the region, and this function satisfies the bound claimed in the statement of the proposition. Since $m$ and $M$ are arbitrary it follows that for each fixed $s$, the function $\mathfrak{G}_{s,0}f$ is holomorphic throughout the region $\Re(z)>0$.
\end{proof}

\begin{lemma}\label{CbddX}
Let $f \in \HT$ and $s \in \CC$ with $\Re(s)>\frac{2}{3}$. Then the expression
\begin{align*}(\mathfrak{C}_sf)(z):=&\sum_{k=0}^\infty \left(\frac{1}{(1+2^{-k}z)^{2s}}f\left(\frac{1}{1+2^{-k}z}\right)-f(1)\right) \\
&+\sum_{k=1}^\infty \frac{1}{(1+2^kz)^{2s}}f\left(\frac{1}{1+2^kz}\right) + f(1)\log_2 z\end{align*}
converges absolutely at each $z$ in the half-plane $\Re(z)>0$ and defines a function which is holomorphic in that region. For all $z \in \CC$ such that $\Re(z)>0$ we have $(\mathfrak{C}_sf)(z)=(\mathfrak{C}_sf)(2z)$, and there exists a constant $C_3>0$ depending only on $s$ such that 
\[\sup\left\{|(\mathfrak{C}_sf)(z)| \colon \Re(z)>0\right\} \leq C_3\|f\|_{\HT}.\]
\end{lemma}
\begin{proof}
Let $f \in \HT$ and let $0<m \leq 1 < M$.  By Lemma \ref{BbddX} there exists $C_1>0$ depending on $M$ and $s$ but not on $f$ such that the series
\[\sum_{k=0}^\infty \left(\frac{1}{(1+2^{-k}z)^{2s}}f\left(\frac{1}{1+2^{-k}z}\right)-f(1)\right)\] 
converges absolutely when $\Re(z)>0$ and is bounded in absolute value by $C_1\|f\|_{\HT}$ when $\Re(z)>0$ and $|z|\leq M$.
Similarly, by Lemma \ref{GbddH2} the series
\begin{equation}\label{gibbon}\sum_{k=1}^\infty \frac{1}{(1+2^kz)^{2s}}f\left(\frac{1}{1+2^kz}\right)\end{equation}
converges absolutely when $\Re(z)>0$, and there is a constant $C_2$ depending on $m, M$ and $s$ but not on $f$ such that the absolute value of \eqref{gibbon} is bounded by $C_2\|f\|_{\HT}$ when $\Re(z)>0$ and $m \leq |z| \leq M$.  Since $m$ and $M$ are arbitrarily it follows in particular that $\mathfrak{C}_sf$ is a well-defined holomorphic function in the half-plane $\Re(z)>0$, and the identity $(\mathfrak{C}_sf)(2z)=(\mathfrak{C}_sf)(z)$ follows simply by substituting the two different values into the definition of $\mathfrak{C}_sf$ and verifying that the results agree.

Let us now fix $\tilde{C}_1, \tilde{C}_2>0$ to be the particular values taken by the constants $C_1,C_2$ in the special case $m:=1$, $M:=2$. In view of the periodicity relation $(\mathfrak{C}_sf)(2z)\equiv (\mathfrak{C}_sf)(z)$  it is clear that
\[\sup\left\{|(\mathfrak{C}_sf)(z)| \colon \Re(z)>0\right\} = \sup\left\{|(\mathfrak{C}_sf)(z)| \colon \Re(z)>0 \text{ and }1 \leq |z| \leq 2\right\}\]
and the latter quantity is bounded by 
\[(\tilde{C}_1+\tilde{C}_2)\|f\|_{\HT}+ \sup\left\{|f(1)\log_2 z| \colon \Re(z)>0 \text{ and }1 \leq |z| \leq 2\right\}.\]
When $\Re(z)>0$ and $1 \leq |z| \leq 2$ we have
\[|\log z| \leq |\log |z||+|\arg z| \leq \log 2 + \frac{\pi}{2}\]
and therefore
\[\left|f(1)\log_2 z\right| \leq \left(1+\frac{\pi}{\log 4}\right) |f(1)| \leq 4|f(1)| \leq 4\|f\|_{\HT}.\]
It follows that $|(\mathfrak{C}_sf)(z)|$ is everywhere bounded by $(\tilde{C}_1+\tilde{C}_2+4)\|f\|_{\HT}$ as required.
\end{proof}
By combining the previous three lemmas we obtain the following result which underpins much of our analysis of the action of $\mathfrak{L}_{s,0}$ on $\XX$:
\begin{proposition}\label{GbddX}
For each $s \in \CC$ with $\Re(s)>\frac{2}{3}$ and each $f \in \HT$ the function $\mathfrak{G}_{s,0}f$ defined in Lemma \ref{GbddH2} belongs to $\XX$, and the function $\mathfrak{G}_{s,0} \colon \HT \to \XX$ thus defined is a bounded linear map. For each $f \in \HT$ there exists $g \in \HI$ such that $(\mathfrak{G}_{s,0}f)(z)=-f(1)\log_2z + g(z)$ for all $z \in \DD$.\end{proposition}
\begin{proof}
Define an operator $\mathfrak{A} \colon \HT \to \XX$ by setting $(\mathfrak{A}f)(z):=-f(1)\log_2 z$ for all $z \in \DD$ for every $f \in \HT$. Since $|f(1)| \leq \|f\|_{\HT}$ for all $f \in \HT$ by Lemma \ref{basic0} it is clear that $\mathfrak{A}$ is a bounded linear map from $\HT$ to $\XX$. Now define two more operators $\mathfrak{B}_s, \mathfrak{C}_s \colon \HT \to \XX$ by taking the function $f \in \HT$ to the functions $\mathfrak{B}_sf$ and $\mathfrak{C}_sf$ defined in Lemmas \ref{BbddX} and \ref{CbddX} respectively. It is clear from Lemmas \ref{BbddX} and \ref{CbddX}  that $\mathfrak{B}_s$ and $\mathfrak{C}_s$ are well-defined bounded linear maps from $\HT$ to $\XX$, and since clearly  $\mathfrak{G}_sf=\mathfrak{A}f+\mathfrak{B}_sf+\mathfrak{C}_sf$ for all $f \in \HT$ we conclude that $\mathfrak{G}_{s,0}:\HT \to \XX$ is a bounded linear map. To derive the expression  $(\mathfrak{G}_{s,0}f)(z)=-f(1)\log_2z + g(z)$ we simply define $g:=(\mathfrak{B}_s+\mathfrak{C}_s)f \in \HI$.
\end{proof}

\subsection{Boundedness of Brent's operator on $\XX$}

\begin{proposition}\label{DbddX}
Let $s \in \CC$ with $\Re(s)>\frac{2}{3}$. For each $f \in \XX$ the series
\[(\mathfrak{D}_{s,0}f)(z):=\sum_{k=1}^\infty \frac{1}{(z+2^k)^{2s}}f\left(\frac{z}{z+2^k}\right)\]
defines a function $\mathfrak{D}_{s,0}f \in \XX$, and the function $\mathfrak{D}_{s,0} \colon \XX \to \XX$ thus defined is a bounded linear map with spectral radius not greater than $\frac{1}{4^{\Re(s)}-1}$. If $f \in \XX$ satisfies $f(z) =\alpha \log_2 z + g(z)$ for all $z \in \DD$ where $\alpha \in \CC$ and $g \in \HI$, then there exists $\hat g \in \HI$ such that $(\mathfrak{D}_{s,0}f)(z)=\frac{\alpha}{4^s-1}\log_2 z + \hat g(z)$ for all $z \in \DD$.
\end{proposition}
\begin{proof}
Fix $s$ throughout the proof. We have seen in the proof of Proposition \ref{FirstEBrent} that $\mathfrak{D}_{s,0}$ acts on $\HT$, and since $\XX\subset\HT$ by Lemma \ref{basic0} it follows that for every $f \in \XX$ the above formula for $\mathfrak{D}_{s,0}f$ converges to a well-defined holomorphic function $\mathfrak{D}_{s,0}f \colon \DD \to \CC$.
We begin by proving the following claim: there exists a constant $K>0$ depending on $s$ such that for all $f \in \HI$  and $n \geq 1$ we have $\mathfrak{D}_{s,0}f \in \HI$ and
\begin{equation}\label{fsh}\left\|\mathfrak{D}_{s,0}^nf\right\|_{\HI} \leq \frac{K\|f\|_{\HI}}{\left(4^{\Re(s)}-1\right)^{n}}.\end{equation}
Let $\phi_k(z):=\frac{z}{z+2^k}$ for all $k \in \mathbb{N}$ and $z \in \DD$. Arguing in the same manner as in the proof of Proposition \ref{FirstEBrent}, for each $n \geq 1$ and $f \in \HI$ we may write
\[\left(\mathfrak{D}_{s,0}^nf\right)(z)=\sum_{k_1,\ldots,k_n=1}^\infty 2^{-s\sum_{i=1}^n k_i}\left((\phi_{k_n} \circ \cdots \circ  \phi_{k_1})'(z)\right)^s f\left(\left(\phi_{k_n}\circ \cdots \circ \phi_{k_1}\right)(z)\right)\]
for all $z \in \DD$, and for each choice of integers $k_1,\ldots,k_n \geq 1$ the inequality
\[\sup_{z \in \DD}\left|\left(2^{-\sum_{i=1}^n k_i}(\phi_{k_n} \circ \cdots \circ  \phi_{k_1})'(z)\right)^s \right| \leq \frac{e^{\pi|\Im(s)|}}{4^{(\sum_{i=1}^nk_i)\Re(s)}}\]
is satisfied. It follows easily that for all $z \in \DD$ we have 
\begin{align*}\left|\left(\mathfrak{D}_{s,0}^nf\right)(z)\right| &\leq \left(\sum_{k_1,\ldots,k_n=1}^\infty \frac{e^{\pi|\Im(s)|}}{4^{(\sum_{i=1}^nk_i)\Re(s)}}\right)\|f\|_{\HI}\\
&=e^{\pi|\Im(s)|}\left(\sum_{k=1}^\infty \frac{1}{4^{k\Re(s)}}\right)^n\|f\|_{\HI}\end{align*}
which implies the validity of \eqref{fsh}.

We next assert that the holomorphic function $h \colon \DD \to \CC$ defined by $h(z):=(\mathfrak{D}_{s,0}\log_2)(z)- \frac{1}{4^s-1}\log_2 z$ belongs to $\HI$. We begin by noting that for all $z \in \DD$,
\begin{align}\label{daisy}|h(z)|&=\left|(\mathfrak{D}_{s,0}\log_2)(z)-\frac{1}{4^s-1}\log_2 z\right| \\\nonumber
&=\left|\sum_{k=1}^\infty \left(\frac{\log_2 \left(z/(z+2^k)\right)}{(z+2^k)^{2s}}- \frac{\log_2 z}{4^{ks}}\right)\right| \\\nonumber
&\leq \left|\sum_{k=1}^\infty \frac{\log_2 z}{(z+2^k)^{2s}}- \frac{\log_2 z}{4^{ks}}\right|+\left|\sum_{k=1}^\infty \frac{\log_2\left(z+2^k\right)}{(z+2^k)^{2s}}\right|.  \end{align}
By Lemma \ref{basic4} there exists a constant $K>0$ depending on $s$ such that for each $k \geq 1$ and $z \in \DD$
\[\left|\frac{1}{(z+2^k)^{2s}} - \frac{1}{4^{ks}}\right| = \frac{1}{4^{k\Re(s)}} \left| \frac{1}{\left(1+2^{-k}z\right)^{2s}}-1\right| \leq \frac{K|z|}{2^{k(1+2\Re(s))}}.\]
Since
\[\sup_{z \in \DD}|z\log_2 z| \leq \sup_{z \in \DD}\left( \left|z\log_2|z|\right|+\frac{\pi|z|}{\log 4} \right)\leq 2+\frac{\pi}{\log 2}<7 \]
it follows that we may estimate 
\begin{equation}\label{adam}\sum_{k=1}^\infty \left|\frac{\log_2 z}{(z+2^k)^{2s}}- \frac{\log_2 z}{4^{ks}}\right| \leq \sum_{k=1}^\infty\frac{K|z\log_2z|}{2^{k(1+2\Re(s))}} \leq 7K\left(\sum_{k=1}^\infty \frac{1}{2^k}\right)\leq 7K\end{equation}
for every $z \in \DD$. On the other hand, to bound the second of the two sums we observe that
\begin{align}\label{madeleine}\sup_{z \in \DD} \left|\sum_{k=1}^\infty \frac{\log_2\left(z+2^k\right)}{(z+2^k)^2}\right| &\leq \sum_{k=1}^\infty\left(\sup_{z \in \DD} \left|\frac{1}{(z+2^k)^s}\right|\right)\left(\sup_{z \in \DD}\left|\log_2\left(z+2^k\right)\right|\right)\\\nonumber
&\leq \sum_{k=1}^\infty \left(\frac{e^{\pi|\Im(s)|}}{4^{k\Re(s)}}\right)\left(k+1 + \frac{\pi}{\log 4}\right) \\\nonumber
&\leq 5e^{\pi|\Im(s)|}\sum_{k=1}^\infty \frac{k}{4^{k\Re(s)}}=\frac{5e^{\pi|\Im(s)|}4^{\Re(s)}}{\left(4^{\Re(s)}-1\right)^2}.
\end{align}
By combining \eqref{daisy}, \eqref{adam} and \eqref{madeleine} we conclude that $h \in \HI$ as claimed.

We may now prove the results asserted in the statement of the proposition. If $f\in\XX$ satisfies $f(z)=\alpha\log_2z + g(z)$ for all $z \in \DD$ where $\alpha \in\CC$ and $g \in \HI$, then we have
\begin{equation}\label{iter}\left(\mathfrak{D}_{s,0}f\right)(z)=\frac{\alpha}{4^s-1}\log_2z + \alpha h(z) + \left(\mathfrak{D}_{s,0}g\right)(z)\end{equation}
for all $z \in \DD$, where $\alpha h + \mathfrak{D}_{s,0}g \in \HI$. This shows that $f$ has the form claimed in the statement of the proposition, and furthermore using \eqref{fsh}
\begin{align*}\left\|\mathfrak{D}_{s,0}f\right\|_{\XX} &\leq \left|\frac{\alpha}{4^s-1}\right|+ |\alpha|\cdot\|h\|_{\HI} + \left\|\mathfrak{D}_{s,0}g\right\|_{\HI}\\
&\leq \left(\frac{1}{4^{\Re(s)}-1} + \left\|h\right\|_{\HI}\right)|\alpha| + \frac{K}{4^{\Re(s)}-1}\|g\|_{\HI}\\
&\leq\left(\frac{K+1}{4^{\Re(s)}-1} + \|h\|_{\HI}\right)\|f\|_{\XX}\end{align*}
which shows that $\mathfrak{D}_{s,0}$ is a bounded linear operator on $\XX$. More generally, by iterating \eqref{iter} we find that for  each $n \geq 1$
\[\left(\mathfrak{D}_{s,0}^nf\right)(z)=\frac{\alpha}{\left(4^s-1\right)^n}\log_2 z+ \alpha \sum_{i=1}^n \frac{\left(\mathfrak{D}_{s,0}^{n-i}h\right)(z)}{(4^s-1)^{i-1}} + \left(\mathfrak{D}_{s,0}^ng\right)(z)\]
for all $z \in \DD$, and therefore using \eqref{fsh} again
\begin{align*}\left\|\mathfrak{D}_{s,0}^nf\right\|_{\XX} &\leq |\alpha|\left(\frac{1}{\left(4^{\Re(s)}-1\right)^n} + \frac{Kn\|h\|_{\HI}}{\left(4^{\Re(s)}-1\right)^{n-1}}\right) + \frac{K\|g\|_{\HI}}{\left(4^{\Re(s)}-1\right)^n}\\
&\leq \left(\frac{1+K+Kn\|h\|_{\HI}\left(4^{\Re(s)}-1\right)}{\left(4^{\Re(s)}-1\right)^n}\right)\|f\|_{\XX}.\end{align*}
Since $f$ is arbitrary it follows by Gelfand's formula that the spectral radius of $\mathfrak{D}_{s,0}$ acting on $\XX$ is not greater than $\frac{1}{4^{\Re(s)}-1}$. This completes the proof of the proposition.
\end{proof}

\begin{corollary}\label{eigs}
For each $s \in \CC$ with $\Re(s)>\frac{2}{3}$, $\mathfrak{L}_{s,0}$ is a bounded linear operator on $\XX$. If $\mathfrak{L}_{s,0}f = \lambda f$ for some $f \in \XX$ and complex number $\lambda \neq \frac{1}{4^s-1}$ then there exists $g \in \HI$ such that 
\[f(z)=-\frac{f(1)}{(\lambda -\frac{1}{4^s-1})}\log_2 z + g(z)\]
for all $z \in\DD$.
\end{corollary}
\begin{proof}
Since $\mathfrak{L}_{s,0}=\mathfrak{G}_{s,0}+\mathfrak{D}_{s,0}$ and $\mathfrak{G}_{s,0},\mathfrak{D}_{s,0} \in \mathcal{B}(\XX)$ by Propositions \ref{GbddX} and \ref{DbddX} it is clear that $\mathfrak{L}_{s,0} \in \mathcal{B}(\XX)$ as claimed. If $f \in \XX$ satisfies $\mathfrak{L}_{s,0}f=\lambda f$ and $f(z)=\alpha \log_2 z + g(z)$ for all $z \in \DD$ where $\alpha \in \CC$ and $g \in \HI$, then by Propositions \ref{GbddX} and \ref{DbddX} there exist $g_1,g_2 \in \HI$ such that
\[\lambda f(z) = (\mathfrak{G}_{s,0}f)(z)+(\mathfrak{D}_{s,0}f)(z)= -f(1)\log_2z + g_1(z) + \frac{\alpha}{4^s-1}\log_2 z + g_2(z)\]
for all $z \in \DD$. It follows that $\lambda\alpha = -f(1) + \frac{\alpha}{4^s-1}$ and since $\lambda \neq \frac{1}{4^s-1}$ this implies the result claimed.
\end{proof}

\subsection{Essential spectrum of Brent's operator on $\XX$}

The principle underlying the following proposition is similar to that behind a theorem of H. Hennion \cite{Henn93}. The author wishes to thank O. Butterley for describing to him some extensions of Hennion's argument.
\begin{proposition}\label{essspec}
The essential spectral radius of $\mathfrak{L}_{s,0}$ acting on $\XX$ is less than or equal to $\frac{\sqrt{2}}{4^{\Re(s)}-\sqrt{2}}$.
\end{proposition}
\begin{proof}
Let $B_{\HT}$ and $B_{\XX}$ denote the closed unit balls of $\HT$ and $\XX$ respectively, and note that $B_{\XX} \subseteq 2B_{\HT}$ by Lemma \ref{basic3}. Let $\varepsilon>0$ be small enough that $\frac{(1+\varepsilon)\sqrt{2}}{4^{\Re(s)}-\sqrt{2}}<1$. By Proposition \ref{GbddX} there exists a constant $K_1>0$ such that $\|\mathfrak{G}_{s,0}f\|_{\XX} \leq K_1\|f\|_{\HT}$ for all $f \in \HT$, so in particular for all $f \in \HT$ and $n \geq 1$
\begin{equation}\label{feeesh}\left\|\mathfrak{G}_{s,0}\mathfrak{L}_{s,0}^{n-1}f\right\|_{\XX} \leq K_1\left\|\mathfrak{L}_{s,0}^{n-1}f\right\|_{\HT}.\end{equation}
By Proposition \ref{DbddX} the spectral radius of $\mathfrak{D}_{s,0}$ acting on $\XX$ is not greater than $\frac{1}{4^{\Re(s)}-1}$, so there clearly exists $K_2>0$ such that 
\begin{equation}\label{feeeeesh}\left\|\mathfrak{D}_{s,0}^n\right\|_{\XX} \leq K_2\left(\frac{(1+\varepsilon)\sqrt{2}}{4^{\Re(s)}-\sqrt{2}}\right)^n\end{equation}
for every $n \geq 0$.

By Proposition \ref{FirstEBrent} the essential spectral radius of $\mathfrak{L}_{s,0}$ acting on $\HT$ is not greater than $\frac{\sqrt{2}}{4^{\Re(s)}-\sqrt{2}}$, so using Theorem \ref{Nbm} we may find a constant $K_3>0$ such that for every integer $n \geq 0$ the Hausdorff measure of noncompactness of $\mathfrak{L}_{s,0}^n$ acting on $\HT$ is strictly less than $K_3\left(\frac{(1+\varepsilon)\sqrt{2}}{4^{\Re(s)}-\sqrt{2}}\right)^n$. In particular, for each $n \geq 0$ there exist an integer $\ell_n \geq 1$ and a finite sequence $U^{n}_1,\ldots,U^{n}_{\ell_n}$ of subsets of $\HT$ such that $\mathfrak{L}_{s,0}^nB_{\HT} \subseteq \bigcup_{i=1}^{\ell_n}U^{n}_i$ and
\begin{equation}\label{feesh}\|f-g\|_{\HT} \leq K_3\left(\frac{(1+\varepsilon)\sqrt{2}}{4^{\Re(s)}-\sqrt{2}}\right)^n\end{equation}
whenever $f$ and $g$ both belong to the same set $U^{n}_i$. We will use these sets to construct a finite covering of $\mathfrak{L}_{s,0}^nB_{\XX}$ by sets of small diameter with respect to $\|\cdot\|_{\XX}$.

Fix $n \geq 1$ and let $\mathcal{I}\subset\mathbb{N}^n$ be the set of all $n$-tuples $(k_0,\ldots,k_{n-1})$ such that $1 \leq k_i \leq \ell_i$ for $i=0,\ldots,n-1$. For each $(k_0,\ldots,k_{n-1}) \in \mathcal{I}$ define
\[V_{(k_0,\ldots,k_{n-1})}:=\left\{f \in B_{\XX} \colon \frac{1}{2}\mathfrak{L}^{i}_{s,0}f \in U^{i}_{k_i}\text{ for all }i=0,\ldots,n-1\right\}.\]
We claim that this collection of sets forms a cover of $B_{\XX}$. To see this suppose that $f \in B_{\XX}$. For each $i$ in the range $0 \leq i \leq n-1$ we have $\frac{1}{2}\mathfrak{L}^{i}_{s,0}f \in \mathfrak{L}^{i}_{s,0}B_{\HT}$ since $f \in 2B_{\HT}$, and since the sets $U^{i}_k$ cover $\mathfrak{L}_{s,0}^{i}B_{\HT}$ there exists $k_i \in \{1,\ldots,\ell_i\}$ such that $\frac{1}{2}\mathfrak{L}^i_{s,0}f \in U^{i}_{k_i}$. Since $i$ is arbitrary it follows that $f$ belongs to at least one of the sets $V_{(k_0,\ldots,k_{n-1})}$. Since $f$ is arbitrary we conclude that
\begin{equation}\label{lissie}\bigcup_{(k_0,\ldots,k_{n-1}) \in \mathcal{I}} V_{(k_0,\ldots,k_{n-1})} = B_{\XX}\end{equation}
as claimed. Now let $\mathcal{J}$ denote the collection of all sets of the form $\mathfrak{L}^n_{s,0}V_{(k_0,\ldots,k_{n-1})}$ for $(k_0,\ldots,k_{n-1})\in\mathcal{I}$. In view of \eqref{lissie} it is clear that the union of the elements of $\mathcal{J}$ is equal to $\mathfrak{L}^n_{s,0}B_{\XX}$. Let us bound the diameters of the elements of $\mathcal{J}$.

Suppose that $f,g \in \mathfrak{L}^n_{s,0}V_{(k_0,\ldots,k_{n-1})} \in \mathcal{J}$. By definition there exist $\hat f, \hat g \in V_{(k_0,\ldots,k_{n-1})}$ such that $f=\mathfrak{L}_{s,0}^n\hat{f}$ and $g=\mathfrak{L}_{s,0}^n\hat{g}$, and we trivially have $\|\hat f - \hat g\|_{\XX} \leq 2$ since both functions belong to $B_{\XX}$. It follows from the definition of $V_{(k_0,\ldots,k_{n-1})}$ that for each $i=0,\ldots,n-1$ the functions $\frac{1}{2}\mathfrak{L}_{s,0}^{i}\hat f$ and $\frac{1}{2}\mathfrak{L}_{s,0}^{i}\hat g$ both belong to $U^i_{k_i}$, and therefore
\begin{equation}\label{feeeesh}\left\|\mathfrak{G}_{s,0}\mathfrak{L}_{s,0}^{i}(\hat f-\hat g)\right\|_{\XX} \leq K_1\left\|\mathfrak{L}_{s,0}^{i}(\hat f-\hat g)\right\|_{\HT} \leq 2K_1K_3\left(\frac{(1+\varepsilon)\sqrt{2}}{4^{\Re(s)}-\sqrt{2}}\right)^{i}\end{equation}
using \eqref{feesh} and \eqref{feeesh}. Now, the relation 
\begin{equation}\label{therel}\mathfrak{L}^m_{s,0} = \sum_{i=0}^{m-1} \mathfrak{D}^i_{s,0}\mathfrak{G}_{s,0}\mathfrak{L}_{s,0}^{m-i-1}+ \mathfrak{D}_{s,0}^m\end{equation}
is easily seen to hold for all integers $m \geq 1$, since the case $m=1$ is simply the identity $\mathfrak{L}_{s,0}=\mathfrak{G}_{s,0}+\mathfrak{D}_{s,0}$ and the same identity facilitates the induction step
\begin{align*}\mathfrak{L}_{s,0}^{m+1}&= \left(\sum_{i=0}^{m-1} \mathfrak{D}^i_{s,0}\mathfrak{G}_{s,0}\mathfrak{L}_{s,0}^{m-i-1}+ \mathfrak{D}_{s,0}^m\right)\mathfrak{L}_{s,0}\\
&= \sum_{i=0}^{m-1} \mathfrak{D}^i_{s,0}\mathfrak{G}_{s,0}\mathfrak{L}_{s,0}^{m-i}+ \mathfrak{D}_{s,0}^m\mathfrak{G}_{s,0}+\mathfrak{D}_{s,0}^{m+1}\\
&= \sum_{i=0}^{m} \mathfrak{D}^i_{s,0}\mathfrak{G}_{s,0}\mathfrak{L}_{s,0}^{m-i}+ \mathfrak{D}_{s,0}^{m+1}.\end{align*}
Using \eqref{therel} followed by \eqref{feeeesh} and \eqref{feeeeesh} we may write
\begin{align*} \left\|\mathfrak{L}_{s,0}^n(\hat f-\hat g)\right\|_{\XX}& \leq \sum_{i=0}^{n-1} \left\|\mathfrak{D}^i_{s,0}\mathfrak{G}_{s,0}\mathfrak{L}_{s,0}^{n-i-1}(\hat f -\hat g)\right\|_{\XX} + \left\|\mathfrak{D}_{s,0}^n(\hat f - \hat g)\right\|_{\XX}\\
&\leq 2K_1K_3\sum_{i=0}^{n-1}\left\|\mathfrak{D}_{s,0}^i\right\|_{\XX}\left(\frac{(1+\varepsilon)\sqrt{2}}{4^{\Re(s)}-\sqrt{2}}\right)^{n-i-1}  +2\left\|\mathfrak{D}_{s,0}^n\right\|_{\XX}\\
&\leq 2K_1K_2K_3n\left(\frac{(1+\varepsilon)\sqrt{2}}{4^{\Re(s)}-\sqrt{2}}\right)^{n-1}  +2K_2\left(\frac{(1+\varepsilon)\sqrt{2}}{4^{\Re(s)}-\sqrt{2}}\right)^{n}\\
&< \left(2K_1K_2K_3n + 2K_2\right)\left(\frac{(1+\varepsilon)\sqrt{2}}{4^{\Re(s)}-\sqrt{2}}\right)^{n-1}\end{align*}
and therefore
\[\left\|f-g\right\|_{\XX} \leq  2K_2\left(K_1K_3n + 1\right)\left(\frac{(1+\varepsilon)\sqrt{2}}{4^{\Re(s)}-\sqrt{2}}\right)^{n-1}\]
whenever $f, g \in \mathfrak{L}_{s,0}^nB_\XX$ belong to the same element of $\mathcal{J}$. 

We have shown that the collection $\mathcal{J}$ of subsets of $\XX$ forms a finite cover of $\mathfrak{L}_{s,0}^nB_{\XX}$ whose elements have diameter bounded by the quantity above. This last expression is therefore an upper bound for the Hausdorff measure of noncompactness of $\mathfrak{L}^n_{s,0}$ acting on $\XX$. Since $n$ is arbitrary we deduce using Theorem \ref{Nbm} that the essential spectral radius of $\mathfrak{L}_{s,0}$ acting on $\XX$ is less than or equal to $\frac{(1+\varepsilon)\sqrt{2}}{4^{\Re(s)}-\sqrt{2}}$, and since $\varepsilon$ is arbitrary the conclusion of the proposition follows.
\end{proof}

\subsection{Point spectrum of Brent's operator on $\XX$}

The result of Proposition \ref{essspec} renders it a straightforward undertaking to bound the spectral radius of $\mathfrak{L}_{s,0}$ as follows.
\begin{lemma}\label{dva}
Let $s \in \CC$ with $\Re(s) \geq1$. Then $\rho(\mathfrak{L}_{s,0}) \leq 1$, and if $\Re(s)>1$ then this inequality is strict.
\end{lemma}
\begin{proof}
Since the essential spectral radius of $\mathfrak{L}_{s,0}$ is strictly less than one it suffices to bound the moduli of the eigenvalues of $\mathfrak{L}_{s,0}$. To this end suppose that $\mathfrak{L}_s\xi_s=\lambda \xi_s$ for some $\lambda \in \CC$ and nonzero $\xi_s \in \XX$. We first consider the case in which $\Re(s)>1$. Since $\xi_s$ is holomorphic but is not the zero function we have $|\xi_s(x)|>0$ for all but countably many $x \in (0,1]$. In particular, for all but countably many $x \in (0,1]$ both of the quantities
\[\left|\xi_s\left(\frac{1}{1+2^kx}\right)\right|,\qquad \left|\xi_s\left(\frac{x}{x+2^k}\right)\right|\]
are nonzero for all $k \geq 1$. Using the inequalities $2\Re(s)>2$ and $0<\frac{1}{1+2^kx} <1$ it follows that for such an $x$
\begin{align*}|\lambda \xi_s(x)|&=\left|\sum_{k=1}^\infty \frac{1}{(1+2^kx)^{2s}}\xi_s\left(\frac{1}{1+2^kx}\right)+\frac{1}{(x+2^k)^{2s}}\xi_s\left(\frac{x}{x+2^k}\right)\right|\\
&\leq \sum_{k=1}^\infty \frac{1}{(1+2^kx)^{2\Re(s)}}\left|\xi_s\left(\frac{1}{1+2^kx}\right)\right|+\frac{1}{(x+2^k)^{2\Re(s)}}\left|\xi_s\left(\frac{x}{x+2^k}\right)\right|\\
&< \sum_{k=1}^\infty \frac{1}{(1+2^kx)^{2}}\left|\xi_s\left(\frac{1}{1+2^kx}\right)\right|+\frac{1}{(x+2^k)^{2}}\left|\xi_s\left(\frac{x}{x+2^k}\right)\right|\\
&=\left(\mathfrak{L}_{1,0}|\xi_s|\right)(x),\end{align*}
where $|\xi_s|$ is understood as an element of $L^1([0,1])$ and $\mathfrak{L}_{1,0}|\xi_s|$ is understood in the sense of Lemma \ref{wath}, since obviously $|\xi_s|\notin \XX$. By integration we deduce
\[|\lambda|\int_0^1|\xi_s(x)|dx = \int_0^1 |\lambda\xi_s(x)|dx <\int_0^1 \mathfrak{L}_{1,0}|\xi_s(x)|dx = \int_0^1 |\xi_s(x)|dx\]
using Lemma \ref{wath}, which implies that $|\lambda|<1$ as claimed. If instead $\Re(s)=1$ then a similar analysis shows that $|\lambda\xi_s(x)| \leq (\mathfrak{L}_{1,0}|\xi_s|)(x)$ for all $x \in (0,1]$ and by integration we deduce that $|\lambda| \leq 1$.
\end{proof}
To proceed further we will use a generalisation of the Kre\u{\i}n-Rutman theorem due to R. Nussbaum \cite{Nuss81}. Following the conventions of Nussbaum's article we shall say that a subset $\mathsf{K}$ of a real Banach space $\mathsf{X}$ is a \emph{cone} if it is closed and convex, satisfies $\lambda x \in \mathsf{K}$ for all $x \in \mathsf{K}$ and  $\lambda \geq 0$, and for every $x \in \mathsf{K} \setminus \{0\}$ we have $-x \notin \mathsf{K}$. 
\begin{theorem}[Nussbaum]\label{nutcracker}
Let $(\mathsf{X},\|\cdot\|)$ be a real Banach space, $\mathsf{K} \subset \mathsf{X}$ a cone, and $L \colon \mathsf{X} \to \mathsf{X}$ a bounded linear operator such that $L\mathsf{K} \subseteq \mathsf{K}$. Let $B_{\mathsf{K}}$ denote the intersection of the closed unit ball of $\mathsf{X}$ with the cone $\mathsf{K}$, and define the spectral radius of $L$ relative to $\mathsf{K}$ to be the quantity
\[\rho^{\mathsf{K}}(L):=\lim_{n \to \infty} \left(\sup\left\{\|L^nx\| \colon x \in B_{\mathsf{K}}\right\}\right)^{\frac{1}{n}}\]
and the essential spectral radius of $L$ relative to $\mathsf{K}$ to be the quantity
\[\rho^{\mathsf{K}}_{\mathrm{ess}}(L):=\lim_{n \to \infty} \left(\psi\left(L^nB_{\mathsf{K}}\right)\right)^{\frac{1}{n}},\]
where $\psi(Z)$ is the Kuratowski measure of noncompactness of the set $Z \subset \mathsf{X}$. If $\rho^{\mathsf{K}}_{\mathrm{ess}}(L)>\rho^{\mathsf{K}}(L)$ then there exists a nonzero function $u \in \mathsf{K}$ such that $Lu=\rho^{\mathsf{K}}(L)u$.
\end{theorem}
We use this theorem to obtain the following:
\begin{lemma}\label{ek}
There exists $\xi \in \XX$ such that $\mathfrak{L}_{1,0}\xi=\xi$, $\int_0^1\xi(x)dx=1$, and $\xi(x)>0$ for all $x \in (0,1]$. There exists $\chi \in \HI$ such that for all $z \in \DD$
\begin{equation}\label{eig2}\xi(z)=-\frac{3}{2}\xi(1)\log_2z + \chi(z).\end{equation}
\end{lemma}
\begin{proof}
Let $\XX_{\mathbb{R}}$ denote the real Banach space of functions $f \in \XX$ such that $f(x)$ is real for every $x \in (0,1]$, equipped with the same norm as $\XX$, and let $\mathsf{K}$ denote the set of all $f \in \XX_{\mathbb{R}}$ such that $f(x)\geq 0$ for all $x \in (0,1]$. It is straightforward to verify that $\mathsf{K}$ is a cone in $\XX_{\mathbb{R}}$ in the sense defined above and that $\mathfrak{L}_{1,0}\mathsf{K} \subseteq \mathsf{K}$. It is clear from Gelfand's formula that the quantity $\rho^{\mathsf{K}}(\mathfrak{L}_{1,0})$ is bounded  above by the spectral radius of the operator $\mathfrak{L}_{1,0}$ acting on $\XX$, and by Lemma \ref{dva} this in turn is bounded above by 1. Conversely, observe that the constant function $\mathbf{1}$ belongs to $B_{\mathsf{K}}$. For each $n \geq 1$ we may choose $\theta_n \in \CC$ and $g_n \in \HI$ such that $(\mathfrak{L}_{1,0}^n\mathbf{1})(z) = \theta_n\log_2z + g_n(z)$ for all $z \in \DD$, which by Lemma \ref{wath} implies
\[1=\int_0^1 \mathbf{1}(x)dx = \int_0^1\left(\mathfrak{L}_{1,0}^n\mathbf{1}\right)(x)dx = \frac{1}{\log 2}\theta_n + \int_0^1g_n(x)dx \leq  \frac{1}{\log 2}\|\mathfrak{L}_{1,0}^n\mathbf{1}\|_{\XX},\]
and since $n$ is arbitrary we deduce that $\rho^{\mathsf{K}}(\mathfrak{L}_{1,0}) \geq 1$. Lastly, it is obvious that for each $n \geq 1$ the quantity $\psi(\mathfrak{L}_{1,0}^nB_{\mathsf{K}})$ is not greater than the Kuratowski measure of noncompactness of the image under $\mathfrak{L}_{1,0}$ of the closed unit ball of $\XX$, and we know by Proposition \ref{essspec} that this quantity decreases to zero with exponential speed as $n \to \infty$. We conclude that $\rho^{\mathsf{K}}_{\mathrm{ess}}(\mathfrak{L}_{1,0})<\rho^{\mathsf{K}}(\mathfrak{L}_{1,0})=1$, and by Theorem \ref{nutcracker} it follows that there exists a nonzero function $\xi \in \mathsf{K}$ such that $\mathfrak{L}_{1,0}\xi=\xi$. It is clear that every nonzero element of $\mathsf{K}$ has positive integral along the interval $[0,1]$, so by multiplying $\xi$ by a real scalar if necessary we may without loss of generality suppose that $\int_0^1\xi(x)dx=1$. We note that Corollary \ref{eigs} immediately yields the validity of the formula \eqref{eig2}.

To complete the proof of the lemma we must show that $\xi(x)>0$ for every $x \in (0,1]$. For a contradiction suppose instead that $\xi(x_0)=0$ for some $x_0 \in (0,1]$. We therefore necessarily have $(\mathfrak{L}_{1,0}^2\xi)(x_0)=0$, and by positivity it follows that $(\mathfrak{G}_{1,0}^2\xi)(x_0)=0$. Thus
\begin{align*}0=(\mathfrak{G}_{1,0}^2\xi)(x_0) &= \sum_{k,\ell=1}^\infty \frac{1}{\left(1+\frac{2^k}{1+2^\ell x_0}\right)^2\left(1+2^\ell x_0\right)^2}\xi\left(\frac{1}{\left(1+\frac{2^k}{1+2^\ell x_0}\right)}\right)\\
&=\sum_{k,\ell=1}^\infty \left(\frac{1}{1+2^\ell x_0 + 2^k}\right)^2\xi\left(\frac{1+2^\ell x_0}{1+2^\ell x_0 + 2^k}\right) \geq 0\end{align*}
which implies that $\xi((1+2^\ell x_0)/(1+2^\ell x_0 + 2^k))=0$ for every $k, \ell \geq 1$. These points accumulate at $1 \in \DD$ as $\ell \to \infty$ and $k$ remains fixed, and since $\xi$ is holomorphic in $\DD$ it follows that $\xi$ must be identically zero. This contradicts the definition of $\xi$, and we conclude that $\xi(x)>0$ for all $x \in (0,1]$ as desired. 
\end{proof}
For the remainder of the article we let $\xi$ denote the function constructed in Lemma \ref{ek} above.
\begin{lemma}\label{tri}
Let $t \in \mathbb{R}$ and suppose that $\mathfrak{L}_{1+it,0}\xi_t=\lambda\xi_t$ for some nonzero function $\xi_t \in \XX$ and some $\lambda \in \CC$ such that $|\lambda|=1$. Then  $\lambda=1$, $t=0$, and $\xi_t$ is a scalar multiple of $\xi$.
\end{lemma}
\begin{proof}
By multiplying $\xi_t$ by a complex number of unit modulus if required, we may assume without loss of generality that $\xi_t(1)$ is real and nonnegative. By Corollary \ref{eigs} and Lemma \ref{ek} there exist $\chi_t,\chi \in \HI$ such that for all $z \in \DD$
\begin{equation}\label{eigenf}\xi_t(z)= -\left(\frac{4^{1+it}-1}{\lambda(4^{1+it}-1)-1}\right)\xi_t(1)\log_2 z + \chi_t(z)\end{equation}
and
\begin{equation}\label{forfuckssake}\xi(z)=-\frac{3}{2}\xi(1)\log_2z + \chi(z),\end{equation}
and it follows in particular that the quantity $\sup_{x \in (0,1]} |\xi_t(x)| \xi(x)^{-1}$ is finite. Multiplying $\xi_t$ by a positive real number if necessary we may assume that this supremum is equal to one. To prove the lemma we will show that under this hypothesis $\lambda=1$, $t=0$ and $\xi_t=\xi$.

Let us investigate the scalar factor which arises in \eqref{eigenf}. Since $|\lambda|=1$ we have
\begin{align*}\left|\frac{4^{1+it}-1}{\lambda(4^{1+it}-1)-1}\right|&=\left|\frac{1}{1-\frac{1}{\lambda(4^{1+it}-1)}}\right|\\
&=\left|\sum_{n=0}^\infty \frac{1}{(\lambda(4^{1+it}-1))^n}\right| \leq \sum_{n=0}^\infty \frac{1}{\left|4^{1+it}-1\right|^n}\leq \sum_{n=0}^\infty \frac{1}{3^n} = \frac{3}{2}.\end{align*}
If $4^{1+it} \neq 4$ then $|4^{1+it}-1|>3$ and therefore the second inequality above is strict. If $4^{1+it}=4$ then $|4^{1+it}-1|=3$, but if additionally $\lambda \neq 1$ then the first inequality must be strict since the terms inside the summation have different arguments and will partially cancel one another. We conclude that
\[\left|\frac{4^{1+it}-1}{\lambda(4^{1+it}-1)-1}\right| \leq \frac{3}{2}\]
with equality if and only if both $\lambda=1$ and $4^{it}=1$. It follows that 
\begin{align}\label{limmmit}\lim_{x \to 0} \frac{|\xi_t(x)|}{\xi(x)} &= \lim_{x \to 0} \left|\frac{-\frac{4^{1+it}-1}{\lambda(4^{1+it}-1)-1}\xi_t(1) \log_2 x + \chi_t(x)}{-\frac{3}{2}\xi(1)\log_2x + \chi(x)}\right|\\\nonumber
 &= \frac{2}{3}\left|\frac{4^{1+it}-1}{\lambda(4^{1+it}-1)-1}\right|\frac{\xi_t(1)}{\xi(1)} \leq \frac{\xi_t(1)}{\xi(1)}\leq 1\end{align}
and if the limit is equal to one then necessarily $\lambda=1$ and $4^{it}=1$.

Let us show that that the limit in \eqref{limmmit} must equal one. For a contradiction let us suppose otherwise. In this case the supremum of $|\xi_t(x)|\xi^{-1}(x)$ over $x \in (0,1]$ is necessarily attained at some point $x_0 \in (0,1]$. Since by hypothesis the limit \eqref{limmmit} is strictly less than one we necessarily have 
\[\left|\xi_t\left(\frac{1}{x_0+2^k}\right)\right| < \xi\left(\frac{1}{x_0+2^k}\right),\qquad\left|\xi_t\left(\frac{x_0}{1+2^kx_0}\right)\right| < \xi\left(\frac{x_0}{1+2^kx_0}\right)\]
for all sufficiently large integers $k$, and hence
\begin{align*}|\lambda \xi_t(x_0)|&=\left|\sum_{k=1}^\infty \frac{1}{(1+2^kx_0)^{2+2it}}\xi_t\left(\frac{1}{1+2^kx_0}\right)+\frac{1}{(x_0+2^k)^{2+2it}}\xi_t\left(\frac{x_0}{x_0+2^k}\right)\right|\\
&\leq \sum_{k=1}^\infty \frac{1}{(1+2^kx_0)^{2}}\left|\xi_t\left(\frac{1}{1+2^kx_0}\right)\right|+\frac{1}{(x_0+2^k)^{2}}\left|\xi_t\left(\frac{x_0}{x_0+2^k}\right)\right|\\
&< \sum_{k=1}^\infty \frac{1}{(1+2^kx_0)^{2}}\xi\left(\frac{1}{1+2^kx_0}\right)+\frac{1}{(x_0+2^k)^{2}}\xi\left(\frac{x_0}{x_0+2^k}\right)\\
&=(\mathfrak{L}_{1,0}\xi)(x_0)=\xi(x_0)=|\xi_t(x_0)|,\end{align*}
contradicting our hypothesis that $|\lambda|=1$. We conclude that the limit in \eqref{limmmit} is equal to one and hence in particular $4^{it}=1$, $\lambda=1$, and $\xi_t(1)=\xi(1)$. In view of the last two identities we have
\begin{align}\label{wukkkk}\xi_t(1)=\left|\left(\mathfrak{L}_{1+it}^2\xi_t\right)(1)\right|&=\Bigg|\sum_{k,\ell=1}^\infty \frac{2}{(1+2^k+2^\ell)^{2+2it}}\xi_t\left(\frac{1+2^\ell}{2^k+2^\ell+1}\right)\\\nonumber
&\qquad+\frac{2}{(1+2^k+2^{k+\ell})^{2+2it}}\xi_t\left(\frac{1}{1+2^k+2^{k+\ell}}\right)\Bigg|\\\nonumber
&\leq \sum_{k,\ell=1}^\infty \frac{2}{(1+2^k+2^\ell)^{2}}\left|\xi_t\left(\frac{1+2^\ell}{2^k+2^\ell+1}\right)\right|\\\nonumber
&\qquad+\frac{2}{(1+2^k+2^{k+\ell})^2}\left|\xi_t\left(\frac{1}{1+2^k+2^{k+\ell}}\right)\right|\\\nonumber
&\leq \sum_{k,\ell=1}^\infty \frac{2}{(1+2^k+2^\ell)^{2}}\xi\left(\frac{1+2^\ell}{2^k+2^\ell+1}\right)\\\nonumber
&\qquad+\frac{2}{(1+2^k+2^{k+\ell})^2}\xi\left(\frac{1}{1+2^k+2^{k+\ell}}\right)\\\nonumber
&=(\mathfrak{L}^2_{1,0}\xi)(1)=\xi(1)=\xi_t(1),\end{align}
where we have simplified the expression for $(\mathfrak{L}_{1+it}^2\xi_t)(1)$ by taking advantage of the fact that the four functions 
\[\frac{1}{\left(z+2^kz+2^{k+\ell}\right)^{2+2it}}\xi_t\left(\frac{z}{z+2^kz+2^{k+\ell}}\right),\]
\[\frac{1}{\left(1+2^\ell z + 2^k\right)^{2+2it}}\xi_t\left(\frac{1+2^\ell z}{1+2^\ell z + 2^k}\right),\]
\[\frac{1}{\left(z+2^\ell+2^kz\right)^{2+2it}}\xi_t\left(\frac{z+2^\ell}{z+2^\ell+2^kz}\right),\]
\[\frac{1}{\left(1+2^kz+2^{k+\ell}z\right)^{2+2it}}\xi_t\left(\frac{1}{1+2^k + 2^{k+\ell}z}\right)\]
which appear in the sum defining $(\mathfrak{L}_{1+it}^2\xi_t)(z)$ take only two distinct values when evaluated at $z=1$, and we have used a similar simplification for $(\mathfrak{L}_{1,0}^2\xi)(1)$. Since the first and final expressions in the chain of inequalities \eqref{wukkkk} are identical, the inequalities in between must necessarily be equations. For this to be possible the expressions
\[ \frac{2}{(1+2^k+2^\ell)^{2+2it}}\xi_t\left(\frac{1+2^\ell}{2^k+2^\ell+1}\right), \qquad \frac{2}{(1+2^k+2^{k+\ell})^{2+2it}}\xi_t\left(\frac{1}{1+2^k+2^{k+\ell}}\right)\]
must have the same argument as one another and must also have constant argument with respect to the choice of $k,\ell \geq 1$, since otherwise the first inequality in \eqref{wukkkk} would be strict due to partial cancellations between terms. Similarly, since $|\xi_t(x)|\leq \xi(x)$ for all $x \in (0,1]$ the identities
\[\left|\xi_t\left(\frac{1+2^\ell}{2^k+2^\ell+1}\right)\right| =\xi\left(\frac{1+2^\ell}{2^k+2^\ell+1}\right)\]
and
\[\left|\xi_t\left(\frac{1}{1+2^k+2^{k+\ell}}\right)\right|=\xi\left(\frac{1}{1+2^k+2^{k+\ell}}\right)\]
must hold for every $k, \ell \geq 1$ since otherwise the second inequality in \eqref{wukkkk} would be strict. It follows that we may choose $\theta \in \mathbb{R}$ such that for all $k,\ell \geq 1$
\[\left(\frac{1}{1+2^k+2^\ell}\right)^{2it}\left(\frac{\xi_t\left(\frac{1+2^\ell}{2^k+2^\ell+1}\right)}{\xi\left(\frac{1+2^\ell}{2^k+2^\ell+1}\right)}\right)=\frac{\frac{2}{(1+2^k+2^\ell)^{2+2it}}\xi_t\left(\frac{1+2^\ell}{2^k+2^\ell+1}\right)}{\frac{2}{(1+2^k+2^\ell)^{2}}\xi\left(\frac{1+2^\ell}{2^k+2^\ell+1}\right)}=e^{i\theta}.\]
Taking $k=1$ and recalling that $4^{it}=1$ it is clear that this implies
\[e^{i\theta}=\lim_{\ell \to \infty} \left(\frac{1}{3+2^\ell}\right)^{2it}\left(\frac{\xi_t\left(\frac{1+2^\ell}{2^\ell+3}\right)}{\xi\left(\frac{1+2^\ell}{2^\ell+3}\right)}\right)=\lim_{\ell \to \infty} \left(\frac{2^\ell}{3+2^\ell}\right)^{2it}\left(\frac{\xi_t\left(\frac{1+2^\ell}{2^\ell+3}\right)}{\xi\left(\frac{1+2^\ell}{2^\ell+3}\right)}\right)=\frac{\xi_t(1)}{\xi(1)}\]
and so in fact $e^{i\theta}=1$. If $r \in \mathbb{N}$ is any integer then taking instead $\ell \equiv k+ r$ we similarly find
\begin{align*}1=e^{i\theta}&=\lim_{k \to \infty} \left(\frac{1}{1+2^k(1+2^r)}\right)^{2it}\left(\frac{\xi_t\left(\frac{1+2^{k+r}}{2^k(2^r+1)+1}\right)}{\xi\left(\frac{1+2^{k+r}}{2^k(2^r+1)+1}\right)}\right)\\
&=\lim_{k \to \infty} \left(\frac{2^{k+r}}{1+2^k(1+2^r)}\right)^{2it}\left(\frac{\xi_t\left(\frac{1+2^{k+r}}{2^k(2^r+1)+1}\right)}{\xi\left(\frac{1+2^{k+r}}{2^k(2^r+1)+1}\right)}\right)=\left(\frac{2^r}{1+2^r}\right)^{2it}\frac{\xi_t\left(\frac{2^r}{1+2^r}\right)}{\xi\left(\frac{2^r}{1+2^r}\right)}.\end{align*}
Since the sequence $\left(\frac{2^r}{1+2^r}\right)_{r=1}^\infty$ takes values in $\DD$ and converges to a limit in $\DD$, the validity of the identity
\[\xi_t\left(\frac{2^r}{1+2^{r}}\right)=\left(\frac{2^r}{1+2^{r}}\right)^{-2it}\xi\left(\frac{2^r}{1+2^{r}}\right)\]
for all integers $r \geq 1$ implies that $\xi_t(z)=z^{-2it}\xi(z)$ for every $z \in \DD$. By \eqref{eigenf} and \eqref{forfuckssake} it follows that for real $x \in (0,1]$
\[\lim_{x \to 0}\frac{1}{x^{2it}} = \lim_{x \to 0} \frac{\xi_t(x)}{\xi(x)} = 1,\]
but this limit fails to exist when $t \neq 0$. We conclude that $t=0$ and therefore $\xi_t(z)=\xi(z)$ for all $z \in \DD$, which completes the proof of the lemma.
\end{proof}
Collating together the results of this subsection we obtain the following result which, in combination with Corollary \ref{eigs} and Proposition \ref{essspec}, completes the proof of Theorem \ref{SimBrenSpec}.
\begin{proposition}\label{itch}
The operator $\mathfrak{L}_{1,0} \in \mathcal{B}(\XX)$ has a simple eigenvalue at $1$ and has no other eigenvalues on the unit circle. There exists $\xi \in \XX$ such that $\mathfrak{L}_{1,0}\xi=\xi$, $\int_0^1\xi(x)dx=1$ and $\xi(x)>0$ for all $x \in (0,1]$. If $s \in \CC$ with $\Re(s) \geq 1$, then $\rho(\mathfrak{L}_{s,0})\leq 1$ with equality if and only if $s=1$.
\end{proposition}
\begin{proof}
All of these properties follow from the combination of Lemmas \ref{dva}, \ref{ek} and \ref{tri} except for the simplicity of the eigenvalue of $\mathfrak{L}_{1,0}$ at $1$. Specifically, while Lemmas \ref{ek} and \ref{tri} together show that $\ker (\mathfrak{L}_{1,0}-\mathrm{Id}_{\XX})$ is one-dimensional, it remains to show that $\ker (\mathfrak{L}_{1,0}-\mathrm{Id}_{\XX})^{n+1}$ is one-dimensional for every $n \geq 1$. Suppose for a contradiction that this is not the case, and let $n \geq 1$ be the smallest integer such that the dimension of $\ker (\mathfrak{L}_{1,0}-\mathrm{Id}_{\XX})^{n+1}$ exceeds one. If $\hat\xi \in \ker (\mathfrak{L}_{1,0}-\mathrm{Id}_{\XX})^{n+1}$ then necessarily $(\mathfrak{L}_{1,0}-\mathrm{Id}_{\XX})\hat\xi\in \ker(\mathfrak{L}_{1,0}-\mathrm{Id}_{\XX})^n$ and so we have $(\mathfrak{L}_{1,0}-\mathrm{Id}_{\XX})\hat\xi =\lambda \xi$ for some $\lambda \in \CC$ since $\ker (\mathfrak{L}_{1,0}-\mathrm{Id}_{\XX})^{n}$ is one-dimensional and contains $\xi$. However, using Lemma \ref{wath} we may calculate
\begin{align*}\lambda = \lambda \left(\int_0^1\xi(x)dx\right)&=\int_0^1\left(\left(\mathfrak{L}_{1,0}-\mathrm{Id}_{\XX}\right)\hat\xi\right)(x)dx\\
&=\int_0^1\left(\mathfrak{L}_{1,0}\hat\xi\right)(x)dx-\int_0^1\hat\xi(x)dx\\
&= \int_0^1\hat\xi(x)dx-\int_0^1\hat\xi(x)dx=0\end{align*}
so that in fact $(\mathfrak{L}_{1,0}-\mathrm{Id}_{\XX})\hat\xi =0$. We conclude that $\hat\xi \in \ker (\mathfrak{L}_{1,0}-\mathrm{Id}_{\XX})$, and since $\hat\xi$ was arbitrary it follows that $\ker (\mathfrak{L}_{1,0}-\mathrm{Id}_{\XX})^{n+1}= \ker (\mathfrak{L}_{1,0}-\mathrm{Id}_{\XX})$, contradicting the hypothesis that $\dim \ker (\mathfrak{L}_{1,0}-\mathrm{Id}_{\XX})^{n+1}>1$. The proof is complete.
\end{proof}

\subsection{Proof of Proposition \ref{bcnj}}\label{brentconj}
\begin{proof}
We assert that $F_n(x)=\int_0^x\left(\mathfrak{L}_{1,0}^n\mathbf{1}\right)(t)dt$ for all $x \in (0,1]$ and $n \geq 0$, which we will prove by induction on $n$. The case $n=0$ is clearly trivial. To prove the induction step, suppose that $F_n(x)=\int_0^x\left(\mathfrak{L}_{1,0}^n\mathbf{1}\right)(t)dt$ for all $x \in (0,1]$ and some integer $n \geq 0$ and note that for each $x \in (0,1]$ we may write $\int_0^x\left(\mathfrak{L}_{1,0}^{n+1}\mathbf{1}\right)(t)dt$ as
\[\int_0^x \sum_{k=1}^\infty\left(\frac{1}{\left(1+2^kt\right)^2}\left(\mathfrak{L}_{1,0}^{n}\mathbf{1}\right)\left(\frac{1}{1+2^kt}\right)+\frac{1}{\left(t+2^k\right)^2}\left(\mathfrak{L}_{1,0}^{n}\mathbf{1}\right)\left(\frac{t}{t+2^k}\right)\right)dt.\]
Since for each $x \in (0,1]$
\begin{align*}\int_0^x \frac{1}{(t+2^k)^2}\left(\mathfrak{L}_{1,0}^{n}\mathbf{1}\right)\left(\frac{t}{t+2^k}\right)dt &= \frac{1}{2^k}\int_{\frac{x}{x+2^k}}^1\left(\mathfrak{L}_{1,0}^{n}\mathbf{1}\right)(u)du \\
&= \frac{1}{2^k}\left(1-\int_0^{\frac{x}{x+2^k}}\left(\mathfrak{L}_{1,0}^{n}\mathbf{1}\right)(u)du\right)\end{align*}
and
\[\int_0^x \frac{1}{(1+2^kt)^2}\left(\mathfrak{L}_{1,0}^{n}\mathbf{1}\right)\left(\frac{1}{1+2^kt}\right)dt =\frac{1}{2^{k}}\int_0^{\frac{1}{1+2^kx}} \left(\mathfrak{L}_{1,0}^{n}\mathbf{1}\right)(v)dv\]
using the change of variable $u=t/(t+2^k)$ and $v=1/(1+2^kt)$ respectively, we have 
\[\int_0^x\left(\mathfrak{L}_{1,0}^{n+1}\mathbf{1}\right)(t)dt = 1+\sum_{k=1}^\infty \frac{1}{2^k}\left(F_n\left(\frac{x}{x+2^k}\right)-F_n\left(\frac{1}{1+2^kx}\right)\right)\]
for all $x \in (0,1]$ as required to complete the induction step.

By Theorem \ref{SimBrenSpec}, $1$ is an isolated point of the spectrum of $\mathfrak{L}_{1,0}$ which does not belong to the essential spectrum and is a simple eigenvalue in the sense of Proposition \ref{splittt}, and the remainder of the spectrum of $\mathfrak{L}_{1,0}$ acting on $\XX$ lies inside a disc about the origin of radius strictly smaller than $1$. It follows from Proposition \ref{splittt} that there exist $P,N \in \mathcal{B}(\XX)$ such that $\mathfrak{L}_{1,0}=P+N$, $PN=NP=0$, $P\mathfrak{L}_{1,0}=\mathfrak{L}_{1,0}P$, $P^2=P$ and $\rho(N)<1$. For each $n \geq 1$ we therefore have \begin{equation}\label{wang'a}\mathfrak{L}^n_{1,0}\mathbf{1}=P\mathbf{1}+N^n\mathbf{1}.\end{equation}
As a particular consequence $\lim_{n \to \infty} \mathfrak{L}_{1,0}^n\mathbf{1}=P\mathbf{1}$. Since $\mathfrak{X}$ embeds continuously in $L^1([0,1])$,
\[\mathfrak{L}_{1,0}P\mathbf{1}=\mathfrak{L}_{1,0}\left(\lim_{n \to \infty}\mathfrak{L}_{1,0}^{n}P\mathbf{1}\right)=P\left(\lim_{n \to \infty}\mathfrak{L}_{1,0}^{n+1}\mathbf{1}\right)=P^2\mathbf{1}=P\mathbf{1}\]
which by Theorem \ref{SimBrenSpec} implies that $P\mathbf{1}$ is a scalar multiple of $\xi$. On the other hand $\int_0^1\left(\mathfrak{L}_{1,0}^n\mathbf{1}\right)(x)dx=\int_0^1\mathbf{1}(x)dx=1$ for every $n \geq 1$ and therefore $\int_0^1 \left(P\mathbf{1}\right)(x)dx=1$, and we conclude that $P\mathbf{1}=\xi$.

 For each $n \geq 1$ let us write $(\mathfrak{L}^n_{1,0}\mathbf{1})(z)=\kappa_n\log_2z + g_n(z)$ for all $z \in \DD$ where $\kappa_n \in \CC$ and $g_n \in \HI$. Since $\rho(N)<1$ there exist $K>0$ and $\theta \in (\rho(N),1)$ such that $\|N^n\mathbf{1}\|_{\XX}\leq \|N^n\|_{\XX}\leq K\theta^n$ for every $n \geq 1$. By \eqref{wang'a} we have $\mathfrak{L}_{1,0}^n\mathbf{1}=\xi+N^n\mathbf{1}$ for all $n \geq 1$. Now, for $x \in (0,\frac{1}{2}]$ and $n \geq 1$
\begin{align*}\left|\int_0^x\left(\mathfrak{L}_{1,0}^n\mathbf{1}\right)(t)dt-\int_0^x \xi(t)dt\right| &=\left|\int_0^x \left(N^n\mathbf{1}\right)(t)dt\right|\\
&= \left|\kappa_n\int_0^x\log_2 t dt + \int_0^xg_n(t)dt\right|\\
&\leq \frac{1}{\log 2}|\kappa_nx\log x| +x\|g_n\|_{\HI}\\
&\leq \frac{1}{\log 2}\|N^n\mathbf{1}\|_{\XX}|x\log x|\\
& \leq 2K\theta^n|x\log x|\end{align*}
since in this interval $x \leq (\log 2)^{-1}|x\log x|$,  and for $x \in (\frac{1}{2},1]$ and $n \geq 1$
\begin{align*}\left|\int_0^x\left(\mathfrak{L}_{1,0}^n\mathbf{1}\right)(t)dt-\int_0^x \xi(t)dt\right| &=\left|\int_x^1\left(\mathfrak{L}_{1,0}^n\mathbf{1}\right)(t)dt-\int_x^1 \xi(t)dt\right|\\
&= \left|\int_x^1 \left(N^n\mathbf{1}\right)(t)dt\right|\\
&= \left|\kappa_n\int_x^1\log_2 t dt + \int_x^1g_n(t)dt\right|\\
&\leq (1-x)(|\kappa_n|+|g_n|_\infty)\\
&= (1-x)\|N^n\mathbf{1}\|_{\XX}\\
&\leq 2K\theta^n|x\log x|\end{align*}
since in this interval $1-x \leq 2|x\log x|$ for all $x$. Defining $F_\infty(x):=\int_0^x\xi(t)dt$ completes the proof of the proposition.
\end{proof}

\section{Conclusion of the proof of Theorem \ref{MainBrent}}\label{sixxxxxxxx}
In this short section we derive clauses (b) to (d) Theorem \ref{MainBrent} from the corresponding parts of Theorem \ref{SimBrenSpec} and proceed to prove Theorem \ref{MainBrent}(e). These actions finally complete the proof of Theorem \ref{MainBrent}.
\begin{lemma}\label{spech2a}
Let $n \geq 1$ and $\lambda,s \in \CC$ with $\Re(s)>\frac{2}{3}$ and $|\lambda|>\frac{\sqrt{2}}{4^{\Re(s)}-\sqrt{2}}$, and let $\eta \colon \DD \to \CC$ be holomorphic. Then $\eta \in \HT$ and $(\mathfrak{L}_{s,0}-\lambda\mathrm{Id}_{\HT})^n\eta=0$ if and only if $\eta \in \XX$ and $(\mathfrak{L}_{s,0}-\lambda\mathrm{Id}_{\XX})^n\eta=0$.
\end{lemma}
\begin{proof}
By Lemma \ref{basic3} every element of $\XX$ belongs to $\HT$, so the `if' part of the lemma is trivial. To prove the converse direction we must therefore prove that
\[\dim \ker (\mathfrak{L}_{s,0}-\lambda\mathrm{Id}_{\HT})^n \leq \dim \ker (\mathfrak{L}_{s,0}-\lambda\mathrm{Id}_{\XX})^n.\]
Since $|\lambda|$ exceeds the essential spectral radius of $\mathfrak{L}_{s,0}$ acting on $\HT$ the operator $\mathfrak{L}_{1,0}-\lambda\mathrm{Id}_{\HT}$ is Fredholm of index zero. Let $d \geq 1$ be the dimension of the subspace $\ker\left(\mathfrak{L}_{1,0}-\lambda\mathrm{Id}_{\HT}\right)^n$ of $\HT$. Since every power of a Fredholm operator of index zero is also Fredholm of index zero, $d$ is precisely the codimension of the image of $(\mathfrak{L}_{1,0}-\lambda\mathrm{Id}_{\HT})^n$, which in turn is equal to the dimension of the kernel of the adjoint operator $\left((\mathfrak{L}_{1,0}-\lambda\mathrm{Id}_{\HT})^n\right)^*$ acting on $\HT^*$. If $\ell \colon \HT \to \CC$ is a nonzero element of this kernel then by definition
\[\ell\left(\left(\mathfrak{L}_{1,0}-\lambda\mathrm{Id}_{\HT}\right)^nf\right)=\ell\left(\sum_{i=0}^n \left(-\lambda\right)^{i+1}\left(\begin{array}{c}n\\i\end{array}\right)\mathfrak{L}^i_{1,0}f \right)=0\]
and $|\ell(f)| \leq C_\ell\|f\|_{\HT}$ for every $f \in \HT$, where $C_\ell$ is a constant depending on $\ell$.
It follows from Lemma \ref{basic3} that for every $f \in \XX$ the quantity $\ell(f)$ is well-defined and satisfies $|\ell(f)|\leq C_\ell\|f\|_{\HT} \leq 2C_\ell\|f\|_{\XX}$, so $\ell$ belongs to $\XX^*$ and therefore
\[\ell\left(\left(\mathfrak{L}_{1,0}-\lambda\mathrm{Id}_{\XX}\right)^nf\right)=\ell\left(\sum_{i=0}^n \left(-\lambda\right)^{i+1}\left(\begin{array}{c}n\\i\end{array}\right)\mathfrak{L}^i_{1,0}f \right)=0\]
for every $f \in \XX^*$. Since $\XX$ contains $\HI$, and $\HI$ is dense in $\HT$, $\ell$ cannot be the zero element of $\XX^*$, and we conclude that $\ell$ is a nonzero element of $\ker \left(\left(\mathfrak{L}_{1,0}-\lambda\mathrm{Id}_{\XX}\right)^n\right)^*$. Since $\ell$ is arbitrary it follows that this kernel has dimension at least $d$. Since $|\lambda|$ exceeds the essential spectral radius of $\mathfrak{L}_{1,0}$ acting on $\XX$ the operator $\left(\mathfrak{L}_{1,0}-\lambda\mathrm{Id}_{\XX}\right)^n$ is also Fredholm of index zero and the image of $\left(\mathfrak{L}_{1,0}-\lambda\mathrm{Id}_{\XX}\right)^n$ is closed. The codimension of this image equals the dimension of $\ker \left(\left(\mathfrak{L}_{1,0}-\lambda\mathrm{Id}_{\XX}\right)^n\right)^*$ and hence is also at least $d$. By the Fredholm property of $\left(\mathfrak{L}_{1,0}-\lambda\mathrm{Id}_{\XX}\right)^n$ it follows that $\ker \left(\mathfrak{L}_{1,0}-\lambda\mathrm{Id}_{\XX}\right)^n$ has dimension at least $d$, and this is proves the lemma. 
\end{proof}
By Proposition \ref{FirstEBrent} the essential spectral radius of $\mathfrak{L}_{s,0}$ acting on $\HT$ is bounded by $\frac{\sqrt{2}}{4^{\Re(s)}-\sqrt{2}}$, so every point of the spectrum of $\mathfrak{L}_{s,0}$ with modulus greater than that quantity is an eigenvalue. The combination of Theorem \ref{SimBrenSpec} and Lemma \ref{spech2a} immediately yields: 
\begin{corollary}\label{spex}
The operator $\mathfrak{L}_{1,0} \in \mathcal{B}(\HT)$ has a simple eigenvalue at $1$ and has no other eigenvalues on the unit circle. There exists $\xi \in \HT$ such that $\mathfrak{L}_{1,0}\xi=\xi$, $\int_0^1\xi(x)dx=1$ and $\xi(x)>0$ for all $x \in (0,1]$. If $\lambda \in \CC$, $\mathfrak{L}_{s,0}\hat\xi=\lambda\hat\xi \in \HT$ and $|\lambda| > \frac{\sqrt{2}}{4^{\Re(s)}-\sqrt{2}}$ then there exists $\hat\chi \in \HI$ such that
\[\hat\xi(z)=-\frac{\hat\xi(1)}{\lambda-\frac{1}{4^s-1}}\log_2z + \hat\chi(z)\]
for all $z \in \DD$. If $s \in \CC$ with $\Re(s) \geq 1$, then $\rho(\mathfrak{L}_{s,0})\leq 1$ with equality if and only if $s=1$.
\end{corollary}
Together with the following proposition the above completes the proof of Theorem \ref{MainBrent}(b)--(d).
\begin{proposition}
Let $\mathfrak{L}_{s,\omega}\eta=\lambda\eta$ where $(s,\omega) \in \mathcal{U}$, $\eta \in \HT$ and $\lambda \neq 0$. Then $\eta$ admits an analytic continuation to the right half-plane $\Re(z)>0$. 
\end{proposition}
\begin{proof}
Let $M>1$ and define $K_M:=\sup\{|\eta(z)| \colon |z-1|\leq M/\sqrt{M^2+1}\}$. If $k \geq 1$, $\Re(z)>0$ and $1/M <|z|<M$ then by Lemma \ref{basic1}
\[\left|\frac{1}{1+2^kz}-1\right| < \frac{M}{\sqrt{M^2+1}}\]
and since $\Re(1/z)>0$ and $\frac{1}{M}<\left|\frac{1}{z}\right|<M$
\[\left|\frac{z}{z+2^k}-1\right| =\left|\frac{1}{1+\frac{2^k}{z}}-1\right| < \frac{M}{\sqrt{M^2+1}}\]
so that $\eta(1/(1+2^kz))$ and $\eta(1/(z+2^k))$ are both well-defined and are bounded in modulus by $K_M$. It follows that for each $k \geq 1$ the quantity
\[\left|\frac{e^{\omega c(1,k)}}{(1+2^kz)^{2s}}\eta\left(\frac{1}{1+2^kz}\right)+\frac{e^{\omega c(2,k)}}{(z+2^k)^{2s}}\eta\left(\frac{z}{z+2^k}\right)\right|\]
is bounded by
\begin{align*}K_M \left(\frac{e^{\pi|\Im(s)|+\Re(\omega)c(1,k)}}{|1+2^kz|^{2\Re(s)}} + \frac{e^{\pi|\Im(s)|+\Re(\omega)c(2,k)}}{|z+2^k|^{2\Re(s)}}\right)&\leq \frac{K_Me^{\pi|\Im(s)| }2^{\frac{k}{6}}\left(M^{2\Re(s)}+1\right)}{4^{k\Re(s)}}\\
&\leq \frac{K_Me^{\pi|\Im(s)| }\left(M^{2\Re(s)}+1\right)}{2^{k-1}}\end{align*}
for all $z$ such that $\Re(z)>0$ and $\frac{1}{M}<|z|<M$, where we have used the bounds $|\omega|c(i,k)<\frac{k}{6}\log 2$ and $4^{k\Re(s)}>2^{\frac{4}{3}k}$ which follow from Proposition \ref{FirstEBrent} together with the elementary bounds $|1+2^kz|>2^k/M$ and $|z+2^k|>2^k$. The formula
\[\hat\eta(z):=\frac{1}{\lambda}\sum_{k=1}^\infty\left(\frac{e^{\omega c(1,k)}}{\left(1+2^kz\right)^{2s}}\eta\left(\frac{1}{1+2^kz}\right)+\frac{e^{\omega c(2,k)}}
{\left(z+2^k\right)^{2s}}\eta\left(\frac{z}{z+2^k}\right)\right) \]
therefore defines a holomorphic function in the region $\Re(z)>0$, $\frac{1}{M}<|z|<M$. Since $M$ is arbitrary it follows that $\hat\eta$ is holomorphic on the entire right half-plane, and since by definition $\hat\eta(z)=\lambda^{-1}\left(\mathfrak{L}_{s,\omega}\eta\right)(z)=\eta(z)$ for $z \in \DD$ we conclude that $\hat\eta$ is the claimed analytic continuation of $\eta$.
\end{proof}
Without further ado we may complete the proof of Theorem \ref{MainBrent}(e) and (f) in the following two propositions.
\begin{proposition}\label{spax}
There exist an open set $\mathcal{V}\subset \CC^2$ containing the point $(1,0)$, holomorphic functions $(s,\omega) \mapsto \mathcal{P}_{s,\omega}$ and $(s,\omega) \mapsto \mathcal{N}_{s,\omega}$ defined for $(s,\omega) \in \mathcal{V}$ and taking values in $\mathcal{B}(\HT)$, and a holomorphic function $\lambda \colon \mathcal{V} \to \CC$ such that for all $(s,\omega) \in \mathcal{V}$:
\begin{enumerate}
\item
The identity $\mathfrak{L}_{s,\omega} = \lambda(s,\omega)\mathcal{P}_{s,\omega}+\mathcal{N}_{s,\omega}$ holds in $\mathcal{B}(\HT)$.
\item
We have $\mathcal{P}_{s,\omega}\mathcal{N}_{s,\omega}=\mathcal{N}_{s,\omega}\mathcal{P}_{s,\omega}=0$.
\item
The spectral radius of $\mathcal{N}_{s,\omega}$ is strictly less than one.
\item
The operator $\mathcal{P}_{s,\omega}$ is a projection with rank equal to one.
\end{enumerate}
The functions $\lambda$ and $\mathcal{P}$ also satisfy $\lambda(1,0)=1$ and $\mathcal{P}_{1,0}f = \left(\int_0^1f(x)dx\right)\xi$ for all $f \in \HT$.
\end{proposition}
\begin{proof}
By Corollary \ref{spex}, 1 is an isolated point of the spectrum of $\mathfrak{L}_{1,0}$, so we may choose a counterclockwise-oriented closed curve $\Gamma$ in $\CC$ which encloses $1$ but does not enclose any other points of the spectrum of $\mathfrak{L}_{1,0}$. By Proposition \ref{FirstEBrent} the essential spectral radius of $\mathfrak{L}_{1,0}$ is less than one and so the operator $\mathfrak{L}_{1,0}-\mathrm{Id}_{\HT}$ is Fredholm of index zero, and it follows from Corollary \ref{spex} that the remainder of the spectrum of $\mathfrak{L}_{1,0}$ lies in a disc about the origin of radius strictly less than one. By \cite[Theorem IV.3.16]{Kato} there exists an open ball $\mathcal{V}$ containing $(1,0)$ such that for all $(s,\omega) \in \mathcal{V}$, the spectrum of $\mathfrak{L}_{s,\omega}$ does not intersect $\Gamma$. For all $(s,\omega) \in \mathcal{V}$ let us define
\[\mathcal{P}_{s,\omega}:=\frac{1}{2\pi i}\int_\Gamma \left(z\mathfrak{L}_{s,\omega}-\mathrm{Id}_{\HT}\right)^{-1}dz\]
which is a projection by \cite[Theorem III.6.17]{Kato} and clearly commutes with $\mathfrak{L}_{s,\omega}$. 
Since $\left(z\mathfrak{L}_{s,\omega}-\mathrm{Id}_{\HT}\right)^{-1}$ depends holomorphically on $(s,\omega)$ within its domain of definition for each fixed $z$, it is easily seen that $\mathcal{P}_{s,\omega}$ depends holomorphically on $(s,\omega)$. Define $\mathcal{N}_{s,\omega}:=\mathfrak{L}_{s,\omega}-\mathfrak{L}_{s,\omega}\mathcal{P}_{s,\omega}$ for each $(s,\omega)$; this operator clearly also depends holomorphically on $(s,\omega)$. The identity $\mathcal{N}_{s,\omega}\mathcal{P}_{s,\omega}=\mathcal{P}_{s,\omega}\mathcal{N}_{s,\omega}$ follows from the definitions and the fact that $\mathcal{P}_{s,\omega}$ is a projection. By Proposition \ref{splittt} the rank of $\mathcal{P}_{1,0}$ is $1$ and we have $\mathfrak{L}_{1,0}=\mathcal{P}_{1,0}+\mathcal{N}_{1,0}$ and $\rho(\mathcal{N}_{1,0})<1$.

 By \cite[Theorem IV.3.16]{Kato} the rank of $\mathcal{P}_{s,\omega}$ is equal to that of $\mathcal{P}_{1,0}$ for all $(s,\omega) \in \mathcal{V}$, and since $\mathfrak{L}_{s,\omega}$ clearly commutes with $\mathcal{P}_{s,\omega}$ the image of $\mathcal{P}_{s,\omega}$ is invariant under $\mathfrak{L}_{s,\omega}$ and hence is a one-dimensional eigenspace. Let $\lambda(s,\omega)$ denote the corresponding eigenvalue; since $\mathfrak{L}_{1,0}=\mathcal{P}_{1,0}+\mathcal{N}_{1,0}$ we have $\lambda(1,0)=1$. By Corollary \ref{spex} it follows that the image of $\mathcal{P}_{1,0}$ is the one-dimensional subspace of $\HT$ spanned by $\xi$. 

Let $f \in \HT$. For each $n \geq 1$ we have $\mathfrak{L}_{1,0}^nf = \mathcal{P}_{1,0}f+\mathcal{N}_{1,0}^nf$ and therefore $\lim_{n \to \infty} \mathfrak{L}^n_{1,0}f=\mathcal{P}_{1,0}f$. By Lemma \ref{basic0} $\HT$ embeds continuously in $L^1([0,1])$ and therefore
\[\int_0^1\left(\mathcal{P}_{1,0}f\right)(x)dx=\lim_{n \to \infty} \int_0^1 \left(\mathfrak{L}_{1,0}^nf\right)(x)dx =\int_0^1 f(x)dx\]
using Lemma \ref{wath}. Since $\mathcal{P}_{1,0}f$ is proportional to $\xi$ and $\int_0^1\xi(x)dx=1$ it follows that $\mathcal{P}_{1,0}f=\left(\int_0^1f(x)dx\right)\xi$ as claimed.

Since $\mathcal{N}_{s,\omega}$ depends continuously on $(s,\omega)$ its spectral radius $\rho(\mathcal{N}_{s,\omega})$ is upper semicontinuous with respect to those variables, so by replacing $\mathcal{V}$ with a smaller neighbourhood of $(1,0)$ if required we may assume without loss of generality that $\rho(\mathcal{N}_{s,\omega})<1$ for all $(s,\omega) \in\mathcal{V}$. Now define $\xi_{s,\omega}:=\mathcal{P}_{s,\omega}\xi$ for every $(s,\omega) \in \mathcal{V}$, and note that $\mathfrak{L}_{s,\omega}\xi_{s,\omega}=\lambda(s,\omega)\xi_{s,\omega}$ for every $(s,\omega) \in \mathcal{V}$. By Corollary \ref{spex} we have $\xi(1)>0$, and by shrinking $\mathcal{V}$ further if necessary we may assume that $\xi_{s,\omega}(1) \neq 0$ for every $(s,\omega) \in \mathcal{V}$. We therefore have $\lambda(s,\omega) = \xi_{s,\omega}(1)^{-1}\left(\mathfrak{L}_{s,\omega}\xi_{s,\omega}\right)(1)$ for every $(s,\omega) \in \mathcal{V}$, and this expression is holomorphic since the linear functional on $\HT$ defined by $f \mapsto f(1)$ is continuous by Lemma \ref{basic0}.
\end{proof}
\begin{proposition}
The operator $\mathfrak{L}_{1,0}$ acts continuously on $L^1([0,1])$ with norm $1$. If $f \in L^1([0,1])$ then $\lim_{n \to \infty}\mathfrak{L}^n_{1,0}f=(\int_0^1f(x)dx)\xi$ and $\int_0^1(\mathfrak{L}_{1,0}f)(x)dx=\int_0^1f(x)dx$.
In particular, if $f \in L^1([0,1])$ and $\mathfrak{L}_{1,0}f=f$ then $f$ is proportional to $\xi$.
\end{proposition}
\begin{proof}
It was shown in Lemma \ref{wath} that if $f \in L^1([0,1])$ then $\mathfrak{L}_{1,0}f \in L^1([0,1])$ and $\int_0^1(\mathfrak{L}_{1,0}f)(x)dx=\int_0^1f(x)dx$. In particular if $f \in L^1([0,1])$ then
\[\left\|\mathfrak{L}_{1,0}f\right\|_{L^1}=\int_0^1\left|\left(\mathfrak{L}_{1,0}f\right)(x)\right|dx \leq \int_0^1 \left(\mathfrak{L}_{1,0}|f|\right)(x)dx=\int_0^1|f(x)|dx=\|f\|_{L^1}\]
so that $\mathfrak{L}_{1,0}$ acts on $L^1([0,1])$ in the manner claimed. 

Now let $g \in \HT$. Using Lemma \ref{basic0} and Proposition \ref{spax}
\begin{align*}\limsup_{n \to \infty} \left\|\mathfrak{L}_{1,0}^ng - \left(\int_0^1g(x)dx\right)\xi\right\|_{L^1}&\leq \limsup_{n \to \infty} \frac{\pi}{2}\left\|\mathfrak{L}_{1,0}^ng - \left(\int_0^1g(x)dx\right)\xi\right\|_{\HT}\\
&=\limsup_{n \to \infty} \frac{\pi}{2}\left\|\mathfrak{L}_{1,0}^ng - \mathcal{P}_{1,0} g\right\|_{\HT}\\
 &= \limsup_{n \to \infty} \left\|\mathcal{N}_{1,0}^ng\right\|_{\HT}\\&\leq \limsup_{n \to \infty} \left\|\mathcal{N}_{1,0}^n\right\|_{\HT}\|g\|_{\HT}=0\end{align*}
so that
\[\lim_{n \to \infty}\left\|\mathfrak{L}_{1,0}^ng-\left(\int_0^1g(x)dx\right)\xi\right\|_{L^1}=0.\]
Given $f \in L^1([0,1])$, we may for each $\varepsilon >0$ choose a polynomial function $g \in H^2(\mathbb{D})$ such that $\|f-g\|_{L^1}<\varepsilon$. For each $n \geq 1$ we have
\[\left\|\mathfrak{L}_{1,0}^nf-\mathfrak{L}^n_{1,0}g\right\|_{L^1} \leq \|f-g\|_{L^1}<\varepsilon\]
and
\[\left\|\left(\int_0^1f(x)dx\right)\xi-\left(\int_0^1g(x)dx\right)\xi\right\|_{L^1} \leq \|f-g\|_{L^1}\|\xi\|_{L^1}<\varepsilon\]
 and therefore
\begin{align*}\limsup_{n\to\infty}\left\|\mathfrak{L}_{1,0}^nf - \left(\int_0^1f(x)dx\right)\xi\right\|_{L^1} &<2\varepsilon + \limsup_{n \to \infty}\left\|\mathfrak{L}_{1,0}^ng-\left(\int_0^1g(x)dx\right)\xi\right\|_{L^1}\\
& =2\varepsilon.\end{align*}
Since $\varepsilon$ is arbitrary we conclude that $\lim_{n \to \infty}\mathfrak{L}_{1,0}^nf=\left(\int_0^1f(x)dx\right)\xi$ as claimed. It follows directly that if $\mathfrak{L}_{1,0}f=f$ then $f=\left(\int_0^1f(x)dx\right)\xi$. 
\end{proof}

\section{The derivatives of the leading eigenvalue}\label{se7en}
We now take our first steps towards the proof of Theorem \ref{outcome} by investigating the derivatives of the function $\lambda$ defined in Theorem \ref{MainBrent}. This will be applied in the following two sections when we relate the operator $\mathfrak{L}_{s,\omega}$ to the quantity $\mu(c)$ defined in Theorem \ref{outcome} via the equation \eqref{brondesbury}.

As well as providing the important information that the derivative of $\lambda(s,0)$ at $s=1$ is nonzero, the following result is crucial in unifying several of the expressions for the asymptotic number of subtraction steps which were stated in Theorem \ref{outcome}. In this and all subsequent sections we use the notation $\lambda_s$ and $\lambda_\omega$ to refer to the partial derivatives of $\lambda$ with respect to the first and second variables respectively.
\begin{proposition}\label{trivium}
Let $\mathcal{V}\subset \CC^2$ and $\lambda \colon \mathcal{V} \to \CC$ be as given in Theorem \ref{MainBrent}. Then
\begin{align*}\lambda_s(1,0)&= \sum_{k=1}^\infty \frac{2}{2^k}\left(\int_0^{\frac{1}{1+2^k}} \log\left(\frac{1-x}{2^k}\right)\xi(x)dx+\int_{\frac{1}{1+2^k}}^1 \log(x)\xi(x)dx\right)\\
&=-\sum_{k=1}^\infty\frac{2}{2^k}\int_0^1 \log\left(\frac{2^k(1+x)}{1+(2^k-1)x}\right)\xi(x)dx\\
&=\int_0^1\log(1-x)\xi(x)dx-\log 4.\end{align*}
\end{proposition}
\begin{proof}
We begin the proof with a calculation of a type which is rather standard in the theory of Ruelle operators (see for example \cite{PP90}). 
Let $V:=\{s \in \CC \colon (s,0)\in \mathcal{V}\}$ and define $\xi_s:=\mathcal{P}_{s,0}\xi$ for every $s \in V$. Clearly the function from $V$ to $\HT$ defined by $s \mapsto \xi_s$ is holomorphic and satisfies $\xi_1= \xi$, and we have \begin{equation}\label{fatuous}\mathfrak{L}_{s,0}\xi_s=\mathfrak{L}_{s,0}\mathcal{P}_{s,0}\xi=\lambda(s,0)\mathcal{P}_{s,0}\xi=\lambda(s,0)\xi_s\end{equation}
for every $s \in V$. For each $s \in V$ let $\xi_s'\in \HT$ denote the first derivative of the function $s \mapsto \xi_s$ evaluated at $s$.

For each $s \in V$ and $z \in \DD$ we may use \eqref{fatuous} to write
\[\lambda(s,0)\xi_s(z) =\sum_{k=1}^\infty \frac{1}{(1+2^kz)^{2s}}\xi_s\left(\frac{1}{1+2^kz}\right) + \frac{1}{(z+2^k)^{2s}}\xi_s\left(\frac{z}{z+2^k}\right)\]
and for each fixed $z \in \DD$ this series converges absolutely in a manner which is locally uniform with respect to $s$. It follows that for each $z \in \DD$ we may differentiate termwise with respect to $s$ at $s=1$ to obtain%\footnote{This is the most hideous piece of mathematical typesetting I have ever done. Suggestions for improvement welcome.}
\begin{eqnarray*}\lefteqn{\lambda_s(1,0)\xi(z) + \xi'_1(z)}\\
 & =&  \sum_{k=1}^\infty \frac{1}{(1+2^kz)^2}\xi_1'\left(\frac{1}{1+2^kz}\right) + \frac{1}{(z+2^k)^2}\xi_1'\left(\frac{z}{z+2^k}\right)\\
& & + \sum_{k=1}^\infty \frac{-2\log(1+2^kz)}{(1+2^kz)^2}\xi\left(\frac{1}{1+2^kz}\right) + \frac{-2\log(z+2^k)}{(z+2^k)^2}\xi\left(\frac{z}{z+2^k}\right)\end{eqnarray*}
of which the right-hand side simplifies to
\[\left(\mathfrak{L}_{1,0}\xi_1'\right)(z) -2\left(\sum_{k=1}^\infty \frac{\log(1+2^kz)}{(1+2^kz)^2}\xi\left(\frac{1}{1+2^kz}\right) + \frac{\log(z+2^k)}{(z+2^k)^2}\xi\left(\frac{z}{z+2^k}\right)\right).\]
Integrating along the interval $(0,1)$, applying Lemma \ref{wath} and eliminating the term $\int_0^1\xi_1'(x)dx$ from both sides of the equation we derive the identity
\begin{align*}\lambda_s(1,0)=&-2\sum_{k=1}^\infty \int_0^1\frac{\log(1+2^kx)}{(1+2^kx)^2}\xi\left(\frac{1}{1+2^kx}\right)dx\\
& -2\sum_{k=1}^\infty \int_0^1\frac{\log(x+2^k)}{(x+2^k)^2}\xi\left(\frac{x}{z+2^k}\right)dx.\end{align*}
Using the substitution $u=\frac{1}{1+2^kx}$ for each $k$ we may obtain
\[\sum_{k=1}^\infty \int_0^1\frac{\log(1+2^kx)}{(1+2^kx)^2}\xi\left(\frac{1}{1+2^kx}\right)dx=-\sum_{k=1}^\infty \frac{1}{2^k}\int_{\frac{1}{1+2^k}}^1(\log u)\xi(u)du,\]
and similarly substituting $v=\frac{x}{x+2^k}$ for each $k$ yields
\begin{align*}\sum_{k=1}^\infty \int_0^1\frac{\log(x+2^k)}{(x+2^k)^2}\xi\left(\frac{x}{x+2^k}\right)dx&=-\sum_{k=1}^\infty \int_0^1\frac{\log\left(\frac{1}{2^k}\left(1-\frac{x}{x+2^k}\right)\right)}{(x+2^k)^2}\xi\left(\frac{x}{x+2^k}\right)dx\\
&=-\sum_{k=1}^\infty \frac{1}{2^k}\int_0^{\frac{1}{1+2^k}}\log \left(\frac{1-v}{2^k}\right)\xi(v)dv,\end{align*}
so by combining these results we may obtain
\begin{equation}\label{nooo}\lambda_s(1,0)=\sum_{k=1}^\infty \frac{2}{2^k}\left(\int_0^{\frac{1}{1+2^k}}\log \left(\frac{1-x}{2^k}\right)\xi(x)dx+\int_{\frac{1}{1+2^k}}^1(\log x)\xi(x)dx\right)\end{equation}
which is the first of the three identities claimed.

We now make the following general assertion: if $f \colon (0,1) \to \mathbb{R}$ is a measurable function such that $\int_0^1|f(x)\xi(x)|dx$ is finite, then
\[\sum_{k=1}^\infty \frac{1}{2^k} \left(\int_0^{\frac{1}{1+2^k}} f\left(\frac{2^kx}{1-x}\right)\xi(x)dx + \int_{\frac{1}{1+2^k}}^1 f\left(\frac{1-x}{2^kx}\right)\xi(x)dx\right)=\int_0^1f(x)\xi(x)dx.\]
Viewed as a statement about the random dynamical system determined by the family of maps $T_k \colon [0,1] \to [0,1]$, this assertion equates to the statement that the product of the probability measure with respect to which the maps are chosen with the absolutely continuous measure on $[0,1]$ with density $\xi$ is stationary with respect to the skew product transformation.

Let us prove the claim. Given such a function $f$, using the substitution $u=(1-x)/2^kx$ yields
\[\frac{1}{2^k}\int_{\frac{1}{1+2^k}}^1 f\left(\frac{1-x}{2^k}\right)\xi(x)dx=\int_0^1 \frac{f(u)}{(1+2^ku)^2}\xi\left(\frac{1}{1+2^ku}\right)du\]
and the substitution $v=2^kx/(1-x)$ similarly yields
\[\frac{1}{2^k}\int_0^{\frac{1}{1+2^k}} f\left(\frac{2^kx}{1-x}\right)\xi(x)dx=\int_0^1 \frac{f(v)}{(v+2^k)^2}\xi\left(\frac{v}{v+2^k}\right)dv.\]
Since by definition $\xi(x)=(\mathfrak{L}_{1,0}\xi)(x)$ for every $x \in (0,1)$ it follows that indeed
\begin{align*}\int_0^1f(x)\xi(x)dx&=\sum_{k=1}^\infty \int_0^1 \frac{f(x)}{(1+2^kx)^2}\xi\left(\frac{1}{1+2^kx}\right) + \frac{f(x)}{(x+2^k)^2}\xi\left(\frac{x}{x+2^k}\right)dx\\
&=\sum_{k=1}^\infty \frac{1}{2^k} \left(\int_0^{\frac{1}{1+2^k}} f\left(\frac{2^kx}{1-x}\right)\xi(x)dx + \int_{\frac{1}{1+2^k}}^1 f\left(\frac{1-x}{2^kx}\right)\xi(x)dx\right)\end{align*}
as was claimed.

Let us now apply the claim with $f(x):=2\log(1+x)$, which clearly satisfies the integrability hypothesis. In this case the claim results in the identity
\begin{eqnarray*}
\lefteqn{2\int_0^1\log(1+x)\xi(x)dx}\\
&  =&\sum_{k=1}^\infty \frac{2}{2^k}\left(\int_0^{\frac{1}{1+2^k}} \log\left(1+\frac{2^kx}{1-x}\right)\xi(x)dx+\int_{\frac{1}{1+2^k}}^1 \log\left(1+\frac{1-x}{2^kx}\right)\xi(x)dx\right)\\
&  =&\sum_{k=1}^\infty \frac{2}{2^k}\int_0^{\frac{1}{1+2^k}} \log\left(\frac{1+(2^k-1)x}{1-x}\right)\xi(x)dx\\
& & +\sum_{k=1}^\infty \frac{2}{2^k}\int_{\frac{1}{1+2^k}}^1 \log\left(\frac{1+(2^k-1)x}{2^kx}\right)\xi(x)dx\end{eqnarray*}
%{\footnotesize{\begin{eqnarray*}\lefteqn{2\int_0^1\log(1+x)\xi(x)dx}\\& & =\sum_{k=1}^\infty \frac{2}{2^k}\left(\int_0^{\frac{1}{1+2^k}} \log\left(1+\frac{2^kx}{1-x}\right)\xi(x)dx+\int_{\frac{1}{1+2^k}}^1 \log\left(1+\frac{1-x}{2^kx}\right)\xi(x)dx\right)\\& & =\sum_{k=1}^\infty \frac{2}{2^k}\left(\int_0^{\frac{1}{1+2^k}} \log\left(\frac{1+(2^k-1)x}{1-x}\right)\xi(x)dx+\int_{\frac{1}{1+2^k}}^1\log\left(\frac{1+(2^k-1)x}{2^kx}\right)\xi(x)dx\right)\end{eqnarray*}}}
and by adding this to the already-established identity \eqref{nooo} we obtain
\[\lambda_s(1,0)+2\int_0^1\log(1+x)\xi(x)dx=\sum_{k=1}^\infty \frac{2}{2^k}\int_0^1\log \left(\frac{1+(2^k-1)x}{2^k}\right)\xi(x)dx\]
or more simply
\[\lambda_s(1,0) = -\sum_{k=1}^\infty \frac{2}{2^k}\int_0^1 \log\left(\frac{2^k(1+x)}{1+(2^k-1)x}\right)\xi(x)dx\]
which is the second identity asserted in the statement of the proposition. Finally let us apply the claim with  $f(x):=\log x$, which meets the integrability hypothesis since $|f(x)\xi(x)|\leq C(1+|\log x|^2)$ for all $x \in (0,1)$ for some positive constant $C$. In this case the claim yields
\begin{eqnarray*}
\lefteqn{\int_0^1(\log x)\xi(x)dx}\\
& & = \sum_{k=1}^\infty \frac{1}{2^k}\left(\int_0^{\frac{1}{1+2^k}}\log\left(\frac{2^kx}{1-x}\right)\xi(x)dx - \int_{\frac{1}{1+2^k}}^1 \log\left(\frac{2^kx}{1-x}\right)\xi(x)dx\right).\end{eqnarray*}
Adding this equation to the previously-established identity \eqref{nooo} results in the identity
\[\lambda_s(1,0)+\int_0^1(\log x)\xi(x)dx=\sum_{k=1}^\infty \frac{1}{2^k} \int_0^1\log\left(\frac{x(1-x)}{2^k}\right)\xi(x)dx\]
which simplifies to
\[\lambda_s(1,0)=\int_0^1\log\left(1-x\right)\xi(x)dx-\log 4,\]
and this is the third identity asserted in the statement of the proposition.  The proof is complete.
\end{proof}
The following result allows us to relate the expression $\mu(c)$ defined in the statement of Theorem \ref{outcome} to the function $\lambda$.
\begin{lemma}\label{deepship}
Let $\mathcal{V}\subset\CC^2$ and $\lambda \colon \mathcal{V} \to \CC$ be as given in Theorem \ref{MainBrent}. Then
\[\lambda_\omega(1,0)=\sum_{k=1}^\infty \frac{1}{2^k}\left(c(2,k)\int_0^{\frac{1}{1+2^k}}\xi(x)dx +c(1,k)\int_{\frac{1}{1+2^k}}^1\xi(x)dx \right).\]
\end{lemma}
\begin{proof}
Similarly to the proof of Proposition \ref{trivium} let $W:=\{\omega \in \CC \colon (1,\omega)\in \mathcal{V}\}$ and define $\xi_\omega:=\mathcal{P}_{1,\omega}\xi$ for every $\omega \in W$. Clearly the function from $W$ to $\HT$ defined by $\omega \mapsto \xi_\omega$ is holomorphic and satisfies $\xi_0= \xi$, and $\mathfrak{L}_{1,\omega}\xi_\omega=\lambda(1,\omega)\xi_\omega$ for every $\omega \in W$. For each $\omega \in W$ let $\xi_\omega'\in \HT$ denote the first derivative of the function $\omega \mapsto \xi_\omega$ evaluated at $\omega$. For each $\omega \in W$ and $z \in \DD$ we have
\[\lambda(1,\omega)\xi_\omega(z) =\sum_{k=1}^\infty \frac{\exp(\omega c(1,k))}{(1+2^kz)^{2}}\xi_\omega\left(\frac{1}{1+2^kz}\right) + \frac{\exp(\omega c(2,k))}{(z+2^k)^{2}}\xi_\omega\left(\frac{z}{z+2^k}\right)\]
and for each fixed $z \in \DD$ this series converges absolutely in a manner which is locally uniform with respect to $\omega$. It follows that for each $z \in \DD$ we may differentiate termwise with respect to $\omega$ at $\omega=0$ to obtain
\begin{align*}\lambda_\omega(1,0)\xi(z) + \xi'_0(z)&=  \sum_{k=1}^\infty \frac{1}{(1+2^kz)^2}\xi_0'\left(\frac{1}{1+2^kz}\right) + \frac{1}{(z+2^k)^2}\xi_0'\left(\frac{z}{z+2^k}\right)\\
&\phantom{=}\quad+ \sum_{k=1}^\infty \frac{c(1,k)}{(1+2^kz)^2}\xi\left(\frac{1}{1+2^kz}\right) + \frac{c(2,k)}{(z+2^k)^2}\xi\left(\frac{z}{z+2^k}\right).\end{align*}
Integrating along the interval $(0,1)$ and subtracting the quantity $\int_0^1\xi'_0(x)dx$ from either side yields
\begin{align*}\lambda_\omega(1,0)&= \int_0^1 \left(\sum_{k=1}^\infty \frac{c(1,k)}{(1+2^kx)^2}\xi\left(\frac{1}{1+2^kx}\right) + \frac{c(2,k)}{(x+2^k)^2}\xi\left(\frac{x}{x+2^k}\right)\right)dx\\
&=\sum_{k=1}^\infty \frac{1}{2^k}\left(c(2,k)\int_0^{\frac{1}{1+2^k}}\xi(x)dx +c(1,k)\int_{\frac{1}{1+2^k}}^1\xi(x)dx \right)\end{align*}
in a straightforward manner.
\end{proof}

\section{Properties of the Dirichlet series}\label{ate}
In this section we establish the equation \eqref{brondesbury} which relates the subject of Theorem \ref{MainBrent} with that of Theorem \ref{outcome}, and apply it to study Dirichlet series in one variable which describe the moments of the distribution of $C(u,v)$ on $\Xi^{(1)}_n$ and $\Xi^{(2)}_n$.  The desired correspondence rests on the following dull but necessary technical lemma:
\begin{lemma}\label{dbnt}
Let $(s,\omega) \in \mathcal{U}$ where $\mathcal{U}$ is as defined in Theorem \ref{MainBrent}. For each $n \geq 1$ let $\Theta_n$ denote the set of all pairs of coprime odd natural numbers $(u,v)$, where $u \leq v$, which are mapped to $(1,1)$ by exactly $n$ steps of the binary Euclidean algorithm. Then for each $(s,\omega) \in \mathcal{U}$ and $n \geq 1$,
\[\left(\mathfrak{D}_{s,\omega}\mathfrak{L}_{s,\omega}^{n-1}\mathbf{1}\right)(1)=\sum_{(u,v) \in \Theta_n} \frac{\exp(\omega C(u,v))}{v^{2s}}.\]
\end{lemma}
\begin{proof}
Let $\mathcal{H}$ denote the set of all linear fractional transformations $h$ which either take the form $h(z)=\frac{1}{1+2^kz}$ for some $k \geq 1$, or take the form $h(z)=\frac{z}{z+2^k}$ for some $k \geq 1$. Define $\mathcal{H}_D$ to be the subset of $\mathcal{H}$ consisting only of those transformations which have the form $h(z)=\frac{z}{z+2^k}$ for some $k \geq 1$.  We define the \emph{branch cost} $\mathfrak{c} \colon \mathcal{H} \to \mathbb{R}$ and \emph{determinant} $\mathfrak{d}\colon \mathcal{H} \to \mathbb{N}$ respectively by defining $\mathfrak{c}(h):=c(1,k)$ and $\mathfrak{d}(h):=-2^k$ when $h(z)=\frac{1}{1+2^kz}$, and 
 $\mathfrak{c}(h):=c(2,k)$ and $\mathfrak{d}(h):=2^k$ when $h(z)=\frac{z}{z+2^k}$. With these conventions the operator $\mathfrak{L}_{s,\omega}$ may be alternatively expressed as
\[\left(\mathfrak{L}_{s,\omega}f\right)(z)=\sum_{h \in \mathcal{H}} e^{\omega\mathfrak{c}(h)}\left(\frac{h'(z)}{\mathfrak{d}(h)}\right)^sf\left(h(z)\right),\]
and furthermore
\[\left(\mathfrak{D}_{s,\omega}f\right)(z)=\sum_{h \in \mathcal{H}_D} e^{\omega\mathfrak{c}(h)}\left(\frac{h'(z)}{\mathfrak{d}(h)}\right)^sf\left(h(z)\right),\]
so for each $n \geq 1$ we have
\begin{align}\label{fbuaclkls}\left(\mathfrak{D}_{s,\omega}\mathfrak{L}_{s,\omega}^{n-1}\mathbf{1}\right)(1)
&=\sum_{\substack{h_1 \in \mathcal{H}_D\\h_2,\ldots,h_n \in \mathcal{H}}}\prod_{i=1}^n e^{\omega \mathfrak{c}(h_1)} \left(\frac{h_i'\left(\left(h_{i-1}\circ \cdots \circ h_1\right)(1)\right)}{\mathfrak{d}(h_i)} \right)^s \\\nonumber
&=\sum_{\substack{h_1 \in \mathcal{H}_D\\h_2,\ldots,h_n \in \mathcal{H}}}e^{\omega\sum_{i=1}^n \mathfrak{c}(h_i)}\left(\frac{\left(h_n \circ \cdots \circ h_1\right)'(1)}{\mathfrak{d}(h_n)\cdots \mathfrak{d}(h_1)}\right)^s.\end{align}
We will show that this last sum matches the second expression given in the statement of the lemma, and to do this we must characterise the sets $\Theta_n$ in terms of the functions $h \in \mathcal{H}$.  Clearly we have $\Theta_0 = \{(1,1)\}$ and $\Theta_1=\{(1,1+2^k)\colon k \geq 1\}$, and $\Theta_n \cap \Theta_m =\emptyset$ when $m \neq n$. We make the following claim: for each $n \geq 1$ we have $(u,v) \in \Theta_n$ if and only if there exists a finite sequence $h_1,\ldots,h_n \in \mathcal{H}$ such that $h_1 \in \mathcal{H}_D$ and
\begin{equation}\label{snowmobile-thud}\frac{u}{v}=\left(h_n \circ \cdots \circ h_1\right)(1),\end{equation}
and to each $(u,v) \in \Theta_n$ there corresponds a unique such sequence $h_1,\ldots,h_n$; furthermore, when \eqref{snowmobile-thud} is satisfied with $h_1 \in \mathcal{H}_D$ we have $C(u,v)=\sum_{i=1}^n \mathfrak{c}(h_i)$. 

We first consider the case $n=1$. We have $(u,v) \in \Theta_1$ if and only if $u=1$ and $v=1+2^k$ for some integer $k \geq 1$. In this case a single step of the algorithm subtracts $u$ from $v$, divides by $2^k$ and does not perform an exchange, so the appropriate cost is $c(2,k)$. It is clear that $\frac{u}{v}=h(1)$ for $h(z)=\frac{z}{z+2^k}$ and that this relation does not hold when $h$ is replaced with a different element of $\mathcal{H}_D$, and we have $\mathfrak{c}(h)=c(2,k)=C(u,v)$ as required. Conversely if $\frac{u}{v}=h(1)$ in least terms for some $h \in \mathcal{H}_D$ then $(u,v)=(1,1+2^k)$ for some integer $k$ and therefore $(u,v) \in \Theta_1$. This completes the proof in the case $n=1$.
 
Let us now suppose that case $n$ of the claim has been proved and deduce case $n+1$. It is sufficient to show that if $(u,v) \in \Theta_n$ and $\frac{u}{v}=(h_n \circ \cdots \circ h_1)(1)$ then the numerator and denominator of $h(\frac{u}{v})$ form a pair belonging to $\Theta_{n+1}$ for every $h \in \mathcal{H}$, that for every $(u,v) \in \Theta_{n+1}$ there exist a unique $(p,q) \in \Theta_n$ and a unique $h \in \mathcal{H}$ such that $\frac{u}{v}=h(\frac{p}{q})$, and that $C(u,v)=\mathfrak{c}(h)+C(p,q)$.

The first assertion is straightforward. If $(u,v) \in \Theta_n$ and $h(z)=\frac{1}{1+2^kz}$, then $h(\frac{u}{v})=\frac{v}{v+2^ku}$ in least terms, so the numerator and the denominator are odd and coprime. It is clear that the pair $(v,v+2^ku)$ is mapped to $(u,v)$ by one step of the algorithm so that $(v,v+2^ku) \in \Theta_{n+1}$ as claimed. Similarly, if $h(z)=\frac{z}{z+2^k}$ then $h(\frac{u}{v})=\frac{u}{u+2^kv}$ in least terms with odd numerator and denominator and we may easily check that $(u,u+2^kv)\in\Theta_{n+1}$.

Let us prove the second assertion. If $(u,v) \in \Theta_{n+1}$ where $v-u$ is divisible by $2$ exactly $k$ times and $2^{-k}(v-u) \geq u$, then a single iteration of the binary algorithm takes $(u,v)$ to $(u,2^{-k}(v-u))$ and this operation contributes a cost of $c(2,k)$. The pair $(p,q):=\left(u,2^{-k}(v-u)\right)$ is clearly also a pair of coprime odd natural numbers with the second term being greater than or equal to the first and hence belongs to $\Theta_n$. Furthermore we may write
\[\frac{u}{v} = \frac{\left(\frac{u}{2^{-k}(v-u)}\right)}{\left(\frac{u}{2^{-k}(v-u)}\right)+2^k}=h\left(\frac{p}{q}\right),\]
say, where $h\in\mathcal{H}$ is given by $h(z)=\frac{z}{z+2^k}$ and thus $C(u,v)=c(2,k)+C(p,q)=\mathfrak{c}(h)+C(p,q)$ as desired. If on the other hand  $v-u$ is divisible by $2$ exactly $k$ times and $u > 2^{-k}(v-u)$, then in a similar fashion a single step of the binary algorithm takes $(u,v)$ to $(p,q):=\left(2^{-k}(v-u),u\right) \in \Theta_n$ contributing a cost of $c(1,k)$, and we may write
\[\frac{u}{v} = \frac{1}{1+2^k\left(\frac{2^{-k}(v-u)}{u}\right)}= h\left(\frac{p}{q}\right)\]
where $h \in \mathcal{H}$ is given by $h(z):=\frac{1}{1+2^kz}$ so that $C(u,v)=c(1,k)+C(p,q)=\mathfrak{c}(h)+C(p,q)$ as required.

We must now shown that this correspondence is unique. Let $(u,v) \in \Theta_{n+1}$ and $(p_1,q_1), (p_2,q_2) \in \Theta_n$ such that $h_1(p_1/q_1)=h_2(p_2/q_2)=u/v$. Using symmetry, the identity $h_1(p_1/q_1)=h_2(p_2/q_2)$ implies that there exist $k,\ell \geq 1$ such that one of the following three equations holds:
\[\frac{1}{1+2^k\frac{p_1}{q_1}}=\frac{1}{1+2^\ell\frac{p_2}{q_2}},\]
\[\frac{1}{1+2^k\frac{p_1}{q_1}}=\frac{\frac{p_2}{q_2}}{\frac{p_2}{q_2}+2^\ell},\]
\[\frac{\frac{p_2}{q_2}}{\frac{p_2}{q_2}+2^k}=\frac{\frac{p_2}{q_2}}{\frac{p_2}{q_2}+2^\ell}.\]
If the first one holds then $2^kq_2p_1=2^\ell p_2q_1$ and therefore $k=\ell$ so that $h_1=h_2$ and $\frac{p_1}{q_1}=\frac{p_2}{q_2}$. If the second holds then $2^kp_2p_1=2^\ell q_2q_1$ so that $k=\ell$ and $1 \leq \frac{q_2}{p_2} = \frac{p_1}{q_1}\leq 1$ which contradicts $n \neq 0$. If the third holds then $2^kp_2q_1=2^\ell q_2p_1$ so that $k=\ell$, $h_1=h_2$ and $\frac{p_1}{q_1}=\frac{p_2}{q_2}$ as required. This completes the proof of case $n+1$ and by induction completes the proof of the claim.

Let us now prove the statement of the lemma. Let $n \geq 1$ and $(u,v) \in \Theta_n$ with $u/v=(h_n \circ \cdots\circ h_1)(1)$.
A simple inductive proof shows that the composition $h_n \circ \cdots \circ h_1$ is a linear fractional transformation $z \mapsto (\alpha z + \beta)/(\gamma z + \delta)$ such that $\alpha\delta-\beta\gamma=\mathfrak{d}(h_n) \cdots \mathfrak{d}(h_1)$ and $\alpha+\beta$ is coprime to $\gamma+\delta$. By direct calculation we have $(h_n\circ \cdots \circ h_1)'(z)=(\mathfrak{d}(h_n) \cdots \mathfrak{d}(h_1))/(\gamma z+\delta)^2$ and $v=\gamma+\delta$, so in particular $(h_n\circ \cdots \circ h_1)'(1)=(\mathfrak{d}(h_n) \cdots \mathfrak{d}(h_1))/v^2$.
It follows that
\[\frac{1}{v^{2s}}=\left(\frac{\left(h_n \circ \cdots \circ h_1\right)'(1)}{\mathfrak{d}(h_n)\cdots \mathfrak{d}(h_1)}\right)^s\]
and thus by the preceding claim together with \eqref{fbuaclkls}
\begin{align*}\sum_{(u,v)\in\Theta_n}\frac{\exp(\omega C(u,v))}{v^{2s}} &=\sum_{\substack{h_1\in\mathcal{H}_D\\h_2,\ldots,h_n \in \mathcal{H}}}e^{\omega\sum_{i=1}^n \mathfrak{c}(h_i)}\left(\frac{\left(h_n \circ \cdots \circ h_1\right)'(1)}{\mathfrak{d}(h_n)\cdots \mathfrak{d}(h_1)}\right)^s \\&=\left(\mathfrak{D}_{s,\omega}\mathfrak{L}_{s,\omega}^{n-1}\mathbf{1}\right)(1)\end{align*}
as required. The proof is complete.
\end{proof}

The following proposition, alluded to in \S\ref{arsity}, relates the cost functions to be studied in Theorem \ref{outcome} to the operators considered in Theorem \ref{MainBrent}. We state this result in a somewhat more general form than is strictly required for the purposes of this article, in case the full statement is found useful in future investigations into the asymptotic distribution of costs.
\begin{proposition}\label{dseriesmain}
There exists an open set $\mathcal{W} \subset \CC^2$ which contains the set
\[\{(s,\omega) \colon \Re(s)>1 \text{ and }\omega=0\}\]
such that for all $(s,\omega) \in \mathcal{W}$ the series
\begin{equation}\label{fourmyle}\sum_{(u,v) \in \Xi^{(1)}} \frac{\exp(\omega C(u,v))}{v^{2s}}\end{equation}
converges absolutely to a holomorphic function defined throughout $\mathcal{W}$. This function admits an analytic continuation to a larger open set which includes the set
\[\{(s,\omega) \colon \Re(s)=1, s \neq 1\text{ and }\omega=0\}.\]
Let $\mathcal{V} \subset \CC^2$,  $\lambda \colon \mathcal{V} \to \CC$ and $\mathcal{P}_{(\cdot,\cdot)} \colon \mathcal{V} \to \mathcal{B}(\HT)$ be as in Theorem \ref{MainBrent}. Then there exists a holomorphic function $R \colon \mathcal{V} \to \CC$ such that for all $(s,\omega) \in \mathcal{V} \cap \mathcal{W}$
\begin{equation}\label{moraghasanicebum}\sum_{(u,v) \in \Xi^{(1)}} \frac{\exp(\omega C(u,v))}{v^{2s}} = \frac{(\mathfrak{D}_{s,\omega}\mathcal{P}_{s,\omega}\mathbf{1})(1)}{1-\lambda(s,\omega)} + R(s,\omega).\end{equation}
Furthermore, for all $(s,\omega) \in \mathcal{W}$ the series 
\[\sum_{(u,v) \in \Xi^{(2)}} \frac{\exp(\omega C(u,v))}{v^{2s}}\]
also converges absolutely and satisfies 
\begin{equation}\label{moragissexy}\sum_{(u,v) \in \Xi^{(2)}} \frac{\exp(\omega C(u,v))}{v^{2s}}  =\zeta(2s)\left(1-4^{-s}\right)\left(\sum_{(u,v) \in \Xi^{(1)}} \frac{\exp(\omega C(u,v))}{v^{2s}}\right)\end{equation}
where $\zeta$ denotes the Riemann zeta function.
In particular this function admits an analytic continuation to the same region as the series \eqref{fourmyle}.
\end{proposition}
\begin{proof}
 Let us define
\[\mathcal{W}:=\left\{(s,\omega) \in \mathcal{U} \colon \rho\left(\mathfrak{L}_{\Re(s),\Re(\omega)}\right)<1\text{ and }\rho\left(\mathfrak{L}_{s,\omega}\right)<1\right\},\]
and
\[\hat{\mathcal{W}}:=\left\{(s,\omega) \in \mathcal{U} \colon \rho\left(\mathfrak{L}_{s,\omega}\right)<1\right\},\]
The spectral radius function $(s,\omega) \mapsto \rho(\mathfrak{L}_{s,\omega})$ is upper semi-continuous since it is an infimum of continuous functions, so $\mathcal{W}$ and $\hat{\mathcal{W}}$ are open. It follows from Theorem \ref{MainBrent}(d) that 
\[\left\{(s,\omega) \in \CC \colon \Re(s)>1\text{ and }\omega=0\right\}\subset\mathcal{W}\]
and
\[\left\{(s,\omega) \in \CC \colon \Re(s) \geq1, \Re(s) \neq 1\text{ and }\omega=0\right\}\subset\hat{\mathcal{W}}.\]
If $(s,\omega) \in \mathcal{W}$ then since $\rho(\mathfrak{L}_{\Re(s),\Re(\omega)})<1$ we have for all $N \geq 1$
\begin{align*}\sum_{n=1}^N \sum_{(u,v) \in \Theta_n} \left|\frac{\exp(\omega C(u,v))}{v^{2s}}\right|&=\sum_{n=1}^N \sum_{(u,v) \in \Theta_n} \frac{\exp(\Re(\omega)C(u,v))}{v^{2\Re(s)}}\\
&=\sum_{n=1}^N\left(\mathfrak{D}_{\Re(s),\Re(\omega)}\mathfrak{L}_{\Re(s),\Re(\omega)}^{n-1}\mathbf{1}\right)(1)\\
&\leq \sum_{n=1}^N\left\|\mathfrak{D}_{\Re(s),\Re(\omega)}\mathfrak{L}_{\Re(s),\Re(\omega)}^{n-1}\mathbf{1}\right\|_{\HT}\\
&< \sum_{n=1}^\infty\left\|\mathfrak{D}_{\Re(s),\Re(\omega)}\right\|_{\HT}\left\|\mathfrak{L}_{\Re(s),\Re(\omega)}^{n-1}\right\|_{\HT}<\infty\end{align*}
using Lemma \ref{dbnt} and Lemma \ref{basic0}, which proves the claimed absolute convergence. For each $N \geq 1$ we have
\[\sum_{n=1}^N \sum_{(u,v) \in \Theta_n}\frac{\exp(\omega C(u,v))}{v^{2s}}=\sum_{n=1}^N\left(\mathfrak{D}_{s,\omega}\mathfrak{L}_{s,\omega}^{n-1}\mathbf{1}\right)(1)\]
by Lemma \ref{dbnt} and since similarly
\begin{equation}\label{dubya}\sum_{n=1}^N\left|\left(\mathfrak{D}_{s,\omega}\mathfrak{L}_{s,\omega}^{n-1}\mathbf{1}\right)(1)\right|\leq \sum_{n=1}^\infty \left\|\mathfrak{D}_{s,\omega}\right\|_{\HT}\left\|\mathfrak{L}_{s,\omega}^{n-1}\right\|_{\HT}<\infty\end{equation}
it follows that by absolute convergence we may let $N \to \infty$ to obtain the expression
\begin{equation}\label{equine}\sum_{(u,v) \in \Xi^{(1)}}\frac{\exp(\omega C(u,v))}{v^{2s}} = \sum_{n=1}^\infty \left(\mathfrak{D}_{s,\omega}\mathfrak{L}_{s,\omega}^{n-1}\mathbf{1}\right)(1)\end{equation}
for all $(s,\omega) \in \mathcal{W}$. Since the convergence is locally uniform in $(s,\omega)$ the resulting function is holomorphic in this region; moreover, since the right-hand side of \eqref{dubya} converges whenever $\rho(\mathfrak{L}_{s,\omega})<1$, it follows that the left-hand side of \eqref{equine} may be analytically extended to a holomorphic function in the larger region $\hat{\mathcal{W}}$. This completes the part of the proof concerned with convergence and analytic continuation for the sum over $\Xi^{(1)}$. 

Let us now consider the behaviour of the series \eqref{fourmyle} in the region $\mathcal{V}$. By Theorem \ref{MainBrent}, for all $(s,\omega) \in \mathcal{V}$ we may write $\mathfrak{L}_{s,\omega}=\lambda(s,\omega)\mathcal{P}_{s,\omega}+\mathcal{N}_{s,\omega}$ where $\mathcal{P}_{s,\omega}\mathcal{N}_{s,\omega}=\mathcal{N}_{s,\omega}\mathcal{P}_{s,\omega}=0$ and $\rho(\mathcal{N}_{s,\omega})<1$, and the functions $(s,\omega) \mapsto \mathcal{P}_{s,\omega}$ and $(s,\omega) \mapsto \mathcal{N}_{s,\omega}$ are holomorphic. 
Define $R(s,\omega):=\sum_{n=1}^\infty\left(\mathfrak{D}_{s,\omega}\mathcal{N}_{s,\omega}^{n-1}\mathbf{1}\right)(1)$ for all $(s,\omega)\in\mathcal{V}$; this series converges absolutely to a holomorphic function $R\colon\mathcal{V}\to\CC$ by the aforementioned considerations. For each $N \geq 1$ and $(s,\omega) \in \mathcal{V}$ we may write
\[\sum_{n=1}^N\left(\mathfrak{D}_{s,\omega}\mathfrak{L}_{s,\omega}^{n-1}\mathbf{1}\right)(1)=\sum_{n=1}^{N}\lambda(s,\omega)^{n-1}\left(\mathfrak{D}_{s,\omega}\mathcal{P}_{s,\omega}\mathbf{1}\right)(1)+\sum_{n=1}^N\left(\mathfrak{D}_{s,\omega}\mathcal{N}_{s,\omega}^{n-1}\mathbf{1}\right)(1)\]
and by taking the limit as $N \to \infty$ when $(s,\omega) \in \mathcal{V} \cap \mathcal{W}$ it follows that
\[\sum_{(u,v) \in \Xi^{(1)}} \frac{\exp(\omega C(u,v))}{v^{2s}} = \frac{(\mathfrak{D}_{s,\omega}\mathcal{P}_{s,\omega}\mathbf{1})(1)}{1-\lambda(s,\omega)} + R(s,\omega)\]
as claimed.

We now turn to the summation over $\Xi^{(2)}$. We recall the following standard result from analytic number theory: if $s \in \CC$ and $(a_n)$ and $(b_n)$ are sequences of complex numbers such that $\sum_{n=1}^\infty n^{-s}a_n$ and $\sum_{n=1}^\infty n^{-s}b_n$ converge absolutely, then 
\[\left(\sum_{n=1}^\infty \frac{a_n}{n^s}\right)\left(\sum_{n=1}^\infty \frac{b_n}{n^s}\right)= \sum_{n=1}^\infty \frac{1}{n^s} \left(\sum_{d | n} a_{d}b_{\frac{n}{d}}\right)\]
and this series also converges absolutely. Now, for each odd integer $v \geq 1$ and for arbitrary $\omega \in \CC$ we have
\begin{align*}\sum_{d|v}\left(\sum_{u \colon (u,d) \in \Xi^{(1)}} \exp(\omega C(u,d))\right) &=\sum_{d|v}\left(\sum_{\substack{1 \leq u < d\\\mathrm{gcd}(u,d)=1\\u\text{ odd}}} \exp(\omega C(u,d))\right)\\
&=\sum_{d|v}\left(\sum_{\substack{1 \leq \tilde{u} < v\\\mathrm{gcd}(\tilde{u},v)=\frac{v}{d}\\\frac{\tilde{u}d}{v}\text{ odd}}} \exp(\omega C(\tilde{u},v))\right)\\
&=\sum_{\tilde{d}|v}\left(\sum_{\substack{1 \leq \tilde{u} < v\\\mathrm{gcd}(\tilde{u},v)=\tilde{d}\\\tilde{u}/\tilde{d}\text{ odd}}} \exp(\omega C(\tilde{u},v))\right)\\
&=\sum_{u \colon (u,v) \in \Xi^{(2)}} \exp\left(\omega C(u,v)\right),\end{align*}
and for $\Re(s)>\frac{1}{2}$,
\begin{equation}\label{partyfants}\sum_{n \text{ odd}}\frac{1}{n^{2s}} = \sum_{n=1}^\infty \frac{1}{n^{2s}}-\sum_{n=1}^\infty \frac{1}{(2n)^{2s}}=\left(1-4^{-s}\right)\zeta(2s)\end{equation}
and this sum is absolutely convergent. It follows that if $(s,\omega) \in \mathcal{W}$ then
\[\left(\sum_{v\text{ odd}} \frac{1}{v^{2s}}\right)\left(\sum_{v\text{ odd}} \frac{1}{v^{2s}} \sum_{u \colon (u,v) \in \Xi^{(1)}} e^{\omega C(u,v)}\right)=\sum_{v \text{ odd}} \frac{1}{v^{2s}} \sum_{u \colon (u,v)\in\Xi^{(2)}} e^{\omega C(u,v)}\]
or, expressed more compactly,
\[\left(1-4^{-s}\right)\zeta(2s)\left(\sum_{(u,v) \in \Xi^{(1)}}\frac{\exp(\omega C(u,v))}{v^{2s}}\right)=\left(\sum_{(u,v) \in \Xi^{(2)}}\frac{\exp(\omega C(u,v))}{v^{2s}}\right)\]
and this series is absolutely convergent as claimed.
\end{proof}

The following result comprises those parts of Proposition \ref{dseriesmain} which will be used in the present paper.
\begin{corollary}\label{usefullll}
Let $p \geq 0$ be an integer. Then for $i=1,2$ the Dirichlet series
\[\sum_{(u,v) \in \Xi^{(i)}} \frac{C(u,v)^p}{v^{2s}}\]
converges absolutely for all $s \in \CC$ such that $\Re(s)>1$, and admits an analytic continuation to an open set which includes the set $\{s \in \CC \colon \Re(s)=1\text{ and }s \neq 1\}$. For each $p \geq 0$ there exist an open set $U$ containing $1$ and meromorphic functions $R_p^{(1)},R_p^{(2)} \colon U \to \CC$ such that for all $s \in U$ with $\Re(s)>1$
\[\sum_{(u,v) \in \Xi^{(1)}} \frac{C(u,v)^p}{v^{2s}}=\frac{p!\lambda_\omega(1,0)^p\xi(1)}{2\lambda_s(1,0)^{p+1}(1-s)^{p+1}}+R_p^{(1)}(s)\]
and
\[\sum_{(u,v) \in \Xi^{(2)}} \frac{C(u,v)^p}{v^{2s}}=\frac{\pi^2p!\lambda_\omega(1,0)^p\xi(1)}{16\lambda_s(1,0)^{p+1}(1-s)^{p+1}}+R_p^{(2)}(s),\]
where each $R_p^{(i)}$ has a pole at $1$ of order not greater than $p$ and is otherwise holomorphic in $U$.
\end{corollary}
\begin{proof}
Let $\mathcal{W}$ be as in Proposition \ref{dseriesmain}. Clearly when $(s,0) \in \mathcal{W}$ and $i=1,2$
\[\frac{\partial^p}{\partial \omega^p}\left[\sum_{(u,v) \in \Xi^{(i)}} \frac{\exp(\omega C(u,v))}{v^{2s}}\right]\Bigg|_{\omega=0}=\sum_{(u,v) \in \Xi^{(i)}} \frac{C(u,v)^p}{v^{2s}}\]
and the absolute convergence of the second series follows from the absolute convergence of the first series by the direct estimate $C(u,v)^p \ll \exp\left(\omega C(u,v)\right) $ when $\omega$ is small and positive. By Proposition \ref{dseriesmain} the above identity also implies that the series being considered admits the desired analytic continuation.

Now let $R$ be as in Proposition \ref{dseriesmain}, let $\mathcal{V}$, $\lambda$ and $\mathcal{P}$ be as in Theorem \ref{MainBrent}, and let $U:=\{s \in \CC \colon (s,0)\in\mathcal{V}\}$ and $U^+:=\{s \in U \colon \Re(s)>1\}\subset \mathcal{V}\cap\mathcal{W}$. When $(s,\omega)\in\mathcal{V}\cap\mathcal{W}$,
\[\sum_{(u,v) \in \Xi^{(1)}} \frac{\exp(\omega C(u,v))}{v^{2s}} = \frac{(\mathfrak{D}_{s,\omega}\mathcal{P}_{s,\omega}\mathbf{1})(1)}{1-\lambda(s,\omega)} + R(s,\omega)\]
and
\[\sum_{(u,v) \in \Xi^{(2)}} \frac{\exp(\omega C(u,v))}{v^{2s}}  =\zeta(2s)\left(1-4^{-s}\right)\left( \frac{(\mathfrak{D}_{s,\omega}\mathcal{P}_{s,\omega}\mathbf{1})(1)}{1-\lambda(s,\omega)} + R(s,\omega)\right).\]
Differentiating the first equation $p$ times with respect to $\omega$ and setting $\omega =0$ we find that for all $s \in U^+$
\[\sum_{(u,v) \in \Xi^{(1)}} \frac{ C(u,v)^p}{v^{2s}} = \frac{p!\lambda_\omega(s,0)^p(\mathfrak{D}_{s,0}\mathcal{P}_{s,0}\mathbf{1})(1)}{(1-\lambda(s,0))^{p+1}} + r_p(s)\]
where $r_p\colon U \to \CC$ has a pole at $s=1$ of order not greater than $p$ and is otherwise holomorphic. Since $\lambda_s(1,0) \neq 0$ by Proposition \ref{trivium}, by replacing $U$ with a smaller open neighbourhood of $1$ if necessary we have $\lambda(s,0) \neq 1$ for every $s \in U \setminus \{1\}$. It follows that there exists a holomorphic function $\Lambda \colon U \to \CC$ such that
\[\frac{1-s}{1-\lambda(s,0)}=\frac{1}{\lambda_s(1,0)} + (s-1)\Lambda(s)\]
for all $s \in U\setminus\{1\}$. We may therefore write
\[\sum_{(u,v) \in \Xi} \frac{ C(u,v)^p}{v^{2s}} = \frac{p!\lambda_\omega(s,0)^p(\mathfrak{D}_{s,0}\mathcal{P}_{s,0}\mathbf{1})(1)}{\lambda_s(1,0)^{p+1}(1-s)^{p+1}} + \hat{r}_p(s)\]
for all $s \in U^+$, where $\hat{r}_p \colon U \to \CC$ is some function which also has a pole at $1$ of order not greater than $p$ and is otherwise holomorphic in $U$. Now, since
\[\left(\mathfrak{D}_{1,0}\mathcal{P}_{1,0}\mathbf{1}\right)(1)=\left(\mathfrak{D}_{1,0}\xi\right)(1)=\sum_{k=1}^\infty \frac{1}{(1+2^k)^2}\xi\left(\frac{1}{1+2^k}\right)=\frac{1}{2}\left(\mathfrak{L}_{1,0}\xi\right)(1)=\frac{\xi(1)}{2}\]
there exists a holomorphic function $g \colon U \to \CC$ such that for all $s \in U$
\[\lambda_\omega(s,0)^p(\mathfrak{D}_{s,0}\mathcal{P}_{s,0}\mathbf{1})(1)=\frac{\lambda_\omega(1,0)^p\xi(1)}{2}+(s-1)g(s)\]
and so it is clear that we may find $R^{(1)}_p \colon U \to \CC$ with the required properties such that
\[\sum_{(u,v) \in \Xi^{(1)}} \frac{C(u,v)^p}{v^{2s}}=\frac{p!\lambda_\omega(1,0)^p\xi(1)}{2\lambda_s(1,0)^{p+1}(1-s)^{p+1}}+R_p^{(1)}(s)\]
for all $s \in U^+$ as required. A similar argument using instead the identity
\[\zeta(2)\left(1-\frac{1}{4}\right)\left(\mathfrak{D}_{1,0}\mathcal{P}_{1,0}\mathbf{1}\right)(1)=\frac{\pi^2\xi(1)}{16}\]
establishes the analogous result for $\Xi^{(2)}$. 
\end{proof}
The following entirely number-theoretic lemma will also be useful in this and the following sections.
\begin{lemma}\label{numthy}
For all $s \in \CC$ such that $\Re(s)>1$,
\[\sum_{(u,v) \in \Xi^{(2)}} \frac{1}{v^{2s}}=\left(\frac{1}{2}-4^{-s}\right)\zeta(2s-1)-\left(\frac{1-4^{-s}}{2}\right)\zeta(2s)\]
and
\[\sum_{(u,v) \in \Xi^{(1)}} \frac{1}{v^{2s}}=\left(\frac{4^s-2}{4^s-1}\right)\left(\frac{\zeta(2s-1)}{2\zeta(2s)}\right) - \frac{1}{2}.\]
\end{lemma}
\begin{proof}
For each $s$ we may write
\[\sum_{(u,v) \in \Xi^{(2)}} \frac{1}{v^{2s}}=\sum_{v\text{ odd}} \left(\sum_{\substack{1\leq u < v\\u\text{ odd}}}\frac{1}{v^{2s}}\right) =\sum_{v\text{ odd}} \frac{v-1}{2v^{2s}}=\frac{1}{2}\left(\sum_{v \text{ odd}} \frac{1}{v^{2s-1}}  - \sum_{v \text{ odd}} \frac{1}{v^{2s}}\right),\]
and since as previously noted in \eqref{partyfants}
\[\sum_{v \text{ odd}} \frac{1}{v^{2s}} = \left(1-4^{-s}\right)\zeta(2s)\]
it follows that 
\[\sum_{(u,v) \in \Xi^{(2)}} \frac{1}{v^{2s}}=\left(\frac{1}{2}-4^{-s}\right)\zeta(2s-1)-\left(\frac{1-4^{-s}}{2}\right)\zeta(2s)\]
as claimed. Applying \eqref{moragissexy} with $\omega=0$ we find that
\[\sum_{(u,v) \in \Xi^{(1)}} \frac{1}{v^{2s}}=\left(1-4^{-s}\right)\zeta(2s)\left(\sum_{(u,v) \in \Xi^{(2)}} \frac{1}{v^{2s}}\right)\]
and the second result follows.
\end{proof}
Finally, we apply the results of this section to obtain a fourth formula for the derivative $\lambda_s(1,0)$ using an argument similar to one employed by B. Vall\'ee \cite[Prop. 6]{Vall98}.
\begin{corollary}\label{sheepdip}
The function $\lambda$ defined in Theorem \ref{MainBrent}(v) satisfies
\[\lambda_s(1,0)=-\frac{\pi^2\xi(1)}{2}.\]
\end{corollary}
\begin{proof}
By Lemma \ref{numthy}
\[\lim_{s \to 1}\left((s-1) \sum_{(u,v) \in \Xi^{(2)}} \frac{1}{v^{2s}}\right)=\lim_{s \to 1} \left((s-1)\left(\frac{1}{2}-4^{-s}\right)\zeta(2s-1)\right)=\frac{1}{8},\]
and by Corollary \ref{usefullll} with $p=0$
\[\lim_{s \to 1}\left((s-1) \sum_{(u,v) \in \Xi^{(2)}} \frac{1}{v^{2s}}\right)=\lim_{s \to 1} \left((s-1)\left(\frac{\pi^2\xi(1)}{16\lambda_s(1,0)(1-s)}\right)\right)=-\frac{\pi^2\xi(1)}{16\lambda_s(1,0)}.\]
Identifying the rightmost term of each line proves the corollary.
\end{proof}

\section{Proof of Theorem \ref{outcome}}\label{nurf}
The existence and properties of $\xi$ mentioned in the statement of Theorem \ref{outcome} were of course proved in Theorem \ref{MainBrent}, so in this section we have only to establish the general asymptotic formula for $C(u,v)$ and apply it to the specific cost measurements $E(u,v)$, $S(u,v)$ and $T(u,v)$. 
We require the following Tauberian theorem due to H. D\'elange (\cite[Th. III]{Del}, see also \cite[p.121-122]{Nark}) which has been found useful in related works on Euclidean algorithms \cite{DaVa04,LhVa98,Vall03} as well as in other investigations of asymptotic phenomena via transfer operators \cite{Poll97,PoSh98}.
\begin{theorem}[Delange]\label{Delang}
Let $\alpha \in \mathbb{R}$ and $k \in \mathbb{N}$, and let $(a_n)$ be a sequence of non-negative real numbers such that the Dirichlet series $\sum_{n=1}^\infty n^{-s}a_n$ converges absolutely for all $s \in \CC$ such that $\Re(s)>\alpha$. Suppose that $f$ and $g$ are holomorphic functions defined on an open subset of $\CC$ which includes the half-plane $\{s \in \CC \colon \Re(s)\geq \alpha\}$ such that $g(\alpha) \neq 0$ and such that when $\Re(s)>\alpha$,
\[\sum_{n=1}^\infty \frac{a_n}{n^s} = \frac{g(s)}{(s-\alpha)^k} + f(s).\]
Then
\[\lim_{N \to \infty} \frac{1}{N^a(\log N)^{k-1}} \sum_{n=1}^N a_n  = \frac{g(\alpha)}{\alpha \Gamma(k)}.\]
\end{theorem}
By Lemma \ref{numthy}, when $t \in\CC$ with $\Re(t)>2$ we have
\[\sum_{(u,v) \in \Xi^{(1)}} \frac{1}{v^t}=\left(\frac{2^t-2}{2^t-1}\right) \left(\frac{\zeta(t-1)}{2\zeta(t)}\right)-\frac{1}{2}\]
and
\[\sum_{(u,v) \in \Xi^{(2)}} \frac{1}{v^t}=\left(\frac{1}{2}-2^{-t}\right)\zeta(t-1)-\left(\frac{1}{2}-2^{-t-1}\right)\zeta(t)\]
and these series converge absolutely. It follows from Theorem \ref{Delang} together with standard properties of the zeta function that
\begin{equation}\label{amonster}\lim_{n \to \infty} \frac{\#\Xi^{(1)}_n}{n^2} = \lim_{n \to \infty} \frac{1}{n^2} \sum_{(u,v) \in \Xi_n^{(1)}} 1 = \frac{1}{\pi^2}\end{equation}
and
\begin{equation}\label{ethanolandswearword}\lim_{n \to \infty} \frac{\#\Xi^{(2)}_n}{n^2} = \lim_{n \to \infty} \frac{1}{n^2} \sum_{(u,v) \in \Xi_n^{(2)}} 1 = \frac{1}{8}.\end{equation}
These limits could of course also be obtained by more direct methods. Let us now define
\[\mu(c):=\frac{4}{\pi^2\xi(1)}\sum_{k=1}^\infty \frac{1}{2^k}\left(c(2,k)\int_0^{\frac{1}{1+2^k}}\xi(x)dx + c(1,k)\int_{\frac{1}{1+2^k}}^1\xi(x)dx\right)\]
which by Lemma \ref{deepship} and Corollary \ref{sheepdip} is precisely $-2\lambda_\omega(1,0)/\lambda_s(1,0)$. By Corollary \ref{usefullll}, for $t \in \CC$ with $\Re(t)>2$ the series 
\[\sum_{(u,v) \in \Xi^{(1)}} \frac{C(u,v)}{v^{t}},\qquad\sum_{(u,v) \in \Xi^{(2)}} \frac{C(u,v)}{v^{t}}\]
both converge absolutely to holomorphic functions which extend analytically to an open neighbourhood of the set $\{t \in \CC \colon \Re(t)=2\text{ and }\Im(t) \neq 0\}$. When $\Re(t)>2$ and $t$ is close to $2$ we have
\[\sum_{(u,v) \in \Xi^{(1)}} \frac{C(u,v)}{v^{t}}=\frac{\lambda_\omega(1,0)\xi(1)}{2\lambda_s(1,0)^2\left(1-\frac{t}{2}\right)^2} + R^{(1)}_1\left(\frac{t}{2}\right)=\frac{2\mu(c)}{\pi^2\left(t-2\right)^2} + R^{(1)}_1\left(\frac{t}{2}\right)\]
and
\[\sum_{(u,v) \in \Xi^{(2)}} \frac{C(u,v)}{v^{t}}=\frac{\mu(c)}{4\left(t-2\right)^2} + R^{(2)}_1\left(\frac{t}{2}\right)\]
where $R^{(1)}_1$ and $R^{(2)}_1$ have the properties described in Corollary \ref{usefullll} and we have again used the identity $\lambda_s(1,0)=-\frac{1}{2}\pi^2\xi(1)$ from Corollary \ref{sheepdip}. Applying Theorem \ref{Delang} again we obtain
\[\lim_{n \to \infty} \frac{1}{n^2\log n} \sum_{(u,v) \in \Xi_n^{(1)}} C(u,v) = \frac{\mu(c)}{\pi^2}\]
and
\[\lim_{n \to \infty} \frac{1}{n^2\log n} \sum_{(u,v) \in \Xi_n^{(2)}} C(u,v) = \frac{\mu(c)}{8}\]
and thus by \eqref{amonster} and \eqref{ethanolandswearword}
\begin{equation}\label{fruntum}\lim_{n \to \infty} \frac{1}{\#\Xi_n^{(i)}\log n} \sum_{(u,v) \in \Xi_n^{(i)}} C(u,v) =\mu(c)\end{equation}
for $i=1,2$ as required.

To treat $\Xi^{(3)}$ and $\Xi^{(4)}$ we require some additional estimates. By considering the series
\[\sum_{(u,v) \in \Xi^{(1)}} \frac{C(u,v)^2}{v^{t}},\qquad\sum_{(u,v) \in \Xi^{(2)}} \frac{C(u,v)^2}{v^{t}}\]
using Corollary \ref{usefullll} and Theorem \ref{Delang} as above we obtain
\[\lim_{n \to \infty} \frac{1}{\#\Xi_n^{(i)}(\log n)^2} \sum_{(u,v) \in \Xi_n^{(i)}} C(u,v)^2 =\mu(c)^2\]
for $i=1,2$. Combining this result with \eqref{fruntum} we deduce that
\begin{equation}\label{cathaddock}\lim_{n \to \infty} \frac{1}{\#\Xi_n^{(i)}(\log n)^2} \sum_{(u,v) \in \Xi_n^{(i)}} \left(C(u,v)-\mu(c)\log n\right)^2=0\end{equation}
for $i=1,2$ by expanding the quadratic, taking the limit for each term individually and verifying that the results cancel.

Let us now consider the sum over $\Xi^{(3)}$. If $(u,v)$ is a pair of natural numbers less than or equal to $n$, then $(u,v) \in \Xi^{(3)}_n$ if and only if there exist $(a,b) \in \Xi^{(1)}$ and $k \geq 0$ such that either $(u,v)=(2^ka,b)$ or $(u,v)=(a,2^kb)$. In particular either $k$ is zero, or we are in the former case and $a < 2^ka \leq n$, or we are in the latter case and $b < 2^kb \leq n$. Thus
\[\sum_{(u,v) \in \Xi^{(3)}_n} C(u,v) = \sum_{(a,b) \in \Xi^{(1)}_n} \left(C(a,b) + \sum_{k\colon a<2^ka \leq n} C(a,b)  + \sum_{k\colon b<2^kb \leq n} C(a,b)\right) \]
\[=\sum_{(a,b)\in\Xi^{(1)}_n} \left(1+\left\lfloor \log_2\frac{n}{a}\right\rfloor+\left\lfloor \log_2\frac{n}{b}\right\rfloor \right)C(a,b)\]
and indeed
\begin{eqnarray*}\lefteqn{\sum_{(u,v) \in \Xi^{(3)}_n} \left|C(u,v)-\mu(c)\log n\right|}& &\\&=& \sum_{(a,b)\in\Xi^{(1)}_n} \left(1+\left\lfloor \log_2\frac{n}{a}\right\rfloor+\left\lfloor \log_2\frac{n}{b}\right\rfloor \right)\left|C(a,b)-\mu(c)\log n\right|.\end{eqnarray*}
By the Cauchy-Schwarz inequality the right-hand side is bounded by the product
\begin{equation}\label{bundy} \left(\sum_{(a,b)\in\Xi^{(1)}_n} \left(1+\left\lfloor \log_2\frac{n}{a}\right\rfloor+\left\lfloor \log_2\frac{n}{b}\right\rfloor \right)^2\right)^{\frac{1}{2}} \left(\sum_{(a,b)\in\Xi^{(1)}_n}\left(C(a,b)-\mu(c)\log n\right)^2\right)^{\frac{1}{2}}.\end{equation}
Since obviously
\begin{align*}\sum_{(a,b)\in\Xi^{(1)}_n} \left(1+\left\lfloor \log_2\frac{n}{a}\right\rfloor+\left\lfloor \log_2\frac{n}{b}\right\rfloor \right)^2 &\leq \sum_{(a,b) \in \Xi^{(2)}_n} \left(1+\left\lfloor \log_2\frac{n}{a}\right\rfloor\right)^2\left(1+\left\lfloor \log_2\frac{n}{b}\right\rfloor\right)^2\\
&\leq \left(\sum_{\ell=1}^n \left(1+\left\lfloor \log_2\frac{n}{\ell}\right\rfloor\right)^2\right)^2\end{align*}
it follows that
\begin{align*}
\limsup_{n \to \infty} \frac{1}{n^2}\!\sum_{(u,v)\in\Xi^{(1)}_n}\! \left(1+\left\lfloor \log_2\frac{n}{u}\right\rfloor+\left\lfloor \log_2\frac{n}{v}\right\rfloor \right)^2 &\leq \lim_{n \to \infty} \left(\frac{1}{n}\sum_{u=1}^n \left(1+\left\lfloor \log_2\frac{n}{u}\right\rfloor\right)^2\right)^2
\\& = \left(\int_0^1 \left(1+\left\lfloor \log_2\frac{1}{x}\right\rfloor\right)dx\right)^2=4\end{align*}
since the sum in the latter limit is just a Riemann sum of the subsequent integral. We deduce that
\[\limsup_{n \to \infty} \left(\frac{1}{\#\Xi^{(3)}_n}\sum_{(a,b)\in\Xi^{(1)}_n} \left(1+\left\lfloor \log_2\frac{n}{a}\right\rfloor+\left\lfloor \log_2\frac{n}{b}\right\rfloor \right)^2\right)^{\frac{1}{2}} <\infty\]
and since by \eqref{cathaddock} with $i=1$
\[\lim_{n \to \infty}\left( \frac{1}{\#\Xi_n^{(1)}(\log n)^2} \sum_{(u,v) \in \Xi_n^{(1)}} \left(C(u,v)-\mu(c)\log n\right)^2\right)^{\frac{1}{2}}=0\]
we deduce from the bound \eqref{bundy} that
\[\lim_{n \to \infty} \frac{1}{\#\Xi^{(3)}_n \log n}\sum_{(u,v) \in \Xi^{(3)}_n} \left|C(u,v)-\mu(c)\log n\right|=0\]
which clearly implies the desired result. To treat $\Xi^{(4)}$ a similar counting argument yields
\begin{eqnarray*}\lefteqn{\sum_{(u,v) \in \Xi^{(4)}_n} \left|C(u,v)-\mu(c)\log n\right|}\\& =& \sum_{(a,b)\in\Xi^{(2)}_n} \left(1+\left\lfloor \log_2\frac{n}{a}\right\rfloor\right)\left(1+\left\lfloor \log_2\frac{n}{b}\right\rfloor\right)\left|C(a,b)-\mu(c)\log n\right|\end{eqnarray*}
and by the same Cauchy-Schwarz estimate
\[\lim_{n \to \infty} \frac{1}{\#\Xi^{(4)}_n \log n}\sum_{(u,v) \in \Xi^{(4)}_n} \left|C(u,v)-\mu(c)\log n\right|=0.\]
This completes the proof of the general part of Theorem \ref{outcome}.

It remains to establish the specific formul{\ae} for $E(u,v)$, $S(u,v)$ and $T(u,v)$ asserted in the statement of the theorem. To evaluate the average of the number of subtractions $S(u,v)$ we note that each of the branches $z \mapsto \frac{z}{z+2^k}$ and $z \mapsto \frac{1}{1+2^kz}$ corresponds to exactly one subtraction, and so defining $c_S(i,k):=1$ for all $k \geq 1$ for $i=1,2$ yields $C(u,v) \equiv S(u,v)$. By the definition of $\mu\left(c_S\right)$ together with Corollary \ref{sheepdip} and the identity $\int_0^1\xi(x)dx=1$ we have
\[\mu\left(c_S\right)=\frac{4}{\pi^2\xi(1)} \left(\sum_{k=1}^\infty \frac{1}{2^k}\right)=\frac{4}{\pi^2\xi(1)}=-\frac{2}{\lambda_s(1,0)}\]
and so by Proposition \ref{trivium} we also have
\[\mu\left(c_S\right)=\frac{1}{\sum_{k=1}^\infty \frac{1}{2^k} \int_0^1\log\left(\frac{2^k(1+x)}{1+(2^k-1)x}\right)\xi(x)dx}=
\frac{2}{\log 4 - \int_0^1\log(1-x)\xi(x)dx}.\]
Applying the general part of the theorem to the cost function $c_S$ yields \eqref{brentform}, \eqref{knuthform} and \eqref{valleform}. 

As was discussed in \S\ref{arsity}, to evaluate $T(u,v)$ we consider the cost function given by $c_T(i,k):=k$ for all $k \geq 1$ for $i=1,2$ with respect to which we have $C(u,v)\equiv T(u,v)$. The definition of $\mu\left(c_T\right)$ yields 
\[\mu\left(c_T\right) = \frac{4}{\pi^2\xi(1)}\left(\sum_{k=1}^\infty \frac{k}{2^k}\right)=\frac{8}{\pi^2\xi(1)}\]
as required to prove the claimed formula for $T(u,v)$. Finally, as discussed in \S\ref{arsity} defining $c_E(1,k):=1$ and $c_E(2,k):=0$ for all $k \geq 1$ yields $C(u,v)\equiv E(u,v)$, and we may evaluate
\[\mu\left(c_E\right) = \frac{4}{\pi^2\xi(1)}\left(\sum_{k=1}^\infty \frac{1}{2^k}\int_{\frac{1}{1+2^k}}^1\xi(x)dx\right)\]
which yields \eqref{exch1}. To prove \eqref{exch2} we note that
\begin{align*}
\int_0^{\frac{1}{2}}\xi(x)dx=&\int_0^{\frac{1}{2}}\left(\mathfrak{L}_{1,0}\xi\right)(x)dx\\
=&\int_0^{\frac{1}{2}}\left(\sum_{k=1}^\infty \frac{1}{\left(1+2^kx\right)^2}\xi\left(\frac{1}{1+2^kx}\right)dx + \frac{1}{\left(x+2^k\right)^2}\xi\left(\frac{x}{x+2^k}\right)\right)dx\\
=&\sum_{k=1}^\infty\frac{1}{2^k}\int_{\frac{1}{1+2^{k-1}}}^1 \xi(x)dx + \sum_{k=1}^\infty \frac{1}{2^k}\int_0^{\frac{1}{1+2^{k+1}}}\xi(x)dx\\
=&\frac{1}{2}\sum_{k=1}^\infty\frac{1}{2^k}\int_{\frac{1}{1+2^k}}^1 \xi(x)dx + 2\sum_{k=1}^\infty \frac{1}{2^k}\int_0^{\frac{1}{1+2^k}}\xi(x)dx\\
& + \frac{1}{2}\int_{\frac{1}{2}}^1\xi(x)dx - \int_0^{\frac{1}{3}}\xi(x)dx\\
=&\frac{1}{2}+\frac{3}{2}\sum_{k=1}^\infty \frac{1}{2^k}\int_0^{\frac{1}{1+2^k}}\xi(x)dx +\frac{1}{2}\int_{\frac{1}{2}}^1\xi(x)dx-\int_0^{\frac{1}{3}}\xi(x)dx
\end{align*}
so that
\[\sum_{k=1}^\infty \frac{1}{2^k}\int_0^{\frac{1}{1+2^k}} \xi(x)dx=\frac{2}{3}\int_0^{\frac{1}{2}}\xi(x)dx+\frac{2}{3}\int_0^{\frac{1}{3}}\xi(x)dx - \frac{1}{3}-\frac{1}{3}\int_{\frac{1}{2}}^1\xi(x)dx\]
and therefore
\[\sum_{k=1}^\infty\frac{1}{2^k}\int_{\frac{1}{1+2^k}}^1\xi(x)dx = 1-\sum_{k=1}^\infty \frac{1}{2^k}\int_0^{\frac{1}{1+2^k}} \xi(x)dx=\int_{\frac{1}{2}}^1\xi(x)dx + \frac{2}{3}\int_{\frac{1}{3}}^1\xi(x)dx\]
as required to derive \eqref{exch2} from \eqref{exch1}. The proof of Theorem \ref{outcome} is complete.

%One typo that I noticed is on page
%24, line 4 of statement of Proposition 6.9, "depending on $g$" should
%be omitted.

\section{Acknowledgments}
The author would like to thank O. Butterley and R. P. Brent for helpful suggestions and correspondence. The author is also grateful to C. Beenakker for suggesting the reference \cite{Gran99} and to the MathOverflow community and its supporters for making this interaction possible.

This research was supported by EPSRC grant EP/L026953/1.

\bibliographystyle{amsplain}
\bibliography{binary}

\def\cprime{$'$}
\providecommand{\bysame}{\leavevmode\hbox to3em{\hrulefill}\thinspace}
\providecommand{\MR}{\relax\ifhmode\unskip\space\fi MR }
\providecommand{\MRhref}[2]{%
  \href{http://www.ams.org/mathscinet-getitem?mr=#1}{#2}
}
\providecommand{\href}[2]{#2}
\begin{thebibliography}{10}

\bibitem{Bala00}
Viviane Baladi, \emph{Positive transfer operators and decay of correlations},
  Advanced Series in Nonlinear Dynamics, vol.~16, World Scientific Publishing
  Co. Inc., River Edge, NJ, 2000.

\bibitem{BaVa05}
Viviane Baladi and Brigitte Vall{\'e}e, \emph{Euclidean algorithms are
  {G}aussian}, J. Number Theory \textbf{110} (2005), no.~2, 331--386.

\bibitem{Bren76}
Richard~P. Brent, \emph{Analysis of the binary {E}uclidean algorithm},
  Algorithms and complexity ({P}roc. {S}ympos., {C}arnegie-{M}ellon {U}niv.,
  {P}ittsburgh, {P}a., 1976), Academic Press, New York, 1976, pp.~321--355.

\bibitem{Bren97}
\bysame, \emph{Simplification of an integral}, unpublished manuscript, 1997.

\bibitem{Bren98}
\bysame, \emph{Further analysis of the binary {E}uclidean algorithm},
  Programming {R}esearch {G}roup technical report TR-7-99, {O}xford
  {U}niversity, November 1999.

\bibitem{BVcorr}
Eda Cesaratto, \emph{A note on ``{E}uclidean algorithms are {G}aussian'' by
  {V}. {B}aladi and {B}. {V}all\'ee}, J. Number Theory \textbf{129} (2009),
  no.~10, 2267--2273.

\bibitem{DaVa04}
Beno{\^{\i}}t Daireaux and Brigitte Vall{\'e}e, \emph{Dynamical analysis of the
  parametrized {L}ehmer-{E}uclid algorithm}, Combin. Probab. Comput.
  \textbf{13} (2004), no.~4-5, 499--536.

\bibitem{Del}
Hubert Delange, \emph{G\'en\'eralisation du th\'eor\`eme de {I}kehara}, Ann.
  Sci. Ecole Norm. Sup. (3) \textbf{71} (1954), 213--242.

\bibitem{Dixon}
John~D. Dixon, \emph{A simple estimate for the number of steps in the
  {E}uclidean algorithm.}, Amer. Math. Monthly \textbf{78} (1971), 374--376.

\bibitem{Dure}
Peter~L. Duren, \emph{Theory of {$H^{p}$} spaces}, Pure and Applied
  Mathematics, Vol. 38, Academic Press, New York, 1970.

\bibitem{EdEv}
David~E. Edmunds and William~D. Evans, \emph{Spectral theory and differential
  operators}, Oxford Mathematical Monographs, The Clarendon Press Oxford
  University Press, New York, 1987, Oxford Science Publications.

\bibitem{Faiv92}
Christian Faivre, \emph{Distribution of {L}\'evy constants for quadratic
  numbers}, Acta Arith. \textbf{61} (1992), no.~1, 13--34.

\bibitem{Finck}
Pierre-Joseph-{\'{E}}tienne Finck, \emph{Lettre}, Nouvelles annales de
  math\'ematiques \textbf{1} (1842), 353--355.

\bibitem{Gabr28}
Robert~M. Gabriel, \emph{Some results concerning the integrals of moduli of
  regular functions along curves of certain types}, Proc. London Math. Soc. (2)
  \textbf{28} (1928), 121--127.

\bibitem{Gab}
\bysame, \emph{An {I}nequality {C}oncerning the {I}ntegrals of {P}ositive
  {S}ubharmonic {F}unctions {A}long {C}ertain {C}ircles}, J. London Math. Soc.
  (1) \textbf{5} (1930), no.~2, 129--131.

\bibitem{Gran99}
Ana Granados, \emph{On a problem raised by {G}abriel and {B}eurling}, Michigan
  Math. J. \textbf{46} (1999), no.~3, 461--487.

\bibitem{Heil69}
Hans Heilbronn, \emph{On the average length of a class of finite continued
  fractions}, Number {T}heory and {A}nalysis ({P}apers in {H}onor of {E}dmund
  {L}andau), Plenum, New York, 1969, pp.~87--96.

\bibitem{Henn93}
Hubert Hennion, \emph{Sur un th\'eor\`eme spectral et son application aux
  noyaux lipchitziens}, Proc. Amer. Math. Soc. \textbf{118} (1993), no.~2,
  627--634.

\bibitem{Hens94}
Doug Hensley, \emph{The number of steps in the {E}uclidean algorithm}, J.
  Number Theory \textbf{49} (1994), no.~2, 142--182.

\bibitem{Kato}
Tosio Kato, \emph{Perturbation theory for linear operators}, Classics in
  Mathematics, Springer-Verlag, Berlin, 1995, Reprint of the 1980 edition.

\bibitem{Knuth81}
Donald~E. Knuth, \emph{The art of computer programming, {V}ol. 2: seminumerical
  algorithms}, 2nd ed., Addison-Wesley Publishing Co., Reading, Mass., 1981.

\bibitem{Knuth97}
\bysame, \emph{The art of computer programming, {V}ol. 2: seminumerical
  algorithms}, 3rd ed., Addison-Wesley Publishing Co., Reading, Mass., 1997.

\bibitem{Kras64}
Mark~A. Krasnosel{\cprime}ski{\u\i}, \emph{Positive solutions of operator
  equations}, Translated from the Russian by Richard E. Flaherty; edited by Leo
  F. Boron, P. Noordhoff Ltd. Groningen, 1964.

\bibitem{LeSc}
Arnold Lebow and Martin Schechter, \emph{Semigroups of operators and measures
  of noncompactness}, J. Functional Analysis \textbf{7} (1971), 1--26.

\bibitem{LhVa98}
Lo{\"{\i}}ck Lhote and Brigitte Vall{\'e}e, \emph{Gaussian laws for the main
  parameters of the {E}uclid algorithms}, Algorithmica \textbf{50} (2008),
  no.~4, 497--554.

\bibitem{Live95}
Carlangelo Liverani, \emph{Decay of correlations}, Ann. of Math. (2)
  \textbf{142} (1995), no.~2, 239--301.

\bibitem{Rose}
Rub{\'e}n~A. Mart{\'{\i}}nez-Avenda{\~n}o and Peter Rosenthal, \emph{An
  introduction to operators on the {H}ardy-{H}ilbert space}, Graduate Texts in
  Mathematics, vol. 237, Springer, New York, 2007.

\bibitem{Maye91}
Dieter~H. Mayer, \emph{Continued fractions and related transformations},
  Ergodic theory, symbolic dynamics, and hyperbolic spaces ({T}rieste, 1989),
  Oxford Sci. Publ., Oxford Univ. Press, New York, 1991, pp.~175--222.

\bibitem{Maze}
G{\'e}rard Maze, \emph{Existence of a limiting distribution for the binary
  {GCD} algorithm}, J. Discrete Algorithms \textbf{5} (2007), no.~1, 176--186.

\bibitem{Nark}
W{\l}adys{\l}aw Narkiewicz, \emph{Number theory}, World Scientific Publishing
  Co., Singapore, 1983, Translated from the Polish by S. Kanemitsu.

\bibitem{Nuss70}
Roger~D. Nussbaum, \emph{The radius of the essential spectrum}, Duke Math. J.
  \textbf{37} (1970), 473--478.

\bibitem{Nuss81}
\bysame, \emph{Eigenvectors of nonlinear positive operators and the linear
  {K}re\u\i n-{R}utman theorem}, Fixed point theory ({S}herbrooke, {Q}ue.,
  1980), Lecture Notes in Math., vol. 886, Springer, Berlin, 1981,
  pp.~309--330.

\bibitem{PP90}
William Parry and Mark Pollicott, \emph{Zeta functions and the periodic orbit
  structure of hyperbolic dynamics}, Ast\'erisque (1990), no.~187-188, 268.

\bibitem{Poll97}
Mark Pollicott, \emph{Asymptotic auto-correlation for closed geodesics}, Comm.
  Math. Phys. \textbf{187} (1997), no.~2, 341--355.

\bibitem{PoSh98}
Mark Pollicott and Richard Sharp, \emph{Comparison theorems and orbit counting
  in hyperbolic geometry}, Trans. Amer. Math. Soc. \textbf{350} (1998), no.~2,
  473--499.

\bibitem{Porter}
John~W. Porter, \emph{On a theorem of {H}eilbronn}, Mathematika \textbf{22}
  (1975), no.~1, 20--28.

\bibitem{Ruel78}
David Ruelle, \emph{Thermodynamic formalism}, Encyclopedia of Mathematics and
  its Applications, vol.~5, Addison-Wesley Publishing Co., Reading, Mass.,
  1978.

\bibitem{Shal2}
Peter Schreiber, \emph{A supplement to {J}. {S}hallit's paper: ``{O}rigins of
  the analysis of the {E}uclidean algorithm''}, Historia Math. \textbf{22}
  (1995), no.~4, 422--424.

\bibitem{Shant}
Jeffrey Shallit, \emph{Origins of the analysis of the {E}uclidean algorithm},
  Historia Math. \textbf{21} (1994), no.~4, 401--419. \MR{1308143 (95h:01015)}

\bibitem{Shap}
Joel~H. Shapiro, \emph{Composition operators and classical function theory},
  Universitext: Tracts in Mathematics, Springer-Verlag, New York, 1993.

\bibitem{Stein}
Josef Stein, \emph{Computational problems associated with {R}acah algebra}, J.
  Comput. Phys. \textbf{1} (1967), 397--405.

\bibitem{Vall98}
Brigitte Vall{\'e}e, \emph{Dynamics of the binary {E}uclidean algorithm:
  functional analysis and operators}, Algorithmica \textbf{22} (1998), no.~4,
  660--685.

\bibitem{Vall03}
\bysame, \emph{Dynamical analysis of a class of {E}uclidean algorithms},
  Theoret. Comput. Sci. \textbf{297} (2003), no.~1-3, 447--486.

\bibitem{Vall06}
\bysame, \emph{Euclidean dynamics}, Discrete Contin. Dyn. Syst. \textbf{15}
  (2006), no.~1, 281--352.

\end{thebibliography}
\end{document}